\documentclass[1 leqno,11pt,psfig]{amsart}
\usepackage{amssymb, amsmath,amsmath,latexsym,amssymb,amsfonts,amsbsy, amsthm,mathtools,graphicx,CJKutf8,CJKnumb,CJKulem,color}
\usepackage{graphicx,epsfig}
\usepackage{graphicx}
\usepackage{amssymb}
\usepackage{graphicx}
\usepackage{graphicx,epsfig}
\usepackage{amsmath}
\usepackage{mathrsfs}
\usepackage{amssymb}
\usepackage{makecell}

\usepackage{diagbox}
\usepackage{array} 
\usepackage{amssymb,mathrsfs,amsfonts,amsmath,amsbsy,tikz}\usepackage{fontenc}\usepackage{textcomp}

\numberwithin{equation}{section}

\allowdisplaybreaks

\let\pa=\partial

\def\bbT{\mathbb{T}}

\newcommand{\beq}{\begin{equation}}
\newcommand{\eeq}{\end{equation}}
\newcommand{\ben}{\begin{eqnarray}}
\newcommand{\een}{\end{eqnarray}}
\newcommand{\beno}{\begin{eqnarray*}}
\newcommand{\eeno}{\end{eqnarray*}}

\newtheorem{Theorem}{Theorem}[section]

\newtheorem{Proposition}[Theorem]{Proposition}
\newtheorem{Lemma}[Theorem]{Lemma}
\newtheorem{Corollary}[Theorem]{Corollary}
\newtheorem{Remark}[Theorem]{Remark}
\newtheorem{Example}[Theorem]{Example}

\begin{document}
\begin{CJK*}{UTF8}{gkai}
\title[Classification of atmospheric traveling waves at cloud level]{Classification  of atmospheric traveling waves \\ at cloud level}

\author{Adrian Constantin}
\address{Faculty of Mathematics, University of Vienna, Oskar-Morgenstern-Platz 1, 1090 Vienna, Austria}
\email{adrian.constantin@univie.ac.at}

\author{Zhiwu Lin}
\address{School of Mathematical Sciences, Fudan University,  200433, Shanghai, P. R. China}
\email{zwlin@fudan.edu.cn}

\author{Hao Zhu}
\address{School of Mathematics, Nanjing University,  210093, Nanjing, Jiangsu, P. R. China \&
 Faculty of Mathematics, University of Vienna, Oskar-Morgenstern-Platz 1, 1090 Vienna, Austria}
\email{haozhu@nju.edu.cn\;\;\& hao.zhu@univie.ac.at}

\date{\today}

\maketitle

\date{\today}

\maketitle

\begin{abstract}
We classify within the  quasi-geostrophic framework all types of traveling waves  in zonal bands of the planetary atmosphere at cloud level according to their wave speeds. This classification pertains to waves of all amplitudes, going beyond the  small-amplitude  perturbative  regime. It provides
 a structurally robust criterion
for determining which traveling-wave profiles are dynamically possible and we show that
  each wave classification type  was observed on Jupiter or Saturn.
  Building on this classification, we also investigate the related rigidity issue for large-amplitude traveling waves and waves propagating near shear flows.
Our study offers a unified quantitative characterization of the intrinsic constraints for traveling waves in the quasi-geostrophic regime  of planetary atmospheric flow.
\end{abstract}

\section{Introduction}

The two-dimensional incompressible $\beta$-plane equation is widely used to study the dynamics of quasi-geostrophic flows that are confined to zonal bands and for which
the vertical motion is negligible. Such flows are predominant in the dynamics of the atmosphere of the gas giants (Jupiter and Saturn) and of the ice giants (Neptune and Uranus) at cloud level.
By approximating the Coriolis parameter as varying linearly with latitude, having the form $f_0+\beta y$ with
\begin{align}\label{def-f0-beta}
f_0=\frac{2\varOmega'R'}{U'}\,\sin\theta_0\quad\text{and}\quad \beta=\frac{2\varOmega'R'}{U'}\,\cos\theta_0\,,
\end{align}
the $\beta$-plane equation captures the essential behavior induced by sphericity combined with rotation;
here $\varOmega'$ is the constant rate of rotation of the spherical planet of radius $R'$ about its polar axis, $U'$ is the typical magnitude of
the horizontal velocity and $\theta_0 \in \big( -\frac{\pi}{2}\,,\,\frac{\pi}{2}\big)$ is the reference latitude of the flow region that does not comprise the poles
(see the discussion in \cite{csz2024a,Johnson2023} and Table \ref{v-par-JS}).
Rather than the more intricate Euler equation in rotating spherical coordinates, the $\beta$-plane equation is formulated in a Cartesian coordinate system, which simplifies the analysis but retains the effects of sphericity at leading order
for the bands of moderate meridional width typical for Jupiter and Saturn \cite{v}. The atmospheric flow at cloud level on Neptune and Uranus has also a zonal structure but the bands are much wider than those on
Jupiter and Saturn: both ice giants have only two major prograde jets, one per hemisphere, with retrograde equatorial jets \cite{sw}. Moreover, the latitudinal averaged zonal flow is stable on Neptune and Uranus
(see \cite{cg}), a feature which impedes the formation of long-lasting wave perturbations,  in marked contrast to the gas giants.

\begin{table}[h!]
\centering

\begin{tabular}{| c | c|c|c|}
\hline
  Planet &   $R'$ & $\Omega'$ &  $ U'$ \\
\hline
Jupiter&$69911$\;km&$1.76\times 10^{-4}$\;rad/s&$150$\;m/s\\
Saturn&$58232$\; km&$1.62\times 10^{-4}$\;rad/s&$150$\;m/s\\
\hline
\end{tabular}
\vspace{0.2cm}
\caption{Values of the relevant parameters for Jupiter and Saturn.}
\label{v-par-JS}
\end{table}

The banded appearance of Jupiter and Saturn is primarily caused by their rapid rotation. Both giant planets exhibit east-west jet streams along the boundaries of these zonal bands,
with speeds ranging up to 150 m$\,$s$^{-1}$ on Jupiter and up to 400 m$\,$s$^{-1}$ on Saturn. These jets dominate the dynamics, containing most of  the kinetic energy at the cloud
tops on Saturn and about 90\% on Jupiter \cite{read}, and lead to the formation of zonally propagating long-lived coherent shear flows.
A plethora of short-lived cloud structures arise as perturbations of these shear flows but some persistent dynamical features can also be observed. Typically these are traveling waves propagating zonally but occasionally vortices, such as Jupiter's Great Red Spot (GRS) and Saturn's polar vortices, also occur. While vortices that persist on large time scales require special
circumstances (for example, the GRS has a quite peculiar dynamics unlike that of terrestrial hurricanes \cite{cj1} and Saturn's polar vortices
are hot spots \cite{sa}), a wide range of traveling waves have been identified by telescope and spacecraft (see \cite{or} for the identification of a variety of waves detected on Jupiter by the Juno spacecraft).
Such waves are unstable perturbations of an underlying shear flow that grow and acquire specific dynamic characteristics which set
them apart from the original state of the flow (e.g. the circumpolar waves that form Saturn's hexagon -- see \cite{cj2}).

Coherent waves can only propagate azimuthally since the strong jets ensure that waves propagating in a direction tilted to the East-West direction are short-lived.
The velocity field of a traveling wave propagating zonally in a band $d_- \le y \le d_+$ at the speed $c$ has the form $(u(x-ct,y),v(x-ct,y))$ with $v=0$ on $y=d_\pm$. Genuine waves correspond to
flows which depend on the variable $(x-ct)$, while an incompressible shear flow is described by a velocity of the form $(U(y),0)$. Within the framework of linear theory (see \cite{lyz, mas}), shear flow perturbations corresponding to
the stream function $\psi(x,y,t)=\phi(y){\rm e}^{{\rm i}\alpha(x-ct)}$, where $\alpha>0$ is the wave-number in the $x$-direction and $c$ is the complex wave speed, are governed by the Rayleigh-Kuo equation
\begin{equation}\label{rk}
(U-c) (\alpha^2-\phi'') - (\beta-U'')\phi=0
\end{equation}
with the boundary conditions $\phi(d_-)=\phi(d_+)=0$. Equation \eqref{rk} has a singularity at a critical layer (a point $y_c$ with $U(y_c)=c$), since the coefficient of the highest-order derivative vanishes. In this case a perturbation analysis
shows the appearance of vortices near the critical layer. In the absence of critical layers, instability can only arise if the imaginary part of the speed $c$, giving the growth rate of the perturbation, does not vanish. This leads to
the condition that $\beta-U''$ vanishes at some level as being necessary for linear instability (see \cite{dr}). In light of these
classical results, we expect that genuine traveling waves can arise as coherent perturbations of a shear flow only under special circumstances.
This perspective is supported by recent rigidity results showing that, under suitable
conditions, certain shear flows do not admit any nontrivial traveling-wave perturbations \cite{Hamel2017,csz2024a}.

From the observational viewpoint, the four giant planets exhibit markedly different amounts of large-scale coherent wave activity. Several long-lived wave patterns have been observed on Jupiter and
Saturn (some will be described below), but Uranus is the most featureless of the giant planets
and Neptune's most notable visible flow patterns are its dark spots. The very wide zonal bands
on Neptune do not confine the dark spots tightly: for example, the largest, Neptune's Great Dark
Spot, first observed by Voyager 2 in 1989, drifted towards the equator and disappeared by 1994 (see \cite{Ingersoll2013}). On
the other hand, on Jupiter and Saturn one finds a profusion of coherent waves: Saturn's stationary
hexagon, meandering ribbons, Jupiter's chevrons and band-limited wave trains, as well as more intricate
wave patterns surrounding the GRS.

The main aim of this paper is to provide a {\it complete and rigorous
classification} of genuine traveling-wave solutions
arising as coherent flow patterns with possibly large meridional velocity components and thus not necessarily representing small perturbations of a shear flow in a zonal band.
 More precisely, we consider
 the $\beta$-plane equation
\begin{equation}\label{bpe}
\begin{cases}
&u_t + uu_x+vu_y - (f_0+\beta y)\, v = - {\frak P}_x\,,\\
&v_t + uv_x+vv_y + (f_0+\beta y) \, u = - {\frak P}_y\,,\\
&u_x + v_y=0
\end{cases}
\end{equation}
confined to a zonal band ${\frak D}_L=\{(x,y):\ x \in {\mathbb T}_L={\mathbb R}/_{L{\mathbb Z}}\,,\ y \in [d_-,d_+]\}$,
where $(u,v)$ is the horizontal velocity field, ${\frak P}$ is the pressure, and $f_0,\beta$ are defined in \eqref{def-f0-beta}.
 The associated  non-permeable boundary condition is
\begin{align}\label{bce}
v=0\quad \text{on}\quad  y = d_\pm\,.
\end{align}
We will prove the following classification theorem.

\begin{Theorem}\label{classification-of-wave-speed-for-a-genuinely-travelling-wave-beta-plane-intro}
Given $L>0$ and $d_+>d_-$, let
$$(u(x-ct,y),v(x-ct,y))\in C^2({\frak D}_L)$$
 be an $L$-periodic genuine traveling-wave solution to the $\beta$-plane equation
\eqref{bpe}-\eqref{bce}.

{\rm(i)} If $\beta>0$, then the wave speed $c$ must fall into one of the following four categories:

\begin{enumerate}
\item  $c$ is a generalized inflection value of $u$, i.e. $\{\beta-\Delta u=0\}\cap \{u=c\}\neq\emptyset$;

\item
 $c$ is a critical value of $u$, i.e. $\{\nabla u=0\}\cap \{u=c\}\neq\emptyset$;
 \item
    $c$ is  a maximum or minimum  of $u$;
    \item
    $c$ is outside $Ran (u)=[u_{\min},u_{\max}]$, specifically, $c\in [c_\beta^+, u_{\min})$ with
    \begin{align}\label{def-c+}
c^{+}_{\beta}=u_{\min}-\frac{\beta (d_+-d_-)^2}{2\pi^2}-\frac{(d_+-d_-)^2}{2\pi^2}\sqrt{\beta^2+{4\pi^2\beta\over (d_+-d_-)^2}(u_{\max}-u_{\min})}\,.
\end{align}
\end{enumerate}

{\rm(ii)} If $\beta=0$ ($f$-plane approximation), then $c$ must be a generalized inflection value of $u$.
\end{Theorem}

In the context of Theorem \ref{classification-of-wave-speed-for-a-genuinely-travelling-wave-beta-plane-intro}, we denoted
 $$f_{\max}=\max_{(x,y)\in {\frak D}_L}\{f(x,y)\}\quad\text{and}\quad f_{\min}=\min_{(x,y)\in {\frak D}_L}\{f(x,y)\}\,$$
for a function $f \in C({\frak D}_L)$.
A genuine  traveling wave is a periodic flow propagating in the East-West direction and presenting meridional variations, whereas flows with
$v \equiv 0$ throughout ${\frak D}_L$ are termed shear flows -- by the third equation in \eqref{bpe}, shear flows depend only on the $y$-variable. The boundary condition \eqref{bce}
expresses the presence of zonal jets along it, confining the wave to the respective zonal band.
 Field data shows that the zonal jet systems of Jupiter and Saturn reach below the troposphere and have remained practically unchanged
since the beginning of detailed spacecraft observations (see \cite{gr}). In addition to these zonal jets, some peculiar flow patterns occur: unlike the majority of Jupiter's and Saturn's
wind jets, near 47$^\circ$N on Saturn, wavy perturbations of a jet produce a long-lived meandering ribbon (see Fig.\,\ref{sat-rib}) -- a similar ribbon on Jupiter at 30$^\circ$N was rather short-lived. Let
us also note that
$$\Gamma=v_x-u_y+ \beta y $$
is the total vorticity, comprising two distinct components: the relative vorticity $\gamma=v_x-u_y$ due to the fluid motion
and the planetary vorticity $\beta y$ due solely to the planet's rotation about its polar axis (this being the $\beta$-plane correspondent of the concepts in
spherical coordinates, discussed in \cite{cj}). By the third equation in \eqref{bpe}, expressing mass conservation, we have $\Gamma_y=\beta -\Delta u$, which provides another perspective for the interpretation of
generalized inflection points as stationary points of the variation of the total vorticity with latitude.

	\begin{figure}[h]
    \centering
	\includegraphics[scale = 0.25]{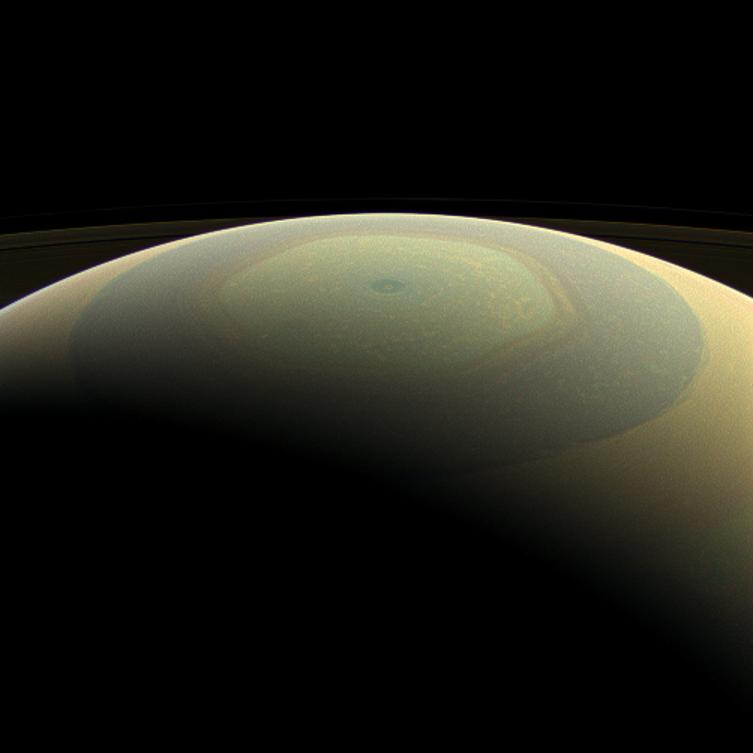}
	\caption{{\footnotesize Natural-color view of Saturn's hexagon, taken with the Cassini spacecraft wide-angle camera on 22 July 2013  (credit: NASA/JPL-Caltech/Space Science Institute). The
	stationary hexagonal band between 72$^\circ$N-78$^\circ$N, appearing somewhat yellow, has persisted dynamically since its discovery by Voyager 2 in 1981 \cite{cj2}. Only the hexagon's
	color changes seasonally, alternating between turquoise and yellow -- during Saturn's summer (lasting for about
7.5 Earth-years), sunlight triggers the formation of photochemical hazes, which give the planet's atmosphere a yellow hue.}}
	\label{sat-hex}
\end{figure}

By means of available data for wave patterns on Jupiter and Saturn observed by spacecraft or Earth-based telescopes and interpreted within the regime of the $\beta$-plane approximation, we show that each classification type of large-scale coherent waves from
Theorem \ref{classification-of-wave-speed-for-a-genuinely-travelling-wave-beta-plane-intro} occurs on the gas giants:
\begin{itemize}
\item Saturn's hexagon (see Fig.\,\ref{sat-hex}) illustrates type (i1) since for the stationary circumpolar wave
	(with speed $c=0$) embedded within the hexagonal strip near the center of the hexagonal band at 75$^\circ$N, we have $u=0$ and $\beta-u_{yy}=0$ at this latitude according to the data gathered in Fig.\,3 of \cite{tr}.

\item Saturn's ribbon (see Fig.\,\ref{sat-rib}) illustrates type (i2), with a wave speed of 150 m$\,$s$^{-1}$
matching the speed of the jet at 47$^\circ$N (see \cite{say}). Since 47$^\circ$N is a line of zero relative vorticity (see \cite{say}), taking the mean value of $v_x-u_y=0$ over a period in the
$x$-variable yields locations where $u_y=0$ at this latitude. Furthermore, $u=c$ implies $u_x=0$ along 47$^\circ$N, so that at these locations we find critical points of $u$.

\begin{figure}[h]
    \centering
	\includegraphics[scale = 0.45]{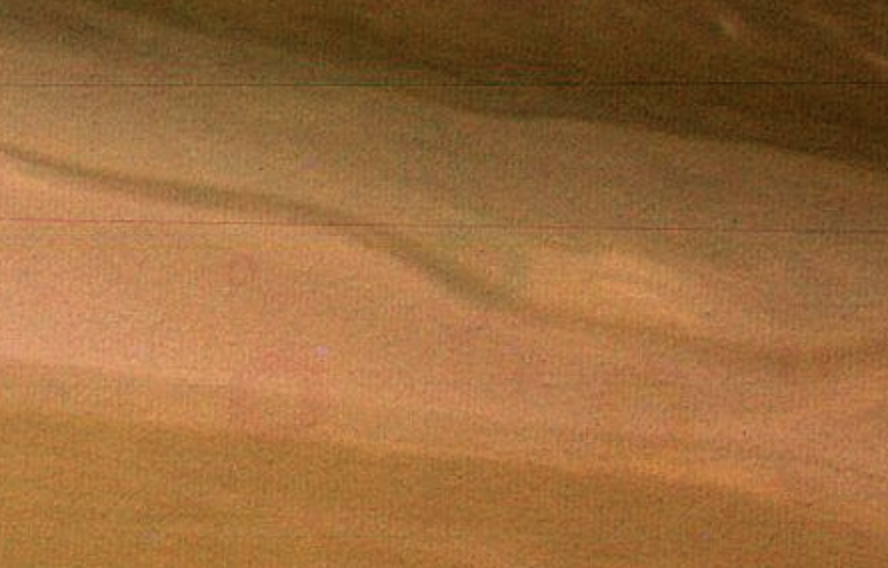}
	\caption{{\footnotesize Voyager 2 image of Saturn's ribbon, captured on
23 August 1981 (credit: NASA/JPL-Caltech/Space Science Institute). North is to the upper right and the wave is the dark meandering streak in
the light band -- in the image the eastward direction of wave propagation (along 47$^\circ$N)  is inclined downwards. Since its discovery, this wave has
been seen in every high-resolution image of Saturn that covers northern mid-latitudes \cite{gr}. The ribbon's oscillation amplitude is about 2 degrees of
latitude from peak to trough, matching the width of the jet it resides in.}}
	\label{sat-rib}
\end{figure}

\item Jupiter's ribbon at 30$^\circ$N and Jupiter's chevrons at 7.5$^\circ$S (see Fig.\,\ref{chevron}) are examples of type (i3) waves. Observations with the Hubble Space Telescope in
the period between 1994 and 2008 revealed that Jupiter's ribbon consists of
wavy oscillations within a band about $4^\circ$ of latitude wide, centered at 30$^\circ$N and propagating westwards at the maximal zonal velocity of 35 m$\,$s$^{-1}$ of the underlying jet (see \cite{cos}), so that in
this case $c=u_{\min}<0$. On the other hand, the eastward jovian jet at 7.5$^\circ$S presents chevron-shaped dark spots, consistent with a wave propagating eastwards at a speed equal  to the maximum zonal jet velocity
$u_{\max}$ of about 140 m$\,$s$^{-1}$ (see \cite{sm}).

\begin{figure}[h]
    \centering
	\includegraphics[scale = 0.5]{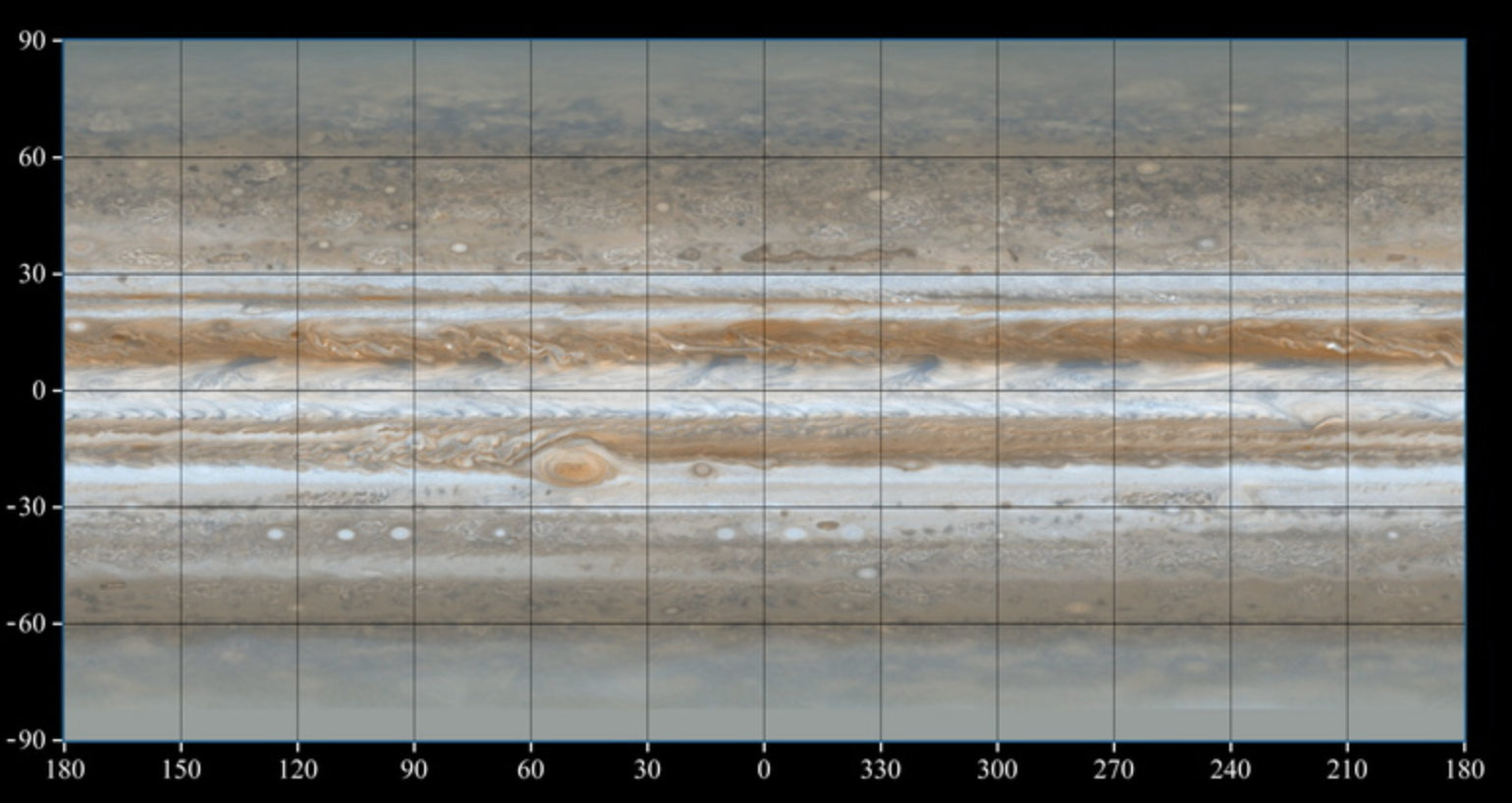}
	\caption{{\footnotesize Color map of Jupiter in cylindrical coordinates, obtained
from images taken by the camera onboard NASA's Cassini spacecraft
on 11--12 December 2000 (credit: NASA). Jupiter's GRS  is confined
to a zone delimited by a westward jet at 19.5$^\circ$S with speed 70 m$\,$s$^{-1}$ and an eastward jet at
26.5$^\circ$S with speed 50 m$\,$s$^{-1}$, chevrons propagating along the latitude 7.5$^\circ$S can be clearly distinguished eastwards of the GRS and Jupiter's ribbon
is the whitish wave pattern visible along the latitude 30$^\circ$N.}}
	\label{chevron}
\end{figure}

\item Saturn's ribbon at 42$^\circ$N was a type (i4) wave observed in the period 2011--2017 by the Cassini spacecraft mission. In December 2010 a violent storm erupted on
Saturn, visible at cloud level between 32$^\circ$N and 38$^\circ$N, lasted for six months and led to a wavy perturbation of the westward jet at 42$^\circ$N, propagating at a speed slightly less than
the minimal zonal jet velocity of about 54 m$\,$s$^{-1}$ (see \cite{gun}).
\item The wave dynamics of the zonal band on Jupiter between the westward jet at 19.5$^\circ$S and the eastward jet at 26.5$^\circ$S is quite intricate, comprising the GRS, a filamentary oscillation near its southern boundary
and a more intricate wave motion near its northern boundary (see Fig.\,\ref{chevron}). Within the $f$-plane approximation and with the origin of the coordinate system chosen at the center of the GRS, the nondimensional velocity field
within the GRS is of the stationary form
$$u(x,y)=- y\,\mu(\sqrt{x^2+y^2}) \quad \text{and}\quad v(x,y)=x\,\mu(\sqrt{x^2+y^2})\,,$$
where $\mu(s)=a -\sqrt{a^2 -b^2 s^k}$ with $a >b >0$ and $k >2$ (for the exact values we refer to \cite{cj1}). Since $u=\Delta u=0$ at the core of the GRS because $\mu'(0)=\mu''(0)=0$, this wave motion is of type (ii).
\end{itemize}
Note that, according to the recent survey of observed wave-like phenomena on Jupiter \cite{or}, the vast majority of  waves are found
at latitudes with strong prevailing eastward winds, heading in the same direction as
Jupiter's rotation. On Saturn the prevalence of waves in regions associated with prograde
motions of the mean zonal flow is also noticeable.

Upper and lower bounds for the wave speeds of genuine traveling waves are established in Theorems 3.2 and 3.11 of \cite{csz2024a} for $\beta >0$  and in \cite{Kalisch12,Hamel2017} for $\beta=0$.
We summarize these results in the caption of Fig.\,\ref{fig-beta-plane}.

\begin{figure}[ht]
    \centering
	\includegraphics[scale = 0.5]{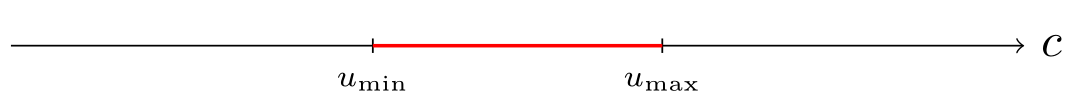}\quad\includegraphics[scale = 0.45]{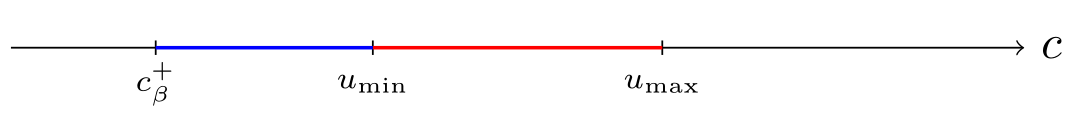}
	\caption{
For a  genuine traveling wave $(u(x-ct,y),v(x-ct,y))\in C^2({\frak D}_L)$, for $\beta=0$ the wave speed $c$ must belong to $Ran(u)$, depicted by the red interval, while for $\beta>0$
 the wave speed $c$ either belongs to $Ran(u)$ or to the interval $[c^{+}_{\beta},u_{\min})$, depicted in blue.
}
	\label{fig-beta-plane}
\end{figure}

In this paper, we provide by means of Theorem \ref{classification-of-wave-speed-for-a-genuinely-travelling-wave-beta-plane-intro} a complete classification of
genuine traveling waves within the framework of the $\beta$-plane approximation. Explicit mathematical examples of traveling waves  for each class described
in Theorem \ref{classification-of-wave-speed-for-a-genuinely-travelling-wave-beta-plane-intro} are given in Examples \ref{ex-generalized-inflection-values}-\ref{ex-beta=0-generalized-inflection-values}, in addition to the observations described above and illustrated in Figs.\,\ref{sat-hex}-\ref{chevron}.
To prove Theorem \ref{classification-of-wave-speed-for-a-genuinely-travelling-wave-beta-plane-intro},
we first analyze in detail which values of $c$, located within the range of the zonal velocity,  can arise as wave speeds of genuine traveling
waves (see Theorem \ref{wave-speed-inside-range}).
By combining this analysis with the case  when $c$ lies outside $Ran (u)$, we then establish the full classification of wave speeds, thereby proving Theorem \ref{classification-of-wave-speed-for-a-genuinely-travelling-wave-beta-plane-intro}.
 For a genuine traveling
wave $(u(x-ct,y),v(x-ct,y))$, the values of $c\in Ran (u)$ are of the following three types: a generalized inflection value of $u$, or
 a critical value of $u$, or
a maximum/minimum  of $u$ if $\beta>0$, while for $\beta=0$ the values of $c\in Ran (u)$ can only be generalized inflection values of $u$.
 The proof for $\beta>0$ proceeds  by contradiction and consists of two main parts. In the first part, we use the $F$-formulation (see Lemma \ref{F-formulation-lem}) introduced in \cite{csz2024a} and the sign of $\beta$ to demonstrate that
$v=0$ on the side $\{u<c\}$ of the level set
$\{u=c\}$. This also enables us to show that the level set $\{u=c\}$ consists of several straight segments parallel to the zonal direction. In the second part, we show that
$v=0$ on the side
$\{u>c\}$ of the level set $\{u=c\}$.
We pursue this as a
unique continuation problem for the $2$-dimensional elliptic equation $\Delta v+{\beta-\Delta u\over u-c}v=0$.
A natural approach is to reduce it to a $1$-dimensional setting and
to analyze it using classical ODE techniques, such as the Gr\"{o}nwall inequality or Fuchs-type methods.
One attempt is to fix an arbitrary $x_0$ and to analyze the resulting ODE, but this generates an inhomogeneous term $v_{xx}(x_0,\cdot)$ that is difficult to control. Alternatively, one can exploit the periodicity in
$x$ and examine the equation mode by mode in Fourier space. In this framework, the nonlinear potential term induces intricate interactions among different modes,
yielding a strongly coupled ODE system that is challenging to handle.
 Consequently, we abandon the $1$-dimensional reduction and instead adopt a $2$-dimensional approach.
 We employ the Carleman estimate \eqref{Carleman-estimate1}-\eqref{Carleman-estimate2} adapted to a boundary-aligned ball  (see Fig.\,\ref{fig-uni-conti}) around the level set $\{u=c\}$ to handle the Laplacian term. This localized estimate provides the necessary weighted
$L^2$
 bounds for
$w$ and
$\nabla w$, allowing the lower-order terms to be absorbed by the Laplacian term and yielding a quantitative unique continuation inequality in the ball, where $w\in C^2$ vanishes both near the center  and outside the ball.
A primary obstacle arises from the singular potential term
${\beta-\Delta u\over u-c}$, which  behaves like ${1\over y-y_{i_0}}\notin L^{>1}$ near a branch $\{y=y_{i_0}\}$ of the level set $\{u=c\}$, so that the unique continuation results available in the
research literature \cite{Schechter-Simon1980, Amrein-Berthier-Georgescu1981, Jerison-Kenig1985} cannot be invoked in this context.
The key observation is that the vanishing of the meridional velocity $v$ on the level set $\{u=c\}$ effectively weakens this singularity,
 enabling the application of the Hardy inequality in the
normal direction. This in turn allows the weighted $L^2$
 contribution of the singular potential to be controlled by that of the lower-order terms.
 Combined with the Carleman estimate, this provides us with the mechanism that ultimately establishes the quantitative unique continuation result across $\{y=y_{i_0}\}$.
We point out that our method establishes a 2-dimensional unique continuation mechanism
for elliptic equations with a borderline ${1\over y-y_{i_0}}$-type potential, when the solution vanishes on
the singular set.

Note that Fig.\;\ref{fig-beta-plane} may alternatively be viewed as illustrating a rigidity phenomenon: namely, for a   traveling wave $(u(x-ct,y),v(x-ct,y))\in C^2({\mathfrak D}_L)$, if the wave speed  $c$ is lying in the black intervals,
  then it must be a shear flow. This gives rise to the natural question whether there are sufficient conditions to ensure that a traveling wave with wave speed in the red and blue intervals of Fig.\;\ref{fig-beta-plane} must be a shear flow.
This question is challenging due to the singularity in the term ${1\over u-c}$ when the  wave speed $c$ lies inside the range of zonal velocity $u$. We present some sufficient conditions to ensure that:
(i) a traveling wave with wave speed in the red intervals of Fig.\;\ref{fig-beta-plane} (i.e. in the range of the zonal velocity) must be a shear flow (see Theorem \ref{beta=0-cor});
(ii) a traveling wave with wave speed in the blue and red intervals of Fig.\;\ref{fig-beta-plane} (i.e. with arbitrary wave speed)
 must be a shear flow (see Theorem \ref{thm-generalization1}). Note that no restrictions are imposed on the amplitudes of the traveling waves.
 The sufficient condition in Theorem \ref{thm-generalization1} is
 a narrow range of $\beta$, and the proof makes full use of the $F$-formulation introduced in \cite{csz2024a} combined with structural properties of the quasi-geostrophioc  equation.
 In particular, the argument establishing $v =0$ on the set $\{u > c\}$ differs from that in Theorem \ref{wave-speed-inside-range}, where the analysis is simplified by exploiting the sign of $\beta - \Delta u$.
 We also extend some of the classification and rigidity results for traveling waves
to an unbounded channel in Theorems \ref{beta-plane-unbounded-thm-wave-speeds-outside-range-zonal-velocity},  \ref{beta-plane-unbounded-channel-thm} and \ref{beta-plane-unbounded-channel-thm-rigidity}.

We also investigate whether the rigidity of  traveling waves (that is, the absence of genuine traveling
waves) holds, or if genuine traveling waves exist near some classes of shear flows.
The clarification of this issue plays an important role in the study of long time dynamics  near a shear flow \cite{LZ,LWZZ,Coti Zelati23}.
On the one hand, the rigidity  of nearby traveling waves is a necessary condition for
 nonlinear inviscid damping of the   shear flow  and thus is a first step towards the asymptotic stability.
 On the other hand, if  genuine nearby traveling waves do exist, then the  long
time dynamics near the shear flow is richer, as these waves might be `end states' of nearby evolutionary flows.

We study the rigidity issue for monotone shear flows in Section \ref{monotone-shear-flows}.
In an unbounded channel,
as an application of Theorem \ref{beta-plane-unbounded-channel-thm}, we give a simple description of the wave speeds of traveling waves near a monotone shear flow, from which, under the Rayleigh stability condition,
we infer the absence  of genuine nearby traveling waves (see Theorem \ref{Monotone-shear-flows-on-an-unbounded-channel-thm} and Remark \ref{rem-rigidity-unbounded-channel}).
 In a bounded channel, unlike the unbounded case, the situation is more complicated.
 Under the Rayleigh stability condition, we
  fully characterize the parameter regimes for $(\beta,L)$ that lead to the rigidity of traveling waves for arbitrary wave speeds versus those that admit genuine unidirectional traveling waves which are
$C^2$-close to the monotone shear flow.
 When the Rayleigh stability condition does not hold, we determine a parameter regime for $(\beta,L)$ that admits genuine unidirectional nearby traveling waves. Outside this regime, the wave speed of
 any genuine nearby traveling wave must be a generalized inflection value (see Theorem \ref{rigidity-near-monotone-shear-flow-arbitrary-wave-speed-thm}).
 To prove this for a wave speed locating outside the blue interval in Fig.\;\ref{fig-beta-plane}, we can apply Theorems \ref{classification-of-wave-speed-for-a-genuinely-travelling-wave-beta-plane} and \ref{beta=0-cor}. However,
in general, it is hard to determine whether the rigidity of traveling waves with wave speeds in the blue interval holds or not. Let us point out that our approach for
the rigidity part in Theorem \ref{rigidity-near-monotone-shear-flow-arbitrary-wave-speed-thm}
  differs from that in \cite{WZZ}, where the method depended heavily on the uniform $H^4$-bound of the $L^2$ normalized meridional velocity, which requires higher regularity. We avoid studying the limit of a sequence of traveling waves, paying
  instead attention to the intrinsic structure of the nonlinear equation for a specific traveling wave
  and taking into account the spectral properties of the linearized operator associated with the $\beta$-plane equation around the monotone shear flow. Basically, the spectral properties provide us a  control from below for the linear part by a uniform $H^1$  bound of the meridional velocity, while the nonlinear part turns out to provide a control from above  by a smaller  $H^1$  bound, which is an impossible set-up.
  For the existence part in Theorem \ref{rigidity-near-monotone-shear-flow-arbitrary-wave-speed-thm}, we can not directly apply the bifurcation lemma established in Lemma 2.5 of \cite{LWZZ}, since the shear-flow profile lacks sufficient regularity. Our approach is to perturb it to a nearby shear flow with higher regularity. An issue that arises  is that it is not clear whether the spectral condition in \cite{LWZZ} is satisfied for the perturbed shear flow. To address this, we
  establish in Lemma \ref{lambda1continuous-profile-lem} the continuous dependence of the principal eigenvalue of the singular Rayleigh-Kuo boundary value problem on the shear-flow profile in $C^2$.

In contrast to the unbounded channel, where rigidity (for arbitrary $c$) holds for all $(\beta, L)$ under the Rayleigh stability condition,
 the bounded channel admits nearby genuine traveling waves for $(\beta, L)$ lying in a critical parameter regime. For $(\beta, L)$ in this critical parameter regime,
we further  prove that the monotone shear flow is
nonlinearly Lyapunov stable.
Nevertheless, the nearby coherent traveling waves  act as nontrivial asymptotic states so that the shear flow is asymptotically unstable (see Theorem \ref{nonlinear-Lyapunov-stable-not-asymptotically-stable-monotone-shear-flow}).
This can be viewed as a subtle dynamical phenomenon arising from the inclusion of the Coriolis force, which is a hallmark of geophysical effects.
The mechanism leading to asymptotic instability is the appearance of genuine traveling waves with wave speeds outside the zonal velocity, a situation that can never occur when the Coriolis effect is neglected.

 Furthermore, by applying
Theorem \ref{thm-generalization1}, we obtain some rigidity results for traveling waves with arbitrary
wave speeds near the Couette-Poiseuille flow and near the Bickley jet for a certain $\beta$-range (see Propositions
\ref{Couette-Poiseuille-beta-plane} and \ref{bickley-beta-plane}), and then we generalize
them to a class of shear flows in Proposition \ref{beta-plane-general}. In the $f$-plane setting, as an application of the classification theorem, we also prove the rigidity of traveling waves
with $c\neq0$  near a Kolmogorov flow, and,  in contrast, we construct nearby
non-sheared  steady flows in Proposition \ref{inviscid-dynamical-structures-near-Kolmogorov-flow}.

The results of this paper complement and improve some recent findings, specific details being provided in remarks throughout the paper.
For the $2$-dimensional incompressible Euler equation  in a bounded periodic channel, the rigidity of steady flows
 with no horizontal stagnation points is proved in \cite{Kalisch12}. The sufficient  condition to ensure the rigidity of steady flows is weakened to flows with
 no stagnation points in \cite{Hamel2017} and with the laminar property in \cite{dn2024}. On the other hand,
 near the Couette flow, the rigidity of traveling waves with arbitrary speeds is proved in (velocity) $H^{5\over2}$
\cite{LZ}. Near the Poiseuille flow, the rigidity of traveling waves with arbitrary speeds is proved in $H^{>6}$  \cite{Coti Zelati23}, and
the rigidity of steady flows is proved in $C^2$ \cite{dn2024}.    See also
\cite{Hamel2019,PConstantin2021,Hamel2023,WZ2023,gxx2024,EHSX2024,dn2024} for  other rigidity results.

The rest of the paper is organized as follows. In Section 2 we recall the governing equations in
vorticity form. Section 3 contains the core classification results: we first analyze which values of
$c$ lying in $Ran (u)$ can serve as wave speeds of genuine traveling waves (Theorem \ref{wave-speed-inside-range}), and then
combine this with the case $c\notin Ran (u)$ to obtain the full classification of wave speeds (Theorem \ref{classification-of-wave-speed-for-a-genuinely-travelling-wave-beta-plane-intro}). We also discuss  sufficient conditions
for the rigidity of traveling waves with arbitrary amplitudes.
Section 4 extends some of the classification and rigidity results to unbounded channels. In Section 5 we fully characterize
the parameter regimes for $(\beta, L)$ that yield the following dynamics near a monotone shear flow: (i) the absence  of genuine traveling waves; (ii) the existence of genuine traveling waves; and (iii)
the constraint that  wave speed of any possible genuine  traveling wave must be a generalized inflection value. Section 6 is devoted to some further
applications  to specific non-monotone profiles.

\section{The governing equations}

It is convenient to write the $\beta$-plane equation \eqref{bpe} in the form
\begin{align}\label{Euler equation}
		\partial_{t}\vec{u}+(\vec{u}\cdot\nabla)\vec{u}=-\nabla P-\beta yJ\vec
{u}\,,\qquad  \nabla\cdot\vec{u}=0 \,,
	\end{align}
where $\beta \ge 0$ is the Coriolis parameter, $\vec{u}=(u,v)$ is the fluid velocity, $P$ is the modified pressure, and
$
J=%
\begin{pmatrix}
0 & -1\\
1 & 0
\end{pmatrix}
$
is the rotation matrix. We study the fluid flow described by \eqref{Euler equation} in the domain
\[D_L=\{(x,y): x\in\bbT_L,\;y\in[-d,d]\}\]	
with the non-permeable boundary condition
\begin{align}\label{boundary condition for euler}
v=0\quad \text{on}\quad  y=\pm d.
\end{align}
The non-dimensional boundary-value problem \eqref{Euler equation}-\eqref{boundary condition for euler} in $D_L$ with $d=\frac{d_+ - d_-}{2}$ is equivalent to the system \eqref{bpe}-\eqref{bce}, as one
can see by performing the translation $(x,y) \mapsto (x,y-y_0)$ with $y_0=\frac{d_++d_-}{2}$ the meridional distance from the Equator to the center of the zonal band ${\frak D}_L$ and by adjusting the fluid
pressure ${\frak P}$ to the modified pressure $P={\frak P} - (f_0+\beta y_0)\psi$, taking advantage of the existence of a stream function $\psi$ with
\[(u,v)=(-\pa_y\psi,\pa_x\psi),\]
ensured by the incompressibility condition $\nabla\cdot\vec{u}=0$.

The vorticity form of (\ref{Euler equation}) is the 2-dimensional quasi-geostrophic vorticity equation
\begin{equation}\label{vore}
\partial_{t}\gamma+(\vec{u}\cdot\nabla)\gamma+\beta v=\partial_{t}\gamma+\{\psi,\gamma\}+\beta v%
=0,
\end{equation}
where $\{\psi,\gamma\}=\partial_x\psi\partial_y\gamma-\partial_y\psi\partial_x\gamma$ is the Poisson bracket, and $\gamma=\partial_{x}v-\partial_{y}u=\Delta\psi$ is the vorticity of the flow. Alternatively,
we can write \eqref{vore} in the form
\begin{equation}\label{tvore}
\partial_{t}\Gamma+(\vec{u}\cdot\nabla)\Gamma=0\,,
\end{equation}
where $\Gamma=\gamma + \beta y$ is the total vorticity. Note that \eqref{tvore} expresses the conservation of total vorticity along a fluid trajectory.

\section{Classification and rigidity of traveling waves}

In this section, we present a comprehensive classification of the genuine traveling-wave solutions of the $\beta$-plane equation. Subsequently, without imposing any smallness constraint on the amplitude, we derive some sufficient conditions ensuring the rigidity of traveling waves with wave speeds lying in the red and blue intervals depicted in Fig.\,\ref{fig-beta-plane}.

\subsection{Classification of genuine traveling waves}
We first investigate which values within the range of the zonal velocity may occur as the wave speed of a genuine traveling wave.

\begin{Theorem}\label{wave-speed-inside-range}
Let $(u(x-ct,y),v(x-ct,y))\in C^2(D_L)$
 be a traveling-wave solution to the $\beta$-plane equation \eqref{Euler equation}-\eqref{boundary condition for euler}.
Assume that $\beta-\Delta u\neq0$ whenever $u=c$, and one of the following conditions

{\rm(i)} $\beta>0$, $c\in (u_{\min}, u_{\max})$ and $\nabla u\neq0$ whenever $u=c$,

{\rm(ii)} $\beta=0$ and $c\in Ran(u)$\\
holds.
Then $(u(x-ct,y),v(x-ct,y))$ is a shear flow.
\end{Theorem}

Theorem \ref{wave-speed-inside-range}, together with Theorem 1.1 in \cite{Hamel2017}, Theorems 3.2 and 3.11 in \cite{csz2024a}, provides a complete classification of the genuine traveling-wave solutions.

\begin{Theorem}[=Theorem \ref{classification-of-wave-speed-for-a-genuinely-travelling-wave-beta-plane-intro}]\label{classification-of-wave-speed-for-a-genuinely-travelling-wave-beta-plane}
Let $(u(x-ct,y),v(x-ct,y))\in C^2(D_L)$
 be a genuine traveling-wave solution to the $\beta$-plane equation \eqref{Euler equation}-\eqref{boundary condition for euler}.

{\rm(i)} If $\beta>0$, then the wave speed $c$ must fall into one of the following four categories:

\begin{itemize}
\item  $c$ is a generalized inflection value of $u$, i.e. $\{\beta-\Delta u=0\}\cap \{u=c\}\neq\emptyset$,

\item
 $c$ is a critical value of $u$,
 \item
    $c$ is  a maximum or minimum  of $u$,
    \item
    $c\in [c_\beta^+, u_{\min})$, with $c_\beta^{+}$  defined in \eqref{def-c+} for $d_\pm=\pm d$.
\end{itemize}

{\rm(ii)} If $\beta=0$, then the wave speed $c$ must be a generalized inflection value of $u$.
\end{Theorem}

We now  present the proof of Theorem \ref{wave-speed-inside-range} and of Theorem \ref{classification-of-wave-speed-for-a-genuinely-travelling-wave-beta-plane}.
To this end, we use the $F$-formulation of the $\beta$-plane equation introduced in \cite{csz2024a}.
\begin{Lemma}\label{F-formulation-lem}
Let $(u(x-ct,y),v(x-ct,y))\in C^2(D_L)$
 be a traveling-wave solution to the $\beta$-plane equation \eqref{Euler equation}-\eqref{boundary condition for euler}. Define
 \begin{align}\label{def-F-beta-plane}
F={v\over u-c}
\end{align}
on $\{u\neq c\}$. Then $F$ solves
\begin{align}\label{F-equation}
\partial_{x}((u-c)^2\partial_{x}F)+\partial_{y}((u-c)^2\partial_{y}F)=-\beta(u-c)F
\end{align}
on $\{u\neq c\}$.
\end{Lemma}

\begin{proof} [Proof of Theorem \ref{wave-speed-inside-range}]
The idea of the proof for $\beta>0$ is as follows. On the side
$\{u<c\}$ of the level set
$\{u=c\}$, we make use of the
$F$-formulation of the governing equation in Lemma \ref{F-formulation-lem} and take the advantage of the sign of $\beta$  to show that
$v=0$. This in turn implies that each connected component of
$\{u=c\}$ must be a straight segment parallel to the
zonal direction and the number of such segments is finite. On the other side
$\{u>c\}$, we employ a boundary-adapted Carleman estimate combined with a Hardy inequality to establish
unique continuation for an elliptic equation with a borderline ${1\over y-y_{i}}$-type potential near a branch $\{y=y_{i}\}$ of the level set $\{u=c\}$. This yields  continuation of $v=0$ across
$\{u=c\}$ and hence throughout the entire domain.
In the $f$-plane setting ($\beta=0$), the key feature that $\beta=0$ renders equation \eqref{F-equation}, and hence the subsequent analysis, considerably simpler than those in the $\beta$-plane setting.

Since
$(u(x-ct,y),v(x-ct,y))$ is a traveling-wave solution to the $\beta$-plane equation \eqref{Euler equation}-\eqref{boundary condition for euler},
we have
\begin{align}\label{un-vn-eq}
{(u-c)}\pa_x\gamma+v(\pa_y\gamma+\beta)=0.
\end{align}
By the incompressibility condition we have $\partial_x\gamma=\Delta v$ and $\partial_y\gamma=-\Delta u$, and thus \eqref{un-vn-eq} can be written as
\begin{align}\label{un-vn-eq-u-v}
{(u-c)}\Delta v+v(-\Delta u+\beta)=0.
\end{align}
We  divide $D_L$ into three parts:
\begin{align}\label{def-S+-0}
S^+=\{(x,y):u>c\}, \quad S^0=\{(x,y):u=c\}, \quad S^-=\{(x,y):u<c\}.
\end{align}

Since $\beta-\Delta u=\beta+\pa_y\gamma\neq0$ whenever $u=c$,
we infer from  \eqref{un-vn-eq} that
\begin{align}\label{v-u-c0}
v=0 \quad \text{on }\quad  S^0\cup\{y=\pm d\}.
\end{align}
Define $F$ by \eqref{def-F-beta-plane} on $S^{\pm}$.
 By Lemma \ref{F-formulation-lem}, $F$ solves the formulation \eqref{F-equation} of the $\beta$-plane equation on $S^{\pm}$.
Any $(x_0,y_0)\in S^0$ is a boundary point of $S^{\pm}$, and thus,
 since  $\beta-(\Delta u)(x_0,y_0)\neq0$, by the continuity of $\Delta u$ there exists a neighborhood $U_{(x_0,y_0)}$ of $(x_0,y_0)$ such that $\beta-\Delta u\neq0$ on $U_{(x_0,y_0)}$.
 Then
 by \eqref{un-vn-eq-u-v}, \eqref{v-u-c0} and $(u,v)\in C^2(D_L)$   we have
\begin{align}\label{boundary-term1}
&\lim_{U_{(x_0,y_0)}\cap S^{\pm}\ni(x,y)\to(x_0,y_0)}((u-c)^2F\partial F)(x,y)\\\nonumber
=&\lim_{U_{(x_0,y_0)}\cap S^{\pm}\ni(x,y)\to(x_0,y_0)}v\left(\partial v-\partial u{\Delta v\over\Delta u-\beta }\right)(x,y)=0,
\end{align}
where $\partial=\partial_x$ or $\partial_y$. Any $(x_0,y_0)\in \{y=\pm d\}\setminus S^0$ is a boundary point of $S^{\pm}$ and thus
$u(x_0,y_0)-c\neq0$ ensures, by the continuity of $ u$, the existence of a neighborhood $V_{(x_0,y_0)}$ of $(x_0,y_0)$ such that $u-c\neq0$ on $V_{(x_0,y_0)}$. By $v(x_0,y_0)=0$ and $(u,v)\in C^1(D_L)$, we have
\begin{align}\label{boundary-term2}
&\lim_{V_{(x_0,y_0)}\cap S^{\pm}\ni(x,y)\to(x_0,y_0)}((u-c)^2F\partial F)(x,y)\\\nonumber
=&\lim_{V_{(x_0,y_0)}\cap S^{\pm}\ni(x,y)\to(x_0,y_0)}\left(v\partial v-{v^2\partial u\over u-c}\right)(x,y)=0.
\end{align}

Let us first prove the statement under the assumption (i). Note that
\begin{align}\label{S-pm-0-T-times-d}
S^0\cap(\mathbb{T}_L\times (-d,d))\neq\emptyset\quad\text{and}\quad S^{\pm}\cap(\mathbb{T}_L\times (-d,d))\neq\emptyset.
\end{align}
In fact,
we choose a curve $\varrho$ such that $\varrho\setminus\{(\tilde x_i,\tilde y_i)\}_{i=1}^2\subset \mathbb{T}_L\times (-d,d)$ and $\varrho$ connects the points $(\tilde x_1,\tilde y_1)$ and $(\tilde x_2, \tilde y_2)$ at which
$u$ attains its minimum and maximum in $D_L$, respectively.
 Since $c\in (u_{\min},u_{\max})$, there exists $(\tilde x_3,\tilde y_3)\in \varrho\setminus\{(\tilde x_i,\tilde y_i)\}_{i=1}^2\subset \mathbb{T}_L\times (-d,d)$ such that $u(\tilde x_3,\tilde y_3)=c$. Thus $S^0\cap(\mathbb{T}_L\times (-d,d))\neq\emptyset.$
 Then $\nabla u(\tilde x_3,\tilde y_3)\neq0$. This implies that $\partial_{x}u(\tilde x_3,\tilde y_3)\neq0$ or $\partial_{y}u(\tilde x_3,\tilde y_3)\neq0$.
 If $\partial_y u(\tilde x_3,\tilde y_3)\neq0$ (resp. $\partial_x u(\tilde x_3,\tilde y_3)\neq0$), there exists $\tilde \delta>0$ small enough such that one of the line segments $\{(\tilde x_3, y):y\in (\tilde y_3-\tilde \delta,\tilde y_3)\}$ and $\{(\tilde x_3, y):y\in (\tilde y_3,\tilde y_3+\tilde \delta)\}$ (resp. $\{(x, \tilde y_3):x\in (\tilde x_3-\tilde \delta,\tilde x_3)\}$ and $\{(x, \tilde y_3):x\in (\tilde x_3,\tilde x_3+\tilde \delta)\}$)  lies in $S^+$ and the other is in $S^-$. This validates \eqref{S-pm-0-T-times-d}.

 Based on the assumption that $c$ is not a critical value of $u$ (i.e. $\nabla u\neq0$ whenever $u=c$), we claim that
  \begin{align}\label{gradient-partial-x-u}
  \partial_{x}u=0\quad\text{whenever}\quad u=c,
  \end{align}
  while
  \begin{align}\label{gradient-partial-y-u}
  \partial_{y}u\neq0\quad\text{whenever}\quad u=c.
  \end{align}
In fact, by \eqref{S-pm-0-T-times-d}, we have
$S^{\pm}\cap(\mathbb{T}_L\times (-d,d))\neq\emptyset$.
 Multiplying \eqref{F-equation} by $F$ and integrating by parts, we infer from $\beta>0$ and
   the definition of $S^-$   that
\begin{align}\label{beta+S-}
0\leq\int_{S^-}(u-c)^2|\nabla F|^2dxdy=\int_{S^-} \beta(u-c)|F|^2dxdy\leq0.
\end{align}
Here, the boundary of $S^-$ is contained in $S^0\cup\{y=\pm d\}$, and the boundary terms arising from integration by parts in \eqref{beta+S-} vanish due to \eqref{boundary-term1}-\eqref{boundary-term2}.
Then $F=0$ and
 \begin{align}\label{v=0S-beta+}
  v=0 \quad\text{on}\quad S^-.
  \end{align}
Since $S^{-}$ is relatively open in $D_L$ and $v=0$ on $S^-$, we have
$\nabla v=0$ on  $S^{-}$. By the continuity of  $\nabla v$,  this extends to $\nabla v=0$ on the boundary set $\{u=c\}$ (i.e. on $S^0$). Due to the incompressibility condition
we have $\partial_{x}u=-\partial_{y}v=0$ whenever $u=c$, which proves \eqref{gradient-partial-x-u}.
Under the assumption that $\nabla u\neq0$ whenever $u=c$, it follows that \eqref{gradient-partial-y-u} holds.

 Using \eqref{gradient-partial-x-u}-\eqref{gradient-partial-y-u} and the Implicit Function
Theorem, we see that  every branch of the level set $\{u=c\}$ is
a straight segment parallel to the zonal direction, spanning the full period.
 Since $\partial_yu$ is continuous and nonzero on the compact set $\{u = c\}$, it attains a
strictly positive minimum in absolute value. This implies that distinct segments
 have a uniform separation and only finitely many such segments can fit in the
interval $(-d, d)$.
 We denote $\{u=c\}\cap(\mathbb{T}_L\times (-d,d))=\{x\in\mathbb{T}_L, y=y_i\}_{i=1}^{k}$ (see Fig.\,\ref{fig-u-c=0}), where $y_i<y_{i+1}$ for $1\leq i\leq k-1$.
\begin{figure}[h]
\begin{center}
 \begin{tikzpicture}[scale=0.7]
 \draw  (-4, 0).. controls (-4, 0) and (4, 0)..(4, 0);
 \draw  (-4, 5).. controls (-4, 5) and (4, 5)..(4, 5);
 \draw  (-4, 0).. controls (-4, 0) and (-4, 5)..(-4, 5);
 \draw  (4, 0).. controls (4, 0) and (4, 5)..(4, 5);

       \node (a) at (-4,-0.5) {\tiny$0$};
       \node (a) at (4,-0.5) {\tiny$L$};
       \node (a) at (-4.5,0) {\tiny$-d$};
       \node (a) at (-4.5,5) {\tiny$d$};
       \node (a) at (-4.5,0.8) {\tiny$y_1$};
       \node (a) at (-4.5,2) {\tiny$y_2$};
       \node (a) at (-4.5,3.2) {\tiny$\vdots$};
       \node (a) at (-4.5,4) {\tiny$y_k$};
       \draw  (-4, 0.8).. controls (-2, 0.8) and (2, 0.8)..(4, 0.8);
       \draw  (-4, 2).. controls (-2, 2) and (2, 2)..(4, 2);
       \draw  (-4, 4).. controls (-2, 4) and (2, 4)..(4, 4);

 \end{tikzpicture}
\end{center}
	\caption{
Configuration of a branch of $\{u=c\}$.
}
\label{fig-u-c=0}
\end{figure}
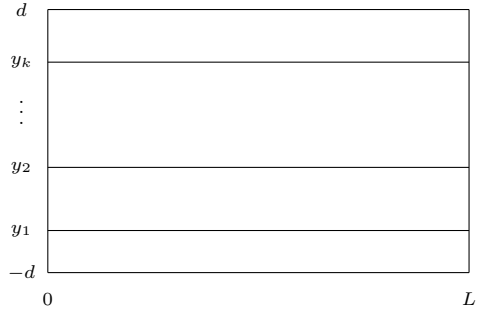

By \eqref{v=0S-beta+} we have  $v=0$ on $S^-$.
We now prove that $v=0$ on $S^+$. For this,
denote $y_0=-d$, let $y_{k+1}=d$ for convenience and
 fix $1\leq i_0\leq k$.
Since  $c$ is not a critical value of $u$, by \eqref{gradient-partial-y-u} we know that $u-c$  takes opposite signs on the two regions $\mathbb{T}_L\times(y_{i_0-1},y_{i_0})$ and $\mathbb{T}_L\times(y_{i_0},y_{i_0+1})$, indicating that one of the above two regions  lies in  $S^-$. Without loss of generality, we assume that $\mathbb{T}_L\times(y_{i_0-1},y_{i_0}) \subset S^-$. Then \begin{align}\label{v=0-on-yi0-1-yi0}
v=0\quad \text{on}\quad \mathbb{T}_L\times(y_{i_0-1},y_{i_0}).
\end{align}
Let $\varepsilon_0\in(0,\min\{1,{1\over2}(y_{i_0+1}-y_{i_0}),{1\over2}(y_{i_0}-y_{i_0-1})\})$. Fix any $x_0\in\mathbb{T}_L$. We denote by $B_{\varepsilon_0}(x_0,y_{i_0}-{\varepsilon_0\over2})$ the open ball centered at $(x_0,y_{i_0}-{\varepsilon_0\over2})$ with radius $\varepsilon_0$. Then
$B_{\varepsilon_0}(x_0,y_{i_0}-{\varepsilon_0\over2})\subset\mathbb{T}_L\times(y_{i_0-1},y_{i_0+1})$ while
$B_{\varepsilon_0\over2}(x_0,y_{i_0}-{\varepsilon_0\over2})\subset\mathbb{T}_L\times(y_{i_0-1},y_{i_0})$,
which implies
\begin{align*}
v=0\quad \text{on}\quad B_{\varepsilon_0\over2}\left(x_0,y_{i_0}-{\varepsilon_0\over2}\right).
\end{align*}
Through the approach of unique continuation for the elliptic equation
 \begin{align}\label{elliptic-equation-v}\Delta v+{\beta-\Delta u\over u-c}v=0
 \end{align}
with the singular potential ${\beta-\Delta u\over u-c}$ near $\{y=y_{i_0}\}$,  we will show that
\begin{align}\label{unique-continuation}
v=0\quad \text{on}\quad B_{\varepsilon_0}\left(x_0,y_{i_0}-{\varepsilon_0\over2}\right).
\end{align}
Note that since $\beta-\Delta u\neq0$ and $\partial_y u\neq0$ on $\{y=y_{i_0}\}$, the singularity of the potential near $\{y=y_{i_0}\}$, is ${\beta-\Delta u\over u-c}\approx{1\over y-y_{i_0}}\notin L^{>1}$, so that
the unique continuation results in \cite{Schechter-Simon1980, Amrein-Berthier-Georgescu1981, Jerison-Kenig1985} for the $2$-dimensional case do not apply -- they require a potential in $L^{>1}$.

\begin{figure}[h]
\begin{center}
 \begin{tikzpicture}[scale=0.58]
 \draw  (-4, 0).. controls (-4, 0) and (4, 0)..(4, 0);
 \draw  (-4, 5).. controls (-4, 5) and (4, 5)..(4, 5);
 \draw  (-4, 0).. controls (-4, 0) and (-4, 5)..(-4, 5);
 \draw  (4, 0).. controls (4, 0) and (4, 5)..(4, 5);

       \node (a) at (-4,-0.5) {\tiny$0$};
       \node (a) at (4,-0.5) {\tiny$L$};
       \node (a) at (-4.7,0) {\tiny$y_{i_0-1}$};
       \node (a) at (-4.7,5) {\tiny$y_{i_0+1}$};
       \node (a) at (-4.7,2.5) {\tiny$y_{i_0}$};
       \node (a) at (-5,1.8) {\tiny$y_{i_0}-{\varepsilon_0\over2}$};
       \node (a) at (-5,3.2) {\tiny$y_{i_0}+{\varepsilon_0\over2}$};
       \draw  (-4, 2.5).. controls (-2, 2.5) and (2, 2.5)..(4, 2.5);
        \draw[fill=green!70] (0,1.8) circle [radius=0.7cm];
        \draw  (-4, 1.8).. controls (-2, 1.8) and (2, 1.8)..(4, 1.8);

        \draw  (0, 0).. controls (0, 0) and (0, 1.8)..(0, 1.8);
         \draw  (-4, 3.2).. controls (-2, 3.2) and (2, 3.2)..(4, 3.2);

        \node (a) at (0,-0.6) {\tiny$x_0$};
    \draw (0,1.8) circle [radius=1.4cm];

    \begin{scope}
        \clip (0,1.8) circle [radius=1.4cm];
        \fill[red!70] (-1.4,3.2) rectangle (1.4,2.5);
    \end{scope}
      \node (a) at (6,2.5)    {\tiny$\Longrightarrow$};

 \end{tikzpicture}
\quad
  \begin{tikzpicture}[scale=0.58]
 \draw  (-4, 0).. controls (-4, 0) and (4, 0)..(4, 0);
 \draw  (-4, 5).. controls (-4, 5) and (4, 5)..(4, 5);
 \draw  (-4, 0).. controls (-4, 0) and (-4, 5)..(-4, 5);
 \draw  (4, 0).. controls (4, 0) and (4, 5)..(4, 5);

       \node (a) at (-4,-0.5) {\tiny$0$};
       \node (a) at (4,-0.5) {\tiny$L$};
       \node (a) at (-4.7,0) {\tiny$y_{i_0-1}$};
       \node (a) at (-4.7,5) {\tiny$y_{i_0+1}$};
       \node (a) at (-4.7,2.5) {\tiny$y_{i_0}$};
       \node (a) at (-5,3.2) {\tiny$y_{i_0}+{\varepsilon_0\over2}$};
       \draw  (-4, 2.5).. controls (-2, 2.5) and (2, 2.5)..(4, 2.5);

         \draw  (-4, 3.2).. controls (-2, 3.2) and (2, 3.2)..(4, 3.2);


    \fill[green!70] (-4, 2.5) rectangle (4,3.2);

 \end{tikzpicture}
\end{center}
	\caption{Unique continuation across the level set $\{u=c\}$, propagating vanishing from one side of $\{y=y_{i_0}\}$ to the other side.
}
\label{fig-uni-conti}
\end{figure}
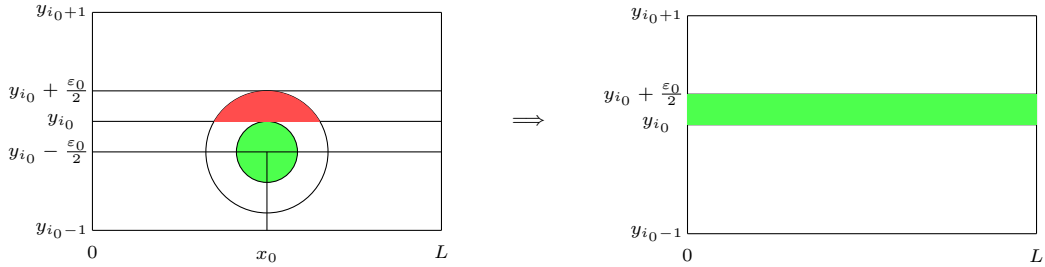

Note that $\varepsilon_0>0$ is independent of $x_0\in\mathbb{T}_L$.
If  \eqref{unique-continuation} holds true,
then by the arbitrary choice of $x_0\in\mathbb{T}_L$, we have  $v=0$ on $\mathbb{T}_L\times(y_{i_0},y_{i_0}+{\varepsilon_0\over2})$, see Fig.\,\ref{fig-uni-conti}.
Since ${\beta-\Delta u\over u-c}\in L_{loc}^\infty(\mathbb{T}_L\times(y_{i_0}+{\varepsilon\over2},y_{i_0+1}))$, by the classical unique continuation (see, for example, \cite{Carleman1939,Muller1954,Hormander1963}), we have $v=0$ on $\mathbb{T}_L\times(y_{i_0},y_{i_0+1})$, and thus on $S^+$ by the arbitrary choice of $1\leq i_0\leq k$.

The rest of the proof under the assumption (i) is devoted to establishing the validity of \eqref{unique-continuation}.
Without loss of generality, we assume that $x_0=0$ and $y_{i_0}-{\varepsilon_0\over2}=0$. Otherwise, we can perform the change of variables $(x,y)\mapsto(\tilde x,\tilde y)=(x-x_0,y-(y_{i_0}-{\varepsilon_0\over2}))$ and define $\tilde v(\tilde x,\tilde y)=v(x,y)$.
Then the analysis can be carried out  for $(0,0)$ and $\tilde v$ instead of $ (x_0,y_{i_0}-{\varepsilon_0\over2})$ and $v$.
We divide the proof of \eqref{unique-continuation} into the following two steps.
\bigskip

\noindent
{\bf{Step 1}.} The following Carleman estimate holds: for any $b>0$ and any $w\in C^2(\mathbb{T}_L\times[-d,d])$  vanishing outside $B_{\varepsilon_0}$ and on $B_{\varepsilon_0\over2}$, we have
\begin{align}\label{Carleman-estimate1}
\int_{B_{\varepsilon_0}}r^{-2b-2}e^{2r^{-b}}|w|^2dxdy\leq& {1\over2b^4}\int_{B_{\varepsilon_0}}r^{b+2}e^{2r^{-b}}|\Delta w|^2dxdy,
\end{align}
and
\begin{align}\label{Carleman-estimate2}
\int_{B_{\varepsilon_0}}e^{2r^{-b}}|\nabla w|^2dxdy\leq&
\left({1\over4b^4}+{1\over2}\varepsilon_0^b+{2\over b^2}\right)\int_{B_{\varepsilon_0}}r^{b+2}e^{2r^{-b}}|\Delta w|^2dxdy,
\end{align}
where  $B_{\varepsilon_0}=B_{\varepsilon_0}(0,0)$ and $r=\sqrt{x^2+y^2}$.\bigskip

The Carleman estimates \eqref{Carleman-estimate1}-\eqref{Carleman-estimate2} provide weighted
$L^2$
 control of
$w$ and $\nabla w$ in terms of the weighted
$L^2$
 norm of
$\Delta w$, and the coefficients on the right-hand side of \eqref{Carleman-estimate1}-\eqref{Carleman-estimate2} can be made arbitrarily small by choosing
$b>0$ sufficiently large. Moreover, the Carleman weight satisfies $e^{r^{-b}}\to \infty$ as $r\to0$, and thus enforces a strong amplification of any nonvanishing solution for $b>0$ sufficiently large. These ingredients
are essential for handling the Laplacian contribution in the quantitative unique continuation argument. The remaining obstacle to unique continuation comes from the singularity of the potential term and will be addressed in Step 2.

The condition that $w$  vanishes on $B_{\varepsilon_0\over2}$ guarantees the existence of the integrals in  \eqref{Carleman-estimate1}-\eqref{Carleman-estimate2}, while the condition that $w$
vanishes outside $B_{\varepsilon_0}$ ensures that the boundary terms in the  integration  by parts below vanish. The Carleman estimates \eqref{Carleman-estimate1}--\eqref{Carleman-estimate2} were derived in \cite{Protter1960}, a slight difference here being that
$b$
 can be chosen as an arbitrary positive number (while in \cite{Protter1960},
$b$ is required to be sufficiently large).  For completeness we provide the proof.

Let \begin{align*}
\eta=e^{r^{-b}}w
\end{align*}
on $\mathbb{T}_L\times[-d,d]$. Then $\eta$  vanishes outside $B_{\varepsilon_0}$ and on $B_{\varepsilon_0\over2}$. Since
\begin{align*}
r^{b+2}e^{r^{-b}}\left(\nabla e^{-r^{-b}}\cdot\nabla\eta\right)={b}(x\partial_x\eta+y\partial_y\eta)\quad\text{and}\quad r^{b+2}e^{r^{-b}}\Delta e^{-r^{-b}}={b^2}(r^{-b}-1),
\end{align*}
we have
\begin{align*}
r^{b+2}e^{2r^{-b}}|\Delta w|^2=&r^{b+2}\left(\Delta\eta+2e^{r^{-b}}\left(\nabla e^{-r^{-b}}\cdot\nabla\eta\right)+\eta e^{r^{-b}}\Delta e^{-r^{-b}}\right)^2\\
\geq&4r^{b+2}e^{r^{-b}}\left(\nabla e^{-r^{-b}}\cdot\nabla\eta\right)\left(\Delta\eta+\eta e^{r^{-b}}\Delta e^{-r^{-b}}\right)\\
=&4{b}(x\partial_x\eta+y\partial_y\eta)(\Delta\eta+{b^2}r^{-b-2}(r^{-b}-1)\eta ).
\end{align*}
Then
\begin{align}\label{Carleman1rightside}
\int_{B_{\varepsilon_0}}r^{b+2}e^{2r^{-b}}|\Delta w|^2dxdy\geq&4b\int_{B_{\varepsilon_0}}\Delta\eta(x\partial_x\eta+y\partial_y\eta)dxdy\\\nonumber
&+4b^3\int_{B_{\varepsilon_0}}r^{-b-2}(r^{-b}-1)(x\partial_x\eta+y\partial_y\eta)\eta  dxdy.
\end{align}
For the first term in the right-hand side of \eqref{Carleman1rightside} we have
\begin{align*}
\int_{B_{\varepsilon_0}}\Delta\eta(x\partial_x\eta+y\partial_y\eta)dxdy
=&{1\over2}\int_{B_{\varepsilon_0}}\left(x((\eta_x)^2)_x-x((\eta_y)^2)_x-y((\eta_x)^2)_y+y((\eta_y)^2)_y\right)dxdy\\
=&{1\over2}\int_{B_{\varepsilon_0}}\left(-(\eta_x)^2+(\eta_y)^2+(\eta_x)^2-(\eta_y)^2\right)dxdy=0.
\end{align*}
Thus \eqref{Carleman1rightside} becomes
\begin{align}\label{Carleman1rightside-new}
\int_{B_{\varepsilon_0}}r^{b+2}e^{2r^{-b}}|\Delta w|^2dxdy\geq&
4b^3\int_{B_{\varepsilon_0}}r^{-2b-2}(x\partial_x\eta+y\partial_y\eta)\eta  dxdy\\\nonumber
&-4b^3\int_{B_{\varepsilon_0}}r^{-b-2}(x\partial_x\eta+y\partial_y\eta)\eta  dxdy=I+II.
\end{align}
Since
\begin{align*}
\partial_x\left({x\over r^{2b+2}}\right)+\partial_y\left({y\over r^{2b+2}}\right)=&{r^2-(2b+2)x^2\over r^{2b+4}}+{r^2-(2b+2)y^2\over r^{2b+4}}
={-2b\over r^{2b+2}},
\end{align*}
we have
\begin{align*}
I=&2b^3\int_{B_{\varepsilon_0}}r^{-2b-2}(x(\eta^2)_x+y(\eta^2)_y)  dxdy\\
=&-2b^3\int_{B_{\varepsilon_0}}\left(\partial_x\left({x\over r^{2b+2}}\right)+\partial_y\left({y\over r^{2b+2}}\right) \right)|\eta|^2 dxdy\\
=&4b^4\int_{B_{\varepsilon_0}}{1\over r^{2b+2}}|\eta|^2 dxdy\geq0.
\end{align*}
Similarly,
\begin{align*}
\partial_x\left({x\over r^{b+2}}\right)+\partial_y\left({y\over r^{b+2}}\right)=&{r^2-(b+2)x^2\over r^{b+4}}+{r^2-(b+2)y^2\over r^{b+4}}
={-b\over r^{b+2}}
\end{align*}
implies
\begin{align*}
II=&-2b^3\int_{B_{\varepsilon_0}}r^{-b-2}(x(\eta^2)_x+y(\eta^2)_y)dxdy\\
=&2b^3\int_{B_{\varepsilon_0}}\left(\partial_x\left({x\over r^{b+2}}\right)+\partial_y\left({y\over r^{b+2}}\right)\right)|\eta|^2dxdy\\
=&-2b^4\int_{B_{\varepsilon_0}}{1\over r^{b+2}}|\eta|^2dxdy\\
\geq&-2b^4\int_{B_{\varepsilon_0}}{1\over r^{2b+2}}|\eta|^2dxdy=-{1\over2} I.
\end{align*}
This, along with \eqref{Carleman1rightside-new}, implies
\begin{align*}
\int_{B_{\varepsilon_0}}r^{b+2}e^{2r^{-b}}|\Delta w|^2dxdy\geq&  {1\over2}I+\left({1\over2} I+II\right)\geq{1\over2}I\\
=&2b^4\int_{B_{\varepsilon_0}}{1\over r^{2b+2}}|\eta|^2 dxdy=2b^4\int_{B_{\varepsilon_0}}{1\over r^{2b+2}}e^{2r^{-b}}|w|^2 dxdy,
\end{align*}
which proves \eqref{Carleman-estimate1}.

We show now the validity of \eqref{Carleman-estimate2}.
Since
\begin{align*}
\left|\Delta\left(e^{2r^{-b}}\right)\right|=e^{2r^{-b}}{4b^2\over r^{2b+2}}|1+r^b|\leq e^{2r^{-b}}{8b^2\over r^{2b+2}},
\end{align*}
we infer from
 integration by parts and \eqref{Carleman-estimate1} that
\begin{align*}
\int_{B_{\varepsilon_0}}e^{2r^{-b}}|\nabla w|^2dxdy=&-\int_{B_{\varepsilon_0}}e^{2r^{-b}} w\Delta wdxdy+{1\over2}\int_{B_{\varepsilon_0}}\Delta\left(e^{2r^{-b}}\right)| w|^2dxdy\\
\leq&\int_{B_{\varepsilon_0}}e^{2r^{-b}} |w||\Delta w|dxdy+{1\over2}\int_{B_{\varepsilon_0}}\left|\Delta\left(e^{2r^{-b}}\right)\right|| w|^2dxdy\\
\leq&\int_{B_{\varepsilon_0}} \left(e^{r^{-b}}r^{-b-1}|w|\right)\left(e^{r^{-b}}r^{b+1}|\Delta w|\right)dxdy+\int_{B_{\varepsilon_0}}e^{2r^{-b}}{4b^2\over r^{2b+2}}| w|^2dxdy\\
\leq&{1\over2}\int_{B_{\varepsilon_0}} e^{2r^{-b}}r^{-2b-2}|w|^2dxdy
+{1\over2}\int_{B_{\varepsilon_0}} e^{2r^{-b}}r^{2b+2}|\Delta w|^2dxdy\\
&+{2\over b^2}\int_{B_{\varepsilon_0}}r^{b+2}e^{2r^{-b}}| \Delta w|^2dxdy\\
\leq&{1\over 4b^4}\int_{B_{\varepsilon_0}}r^{b+2}e^{2r^{-b}}| \Delta w|^2dxdy
+{1\over2}\varepsilon_0^b\int_{B_{\varepsilon_0}} e^{2r^{-b}}r^{b+2}|\Delta w|^2dxdy\\
&+{2\over b^2}\int_{B_{\varepsilon_0}}r^{b+2}e^{2r^{-b}}| \Delta w|^2dxdy\\
=&\left({1\over 4b^4}+{1\over2}\varepsilon_0^b+{2\over b^2}\right)\int_{B_{\varepsilon_0}}r^{b+2}e^{2r^{-b}}| \Delta w|^2dxdy.
\end{align*}
This proves \eqref{Carleman-estimate2}.
\bigskip

\noindent
{\bf{Step 2.}} We establish the validity of \eqref{unique-continuation}.\bigskip

Let $a\in \left({\varepsilon_0\over2},\varepsilon_0\right)$.
Define $\chi\in C^\infty(\mathbb{T}_L\times[-d,d])$ to be a cutoff function such that $\chi\equiv1$ on $B_a$ and $\chi\equiv0$ outside $B_{\varepsilon_0}$. Then
\begin{align*}
\chi v\in C^2(\mathbb{T}_L\times[-d,d]) \text{ \; vanishes  \;on\;  }  B_{\varepsilon_0\over2} \text{\;  and \; outside\;  }B_{\varepsilon_0}.
\end{align*}
Since $\beta-\Delta u\neq0$ and $\partial_y u\neq0$ on $\{y=y_{i_0}\}$,
by shrinking $\varepsilon_0>0$ if necessary we have
\begin{align*}
\left|{\beta-(\Delta u)(x,y)\over u(x,y)-c}\right|=\left|{\beta-(\Delta u)(x,y)\over u(x,y)-u(x,y_{i_0})}\right|\leq {C\over|y-y_{i_0}|}
\end{align*}
 for some $C>0$, where  $(x,y)\in B_{\varepsilon_0}$.
By \eqref{v=0-on-yi0-1-yi0}, we know that $e^{r^{-b}}\chi v=0$ on $\left(\mathbb{T}_L\times(y_{i_0-1},y_{i_0})\right)\cap B_{\varepsilon_0}$.
For any $\phi\in H^{1}(a,b)$ that vanishes at some point $y_*\in[a,b]$,  the Hardy inequality \cite{Hardy-Littlewood-Polya1988,LWZZ} reads
\begin{align}\label{Hardy-inequality}
\left\|{\phi\over y-y_*}\right\|_{L^2(a,b)}^2\leq C\|\phi'\|_{L^2(a,b)}^2,
\end{align}
where $a, b\in\mathbb{R}$ with $a<b$.
Then by the Hardy inequality \eqref{Hardy-inequality} in the $y$-direction, we have
\begin{align}\label{singular-potential-estimate}
\left\|{\beta-\Delta u\over u-c}\left(e^{r^{-b}}\chi v\right)\right\|_{L^2(B_{\varepsilon_0})}^2\leq&\int_{B_{\varepsilon_0}}{C\over (y-y_{i_0})^2}\left(e^{r^{-b}}\chi v\right)^2dxdy\\\nonumber
=&C\int_{-{\sqrt{3}\over2}\varepsilon_0}^{{\sqrt{3}\over2}\varepsilon_0}
\int_{\varepsilon_0\over2}^{\sqrt{\varepsilon_0^2-x^2}}{1\over (y-{\varepsilon_0\over2})^2}\left(e^{r^{-b}}\chi v\right)^2dydx\\\nonumber
\leq&C\int_{-{\sqrt{3}\over2}\varepsilon_0}^{{\sqrt{3}\over2}\varepsilon_0}
\int_{\varepsilon_0\over2}^{\sqrt{\varepsilon_0^2-x^2}}\left|\partial_y\left(e^{r^{-b}}\chi v\right)\right|^2dydx\\\nonumber
=&C\left\|\partial_y\left(e^{r^{-b}}\chi v\right)\right\|_{L^2(B_{\varepsilon_0})}^2,
\end{align}
where we recall that $x_0=0$ and $y_{i_0}-{\varepsilon_0\over2}=0$. Note that
 \begin{align*}
 \left|\partial_y e^{r^{-b}}\right|^2=b^2r^{-2b-4}e^{2r^{-b}}y^2\leq b^2r^{-2b-2}e^{2r^{-b}}.
 \end{align*}
 By \eqref{elliptic-equation-v}, \eqref{singular-potential-estimate}, the Carleman estimate \eqref{Carleman-estimate1}-\eqref{Carleman-estimate2} and $v=\chi v$ on $B_a$, we have
\begin{align}\label{Ba-weighted-laplacian-estimate}
\int_{B_a}r^{b+2}e^{2r^{-b}}\left|\Delta v\right|^2dxdy=&\int_{B_a}r^{b+2}\left|e^{r^{-b}}{\beta-\Delta u\over u-c} v\right|^2dxdy\\\nonumber
\leq&\int_{B_{\varepsilon_0}}r^{b+2}\left|e^{r^{-b}}{\beta-\Delta u\over u-c}\chi v\right|^2dxdy\\\nonumber
\leq&\left\|{\beta-\Delta u\over u-c}\left(e^{r^{-b}}\chi v\right)\right\|_{L^2(B_{\varepsilon_0})}^2\\\nonumber
\leq&C\left\|\partial_y\left(e^{r^{-b}}\chi v\right)\right\|_{L^2(B_{\varepsilon_0})}^2\\\nonumber
\leq&C\int_{B_{\varepsilon_0}}\left(e^{2r^{-b}}|\partial_y(\chi v)|^2+\left|\partial_y e^{r^{-b}}\right|^2|\chi v|^2 \right)dxdy\\\nonumber
\leq&C\int_{B_{\varepsilon_0}}\left(e^{2r^{-b}}|\partial_y(\chi v)|^2+b^2r^{-2b-2}e^{2r^{-b}}|\chi v|^2 \right)dxdy\\\nonumber
\leq&C \left({1\over4b^4}+{1\over2}\varepsilon_0^b+{2\over b^2}\right)\int_{B_{\varepsilon_0}}r^{b+2}e^{2r^{-b}}|\Delta (\chi v)|^2dxdy\\\nonumber
&+{C\over2b^2}\int_{B_{\varepsilon_0}}r^{b+2}e^{2r^{-b}}|\Delta (\chi v)|^2dxdy\\\nonumber
=&C\left({1\over4b^4}+{1\over2}\varepsilon_0^b+{5\over 2b^2}\right)\int_{B_{\varepsilon_0}}r^{b+2}e^{2r^{-b}}|\Delta (\chi v)|^2dxdy.
\end{align}
Choose $b>0$ large enough such that $C\left({1\over4b^4}+{1\over2}\varepsilon_0^b+{5\over 2b^2}\right)<{1\over2}$. By \eqref{Ba-weighted-laplacian-estimate} we have
\begin{align*}
&\int_{B_a}r^{b+2}e^{2r^{-b}}\left|\Delta v\right|^2dxdy
\leq C\left({1\over4b^4}+{1\over2}\varepsilon_0^b+{5\over 2b^2}\right)\left(\int_{B_{a}}+\int_{B_{\varepsilon_0}\setminus B_a}\right)r^{b+2}e^{2r^{-b}}|\Delta (\chi v)|^2dxdy\\
\leq&{1\over2}\int_{B_a}r^{b+2}e^{2r^{-b}}\left|\Delta v\right|^2dxdy+
C\left({1\over4b^4}+{1\over2}\varepsilon_0^b+{5\over 2b^2}\right)\int_{B_{\varepsilon_0}\setminus B_a}r^{b+2}e^{2r^{-b}}|\Delta (\chi v)|^2dxdy,
\end{align*}
which implies
\begin{align}\label{Ba-weighted-laplacian-annulus-region-weighted-laplacian}
\int_{B_a}r^{b+2}e^{2r^{-b}}\left|\Delta v\right|^2dxdy\leq
C\left({1\over2b^4}+\varepsilon_0^b+{5\over b^2}\right)\int_{B_{\varepsilon_0}\setminus B_a}r^{b+2}e^{2r^{-b}}|\Delta (\chi v)|^2dxdy.
\end{align}
By \eqref{Carleman-estimate1} and \eqref{Ba-weighted-laplacian-annulus-region-weighted-laplacian}, we have
\begin{align}\label{Bv-zero-term-Bv-Ba-Laplacian}
&\int_{B_{\varepsilon_0}}r^{-2b-2}e^{2r^{-b}}|\chi v|^2dxdy\\\nonumber
\leq& {1\over2b^4}\int_{B_{\varepsilon_0}}r^{b+2}e^{2r^{-b}}|\Delta (\chi v)|^2dxdy\\\nonumber
=& {1\over2b^4}\left(\int_{B_{a}}+\int_{B_{\varepsilon_0}\setminus B_a}\right)r^{b+2}e^{2r^{-b}}|\Delta (\chi v)|^2dxdy\\\nonumber
\leq&{1\over2b^4}\left(C\left({1\over2b^4}+\varepsilon_0^b+{5\over b^2}\right)+1\right)\int_{B_{\varepsilon_0}\setminus B_a}r^{b+2}e^{2r^{-b}}|\Delta (\chi v)|^2dxdy.
\end{align}
Note that $C>0$ is independent of $b>0$. Using \eqref{Bv-zero-term-Bv-Ba-Laplacian} we obtain
\begin{align*}
\int_{B_a}|v|^2dxdy=&\int_{B_a}r^{2b+2}r^{-2b-2}|\chi v|^2dxdy\leq a^{2b+2}\int_{B_a}r^{-2b-2}|\chi v|^2dxdy\\
=&a^{2b+2}\int_{B_a}r^{-2b-2}e^{-2r^{-b}}\left|e^{r^{-b}}\chi v\right|^2dxdy\\
\leq&a^{2b+2}e^{-2a^{-b}}\int_{B_{\varepsilon_0}}r^{-2b-2}e^{2r^{-b}}\left|\chi v\right|^2dxdy\\
\leq &a^{2b+2}e^{-2a^{-b}}{1\over2b^4}\left(C\left({1\over2b^4}+\varepsilon_0^b+{5\over b^2}\right)+1\right)\int_{B_{\varepsilon_0}\setminus B_a}r^{b+2}e^{2r^{-b}}|\Delta (\chi v)|^2dxdy\\
\leq&a^{2b+2}{1\over2b^4}\left(C\left({1\over2b^4}+\varepsilon_0^b+{5\over b^2}\right)+1\right)\varepsilon_0^{b+2}\int_{B_{\varepsilon_0}\setminus B_a}|\Delta (\chi v)|^2dxdy\\
\to&0\quad \text{as}\quad b\to \infty,
\end{align*}
due to $a<\varepsilon_0<1$, $a^{2b+2}{1\over2b^4}\left(C\left({1\over2b^4}+\varepsilon_0^b+{5\over b^2}\right)+1\right)\varepsilon_0^{b+2}\to0$ as $b\to \infty$, and to the fact that
the integral $\int_{B_{\varepsilon_0}\setminus B_a}|\Delta (\chi v)|^2dxdy$ is independent of $b>0$.
Thus $v=0$ on $B_a$. Since $a\in({\varepsilon_0\over2},\varepsilon_0)$ is arbitrary we have $v=0$ on $B_{\varepsilon_0}$. This shows the validity of  \eqref{unique-continuation} under the assumption (i).

Now we prove the statement under the assumption (ii).
Recall the partition of the domain $D_L$ into $S^+, S^0$ and $S^-$ according to \eqref{def-S+-0},
and
define $F$ on $S^{\pm}$ as  in \eqref{def-F-beta-plane}.
Since $\beta=0$, by Lemma \ref{F-formulation-lem}  the $F$-formulation  of the $\beta$-plane equation simplifies to
\begin{align}\label{F-equation1}
\partial_{x}((u-c)^2\partial_{x}F)+\partial_{y}((u-c)^2\partial_{y}F)=0.
\end{align}
Multiplying \eqref{F-equation1} by $F$ and integrating over $S^{\pm}$, by \eqref{boundary-term1}-\eqref{boundary-term2}, we obtain
\begin{align*}
\int_{S^{\pm}}(u-c)^2|\nabla F|^2dxdy=0.
\end{align*}
Thus, $\nabla F=0$ on $S^{\pm}$ (i.e. $\nabla F=0$ on both sides of the level set $\{u=c\}$).
We show now that $v=0$ on $S^{\pm}$. Since the considerations for $S^{\pm}$ are similar, we only give the proof for $S^{-}$.  Let $S_1^{-}$ be a connected component of $S^{-}$. Then $\nabla F=0$ and thus $F={v\over u-c}=a_{1}$ on $S_1^{-}$, where $a_{1}\in\mathbb{R}$ is a constant. If $a_{1}=0$, then $F={v\over u-c}=0$ and $v=0$ on $S_1^{-}$. If $a_{1}\neq0$, by the incompressibility  condition we have $\partial_xu+\partial_y v={1\over a_{1}}\partial_x v+\partial_y v=0$ on $S_1^{-}$. Let $(x_1,y_1)\in S_1^{-}$ and $\{(x_1+{1\over a_1} s, y_1+s), s\in\mathbb{R}\}$ be a line. Then there exists $s_1>0$ such that $(x_1+{1\over a_1} s_1, y_1+s_1)\in S^0\cup\{y=\pm d\}$ and $\{(x_1+{1\over a_1} s, y_1+s), s\in[0,s_1)\}\subset S_1^-$. Since
 $$\partial_sv\left(x_1+{1\over a_1} s, y_1+s\right)=\left({1\over a_1}\partial_xv+\partial_yv\right)\left(x_1+{1\over a_1} s, y_1+s\right)=0$$
 on $s\in[0,s_1)$, $v\in C^1(D_L)$ and $v(x_1+{1\over a_1} s_1, y_1+s_1)=0$, we have $v(x_1,y_1)=0$. However, $v=a_1(u-c)\neq0$ on $S_1^-$, which is a contradiction. In conclusion, we have $
 v=0$ on  $S^-$.
\end{proof}

\begin{Remark}\label{supplement-to-Theorem-beta=0i}
  We provide a supplement to Theorem \ref{wave-speed-inside-range} $(i)$.
For a traveling-wave solution $(u(x-ct,y),v(x-ct,y))\in C^2(D_L)$ to the $\beta$-plane equation \eqref{Euler equation}-\eqref{boundary condition for euler},
if $\beta-\Delta u\neq0$ whenever $u=c$,
 $\beta>0$ and $c=u_{\max}$,
then it is a shear flow.

In fact, in the case  $\beta>0$ and $c=u_{\max}$, we have $S^+=\emptyset$ and
$v=0$ on $S^0$ due to \eqref{un-vn-eq-u-v} and $\beta-\Delta u\neq0$ whenever $u=c$. Moreover, by \eqref{beta+S-} we have
$v=0$ on $S^-$. Here, $S^{\pm}$ and $S^0$ are defined in \eqref{def-S+-0}.
\end{Remark}
\begin{proof}[Proof of Theorem \ref{classification-of-wave-speed-for-a-genuinely-travelling-wave-beta-plane}]
Let $\beta>0$. Since $(u(x-ct,y),v(x-ct,y))$ is a genuine traveling wave, by Theorems 3.2 and 3.11 in \cite{csz2024a} we have $c\in[c^{+}_{\beta},u_{\max}]$. If $c$ is outside $Ran (u)$, then $c\in [c^{+}_{\beta},u_{\min})$. Next, we consider $c\in Ran(u)$.
If $c\in(u_{\min}, u_{\max})$,
 by Theorem \ref{wave-speed-inside-range} (i),  we know that $c$ is either  a generalized inflection value of $u$ or a critical value of $u$.

For $\beta=0$, since $(u(x-ct,y),v(x-ct,y))$ is a genuine traveling wave, by Theorem 1.1 in \cite{Hamel2017} (or Theorem 1.1 in \cite{Kalisch12}) we know that $c\in Ran(u)$. By Theorem \ref{wave-speed-inside-range} (ii),   $c$ must be  a generalized inflection value of $u$.
\end{proof}

We now provide mathematical examples to demonstrate that genuine traveling waves in each of the classification types described in Theorem \ref{classification-of-wave-speed-for-a-genuinely-travelling-wave-beta-plane-intro} do exist.

\begin{Example}\label{ex-generalized-inflection-values}
In the  $\beta$-plane setting there are genuine traveling waves whose  wave speeds are  generalized inflection values of the zonal velocity. As in Example 3.1 from \cite{csz2024a}, we
consider the flow in the channel $\mathbb{T}_{2\pi}\times[-1,1]$ with stream function $\psi$ given by
\[\psi(x,y)={{A}\over{n}}\cos\left(\frac{(2k+1)\pi y}{2}\right)\sin(nx)+\tilde{A}\cos(\sqrt{-\lambda}y)+B\sin(\sqrt{-\lambda}y)
+\frac{\lambda c-\beta}{-\lambda}y+\frac{\xi}{-\lambda},\]
with integers $n, k\geq1$, $\lambda=-n^{2}-\frac{(2k+1)^2\pi^2}{4},$ and real constants $c$, $\xi$, $A$, $\tilde{A}$ and $B$. Since
\begin{align}\label{psi-equation-ex1}
\left\{
\begin{aligned}
\Delta\psi+\beta y=\lambda(\psi+cy)+\xi, \\
\partial_{x}\psi=0 \quad \text{on}\quad  y=\pm 1,
\end{aligned}
\right.
\end{align}
the corresponding velocity field  $(u(x-ct,y),v(x-ct,y))$ is a traveling-wave solution to \eqref{Euler equation}-\eqref{boundary condition for euler}.
It follows from \eqref{un-vn-eq-u-v} and \eqref{psi-equation-ex1} that
$
({\beta-\Delta u})/(u-c)=-\Delta v/v=-\lambda.
$
This yields that $\{\beta-\Delta u=0\}=\{u=c\}\neq\emptyset$ as long as $c\in Ran (u)$. Hence, $c$ is a generalized inflection value of $u$ provided that $c\in Ran (u)$.
\end{Example}

\begin{Example}\label{ex-maximum-minimum-critical-value}
In the  $\beta$-plane setting, there are  genuine traveling waves whose wave speeds are of the following types:
 (i)  a minimum  of the zonal velocity; (ii) a critical value of the zonal velocity,
(iii)  they lie outside the range of the zonal velocity.
To see this consider, as in Example 3.13 from \cite{csz2024a}, the genuine traveling-wave flow $(u(x-ct,y),v(x-ct,y))$ in $\mathbb{T}_{2\pi}\times[-1,1]$ with stream function
\begin{align*}
\psi(x,y)=\left(\frac{-\beta}{{\pi^2\over 4}+1}-c\right)y+\cos (x)\cos \left(\frac{\pi y}{2}\right).
\end{align*}
A direct computation yields
\begin{align*}
Ran(u)=\left[c+\frac{\beta}{{\pi^2\over 4}+1}-\frac{\pi}{2},c+\frac{\beta}{{\pi^2\over 4}+1}+\frac{\pi}{2}\right].
\end{align*}
Let $\beta_0={\pi\over 2}({\pi^2\over 4}+1)$.

{\rm(i)} If $\beta=\beta_0>0$, then
$u_{\min}=c+\frac{\beta}{{\pi^2\over 4}+1}-\frac{\pi}{2}=c$, so that $c$ is a
 minimum  of $u$.

{\rm(ii)} For $\beta=\beta_0$, we have $u(\pi, 1)=c$ and
$ (\nabla u)(\pi,1)=0$.
Thus, $c$ is a critical value of $u$.

{\rm(iii)}
It is shown in Example 3.13 of \cite{csz2024a} that for $\beta>\beta_0$, $c$ is outside $Ran (u)$, more specifically, $c\in(c^{+}_{\beta},u_{\min})$, where $c^{+}_{\beta}$ is defined in \eqref{def-c+} with $d_{\pm}=\pm d$.
\end{Example}

\begin{Example}\label{ex-beta=0-generalized-inflection-values}
In the  $f$-plane setting, there are  genuine traveling waves whose  wave speeds are  generalized inflection values of the zonal velocity. Indeed, for any $\varepsilon\in\mathbb{R}$, consider the genuine steady solution with velocity field
\begin{align}\label{steady-flow-Kolmogorov}
(u_\varepsilon,v_\varepsilon)=\left(\sin(y)+{\varepsilon\over2}\sin\left({y\over2}\right)\sin\left({\sqrt{3}\over2}x\right),{\sqrt{3}\varepsilon\over2}\cos\left({y\over2}\right)\cos\left({\sqrt{3}\over2}x\right)\right)
\end{align}
in the channel $\mathbb{T}_{L_0}\times [-\pi,\pi]$ with $L_0={4\pi\over \sqrt{3}}$. Since
$$-\Delta u_\varepsilon=u_\varepsilon=\sin(y)+{\varepsilon\over2}\sin\left({y\over2}\right)\sin\left({\sqrt{3}\over2}x\right)\,,$$
we see that $(x,0)$ is a generalized inflection point of $u_\varepsilon$ for any $x\in\mathbb{T}_{L_0}$ and $c=0=u_\varepsilon(x,0)$.
Thus $c$ is a generalized inflection value of $u_\varepsilon$.
\end{Example}

\subsection{Rigidity of traveling waves with arbitrary amplitudes}
We first establish a rigidity result for traveling waves with wave speeds lying in the red intervals  of Fig.\,\ref{fig-beta-plane}. The conclusions differ between the $f$-plane  and  $\beta$-plane settings.

\begin{Theorem}\label{beta=0-cor} Let $\beta\notin Ran(\Delta u)$ and $(u(x-ct,y),v(x-ct,y))\in C^2(D_L)$
be a traveling-wave solution to the $\beta$-plane equation \eqref{Euler equation}-\eqref{boundary condition for euler}. Assume that   one of the following conditions

{\rm(i)} $\beta>0$, $c\notin [c_\beta^+, u_{\min}]$ and $\nabla u\neq0$ on $D_L$,

{\rm(ii)}  $\beta=0$ and  $c\in\mathbb{R}$\\
holds.
  Then $(u(x-ct,y),v(x-ct,y))$ is a shear flow,
  where $c_\beta^{+}$ is defined in \eqref{def-c+} with $d_{\pm}=\pm d$.
\end{Theorem}

\begin{proof} Let $\beta>0$.  By  Theorem
\ref{classification-of-wave-speed-for-a-genuinely-travelling-wave-beta-plane} (i), if $c\notin [c_\beta^+, u_{\min}]\cup\{u_{\max}\}$, then $(u(x-ct,y),v(x-ct,y))$ is a shear flow. It follows from Remark \ref{supplement-to-Theorem-beta=0i} that if
$c=u_{\max}$, then  $(u(x-ct,y),v(x-ct,y))$ is a shear flow.

Let $\beta=0$.  By  Theorem
\ref{classification-of-wave-speed-for-a-genuinely-travelling-wave-beta-plane} (ii),   $(u(x-ct,y),v(x-ct,y))$ is a shear flow for all $c\in\mathbb{R}$.
\end{proof}

\begin{Remark} Recall that in the  $\beta$-plane setting  $[c_\beta^+, u_{\min}]$  is the blue interval of Fig.\,\ref{fig-beta-plane}.  Within this setting,  there may exist traveling waves with speeds in the blue interval   $[c_\beta^+, u_{\min}]$, see Example \ref{ex-maximum-minimum-critical-value}. In general, it is hard to determine whether the rigidity of traveling waves with speeds in the blue interval holds or not.
Nevertheless,
\begin{itemize}
\item for certain specific ranges of $\beta$, in
 Theorem \ref{thm-generalization1} we prove the rigidity of traveling waves with speeds in the blue interval (and thus for arbitrary wave speeds),
  \item for traveling waves near monotone shear flows under  the Rayleigh stability  condition,
 in Theorem \ref{rigidity-near-monotone-shear-flow-arbitrary-wave-speed-thm} we  determine for which $\beta$ the rigidity of traveling waves with speeds  in the blue interval
 holds for all zonal periods $L$, while for other values of $\beta$ there exist genuine traveling waves with a specific zonal  period  $L_0$.
\end{itemize}
\end{Remark}

Next, we consider the rigidity of traveling waves with wave speeds in the  red and blue  intervals  of Fig.\,\ref{fig-beta-plane}.

\begin{Theorem}\label{thm-generalization1}
Assume that one of the following conditions

{\rm(i)} $(\Delta u)_{\min}>0$ and $\beta\in(0,(\Delta u)_{\min})$,

{\rm(ii)} $\Delta u\neq0$ on $D_L$ and $\beta=0$\\
holds. If $(u(x-ct,y),v(x-ct,y))\in C^2(D_L)$ is a traveling-wave
solution to the $\beta$-plane equation \eqref{Euler equation}-\eqref{boundary condition for euler} with $c\in \mathbb{R}$,
then $(u(x-ct,y),v(x-ct,y))$ is a shear flow.
\end{Theorem}

\begin{Remark}
In a bounded periodic channel,
  Theorem $\ref{thm-generalization1}$  offers a generalization of Corollary  $1.6\; ({\rm i})$ in \cite{gxx2024} through a different approach, improving it in two  aspects:
  one extension is from the $f$-plane to the $\beta$-plane setting,
and the other is that the amplitudes of the traveling waves  are arbitrary, with no smallness constraint. In Subsection \ref{cp-b} we present its application to the rigidity of traveling waves near a class of shear flows which
include the Couette-Poiseuille flow and the Bickley jet.
\end{Remark}

\begin{proof}
Since
$(u(x-ct,y),v(x-ct,y))$ is a traveling-wave
solution to the $\beta$-plane equation \eqref{Euler equation}-\eqref{boundary condition for euler}, we know that \eqref{un-vn-eq}-\eqref{un-vn-eq-u-v} hold.
We partition the domain $D_L$ into $S^+, S^0$ and $S^-$ according to \eqref{def-S+-0}.

We first prove the statement under the assumption (i).
Since $\beta\in(0,(\Delta u)_{\min})$, we have
\begin{align}\label{beta-Delta-u-min-arbitrary-speed}
\beta+\pa_y \gamma=\beta-\Delta u\leq\beta-(\Delta u)_{\min}<0
\end{align}
for $(x,y)\in D_L$.
This, along with \eqref{un-vn-eq-u-v}, implies that
\begin{align}\label{v-u-c0-arbitrary-speed}
v=0 \quad \text{on }\quad \{y=\pm d\}\cup S^0.
\end{align}

We now prove that  $v=0$ on $S^-$. We define $F$ on $S^-$ as in \eqref{def-F-beta-plane} and then use the  $F$-formulation  of the $\beta$-plane equation \eqref{F-equation} on $S^-$.
Any $(x_0,y_0)\in S^0\cup\{y=\pm d\}$ being a boundary point of $S^-$, by \eqref{un-vn-eq-u-v}, \eqref{beta-Delta-u-min-arbitrary-speed}-\eqref{v-u-c0-arbitrary-speed} and $(u,v)\in C^2(D_L)$   we have
\begin{align}\label{boundary-term-beta-plane22}
&\lim_{S^-\ni(x,y)\to(x_0,y_0)}((u-c)^2F\partial F)(x,y)=\lim_{S^-\ni(x,y)\to(x_0,y_0)}\left(v\partial v-{v^2\partial u\over u-c}\right)(x,y)\\\nonumber
=&\lim_{S^-\ni(x,y)\to(x_0,y_0)}v\left(\partial v-\partial u{\Delta v\over\Delta u-\beta }\right)(x,y)=0,
\end{align}
where $\partial=\partial_x$ or $\partial_y$.
 Thus, multiplying \eqref{F-equation} by $F$ and integrating by parts, we infer from
   the definition of $S^-$ and $ \beta>0$ that
\begin{align*}
0\leq\int_{S^-}(u-c)^2|\nabla F|^2dxdy=\int_{S^-} \beta(u-c)|F|^2dxdy\leq0.
\end{align*}
Therefore, $F=0$ and $v=0$ on $S^-$.

We now prove that $v=0$ on $S^+$. Unlike the proof of Theorem \ref{wave-speed-inside-range}, the present argument takes advantage of the sign of $\partial_y\gamma +\beta$ in \eqref{beta-Delta-u-min-arbitrary-speed}.
 By \eqref{un-vn-eq-u-v}  the original governing equation of the traveling wave $(u(x-ct,y),v(x-ct,y))$ is
\begin{align}\label{v-formulation-beta-plane-2}
\Delta  v+\frac{ \pa_y\gamma+\beta}{u-c}v=0\quad\text{on}\quad S^+.
\end{align}
 Since $ \pa_y \gamma+\beta<0$ and $u-c>0$ on $S^+$, it follows that $\frac{ \pa_y\gamma+\beta}{u-c}<0$ on $S^+$. Using \eqref{v-u-c0-arbitrary-speed} and the fact that $v\in C^2(D_L)$, we can apply the weak maximum principle for  elliptic  equations (see, for example, Theorem 2.3 in \cite{han-lin}) to both $v$ and $-v$, which immediately yields
\begin{align*}
v= 0
\end{align*}
on any connected component of $S^+$.

By Theorem \ref{beta=0-cor}, $(u(x-ct,y),v(x-ct,y))$ is a shear flow under the assumption (ii).
\end{proof}

\section{Results for an unbounded channel}

We study the fluid flow in the unbounded channel
\[D_L^{\infty}=\{(x,y): x\in\bbT_L,\;y\in(-\infty,+\infty)\}\]
with the asymptotic behavior
\begin{align}\label{7211}
v\to0\quad \text{as}\quad  y\to\pm\infty
\end{align}
for $x\in\mathbb{T}_L$.
Denote
 $$F_{\sup}=\sup_{(x,y)\in D_L^{\infty}}\{F(x,y)\}\,,\quad F_{\inf}=\inf_{(x,y)\in D_L^{\infty}}\{F(x,y)\}\,$$
for a function $F \in C(D_L^{\infty})$.

 We study the propagation speeds and structural properties of traveling waves, extending prior results to an unbounded channel. First, we
 extend Theorems 3.2 and 3.11 in \cite{csz2024a} to an unbounded channel, thus addressing the rigidity issue for wave speeds located outside the range of the zonal velocity.
 \begin{Theorem}\label{beta-plane-unbounded-thm-wave-speeds-outside-range-zonal-velocity}
 Let $(u(x-ct,y),v(x-ct,y))\in C^2(D_L^{\infty})$ be a traveling-wave solution
to the $\beta$-plane equation \eqref{Euler equation}-\eqref{boundary condition for euler}. Assume that $v\partial_yv\to0$ and $v^2\partial_y u\to0$ as $y\to\pm\infty$ for $x\in \mathbb{T}_L$, and one of the following conditions

{\rm(i)}  $\beta>0$ and $c>u_{\sup}$;

{\rm(ii)} $\beta=0$ and $c\notin \overline{Ran(u)}$\\
holds. Then $(u(x-ct,y),v(x-ct,y))$ is a shear flow.
\end{Theorem}

\begin{Remark}
The condition ``$v\partial_yv\to0$ and $v^2\partial_y u\to0$ as $y\to\pm\infty$ for $x\in \mathbb{T}_L$" is mild. In fact,
by \eqref{7211} and the incompressibility condition,
  $u\in\dot{W}^{1,\infty}(D_L^{\infty})=\{u:\partial_xu,\partial_y u\in L^\infty(D_L^\infty)\}$ is sufficient to guarantee this.
\end{Remark}

 Note that in the unbounded case, we do not obtain a lower  bound for the wave speeds of genuine traveling waves in the  $\beta$-plane setting. This is because the unbounded domain prevents the Poincar\'{e} inequality (see (3.29) in \cite{csz2024a})  from being carried out. This can also be seen from $c_\beta^+\to-\infty$ as $d_{\pm}\to\pm \infty$, where $c_\beta^+$ is defined in \eqref{def-c+}.

\begin{proof} Since $c$ does not touch $\overline{Ran(u)}$, we can define $F={v\over u-c}$ on the whole domain $D_{L}^\infty$.
By Lemma \ref{F-formulation-lem}, the $F$-formulation of the $\beta$-plane equation is
\begin{align*}
\partial_{x}((u-c)^2\partial_{x}F)+\partial_{y}((u-c)^2\partial_{y}F)=-\beta(u-c)F\quad \text{on}\quad D_L^\infty.
\end{align*}
Since $v\partial_yv\to0$ and $v^2\partial_y u\to0$ as $y\to\pm\infty$, we have
\begin{align}\label{limit-unbound-channel}
\lim_{y\to\pm\infty}((u-c)^2F\partial_y F)(x,y)
=\lim_{y\to\pm\infty}v\left(\partial_y v-{v\partial_y u\over u-c}\right)(x,y)=0
\end{align}
for $x\in \mathbb{T}_L$. For any $d>0$,  we have
\begin{align}\label{beta-whole-domain-D-T-infty}
&\int_{-d}^d\int_{\mathbb{T}_L}(u-c)^2|\nabla F|^2dxdy-\int_{\mathbb{T}_L}((u-c)^2\partial_yFF)\Big|_{y=-d}^ddx\\\nonumber
=&\int_{-d}^d\int_{\mathbb{T}_L} \beta(u-c)|F|^2dxdy.
\end{align}
If one of (i) and (ii)  holds, then  $\beta(u-c)\leq0$ on $D_L^\infty$. We claim that
$\nabla F=0$ on $D_L^\infty$, which implies $v=0$ on $D_L^\infty$ by \eqref{7211}. Otherwise, there exists $d_0>0$ such that
\begin{align}\label{d0}
&\int_{-d_0}^{d_0}\int_{\mathbb{T}_L}(u-c)^2|\nabla F|^2dxdy>0.
 \end{align}
By \eqref{limit-unbound-channel} we can choose $d_1>d_0$ such that
\begin{align*}
&\int_{\mathbb{T}_L}((u-c)^2\partial_y F F)|_{y=-d_1}^{d_1}dx<{1\over2}\int_{-d_0}^{d_0}\int_{\mathbb{T}_L}(u-c)^2|\nabla F|^2dxdy.
 \end{align*}
Using now \eqref{beta-whole-domain-D-T-infty}-\eqref{d0}, $\beta(u-c)\leq0$ on $D_L^\infty$ and $d_1>d_0$, we obtain
\begin{align}\label{d1}
0\geq&\int_{-d_1}^{d_1}\int_{\mathbb{T}_L} \beta(u-c)|F|^2dxdy\\\nonumber
>&\int_{-d_1}^{d_1}\int_{\mathbb{T}_L}(u-c)^2|\nabla F|^2dxdy-{1\over2}\int_{-d_0}^{d_0}\int_{\mathbb{T}_L}(u-c)^2|\nabla F|^2dxdy\\\nonumber
\geq&{1\over2}\int_{-d_0}^{d_0}\int_{\mathbb{T}_L}(u-c)^2|\nabla F|^2dxdy>0,
\end{align}
which is a contradiction.
\end{proof}

We now extend the classification of the traveling waves to an unbounded channel. We do not intend to get the generalized version as in Theorem \ref{classification-of-wave-speed-for-a-genuinely-travelling-wave-beta-plane}, but focus on some special case, which is applicable to the traveling waves near monotone shear flows in the next section.
\begin{Theorem}\label{beta-plane-unbounded-channel-thm}
Let $(u(x-ct,y),v(x-ct,y))\in C^2(D_L^{\infty})$ be a genuine traveling-wave solution
to the $\beta$-plane equation \eqref{Euler equation}-\eqref{boundary condition for euler}. Assume that $|\partial_{y}u|\geq C_0$ on $D_L^{\infty}$ for some $C_0>0$,
$v\in \dot{W}^{1,\infty}(D_L^{\infty})$,  and ${v^2\partial_y u\over u}\to0$ as $y\to \pm\infty$ for $x\in\mathbb{T}_L$. Then the wave speed $c$ must be a generalized inflection value of $u$.
\end{Theorem}

\begin{Remark} $({\rm i})$
Note that $\lim_{y\to\pm\infty}|u(x,y)|=\infty$   and thus $Ran(u(x,\cdot))=Ran(u)$ spans $\mathbb{R}$ for $x\in\mathbb{T}_L$ in Theorem \ref{beta-plane-unbounded-channel-thm}, a situation that differs
from the case of a bounded channel (for which $Ran(u)$ is compact).

${\rm (ii)}$
By \eqref{7211}, a sufficient condition for the asymptotic behavior  ``${v^2\partial_y u\over u}\to0$  as $y\to \pm\infty$ for $x\in\mathbb{T}_L$" is ${\partial_{y}u\over u}$ is bounded for $|y|$ sufficiently large.
\end{Remark}
\begin{proof}
The proof essentially follows the same approach as in Theorem \ref{wave-speed-inside-range}. We highlight only a  difference, primarily concerning the asymptotic behavior as $y\to\pm\infty$.

Assume that $c$ is not a generalized inflection value of $u$. We will show that $(u(x-ct,y),v(x-ct,y))$ is a shear flow.
 Note that $Ran (u)=Ran(u(x,\cdot))=\mathbb{R}$  for $x\in\mathbb{T}_L$, due to $|\partial_{y}u|\geq C_0$ on $D_L^{\infty}$ for some $C_0>0$.
 Then there exists a unique $y_x\in \mathbb{R}$ such that $u(x,y_x)=c$ for  $x\in\mathbb{T}_L$. This, combined with \eqref{gradient-partial-x-u}, implies that
 $$
 S^0=\{(x,y_x):x\in\mathbb{T}_L\}
 $$
  is a straight segment parallel to the zonal direction.
  The subset $S^-$ lies on one side of the curve $S^0$, while  $S^+$ occupies the other side.
The difference is that the calculation of the boundary term in \eqref{boundary-term2} should be replaced by the asymptotic behavior
\begin{align*}
\lim_{y\to\pm\infty}((u-c)^2F\partial_y F)(x,y)
=\lim_{y\to\pm\infty}v\left(\partial_y v-{v\partial_y u\over u-c}\right)(x,y)=0
\end{align*}
for $x\in \mathbb{T}_L$, due to \eqref{7211}, $\partial_y v\in L^{\infty}(D_L^\infty)$ and ${v^2\partial_y u\over u}\to0$ as $y\to \pm\infty$.
Now $v=0$ on $S^-$ follows from an argument similar to that in \eqref{beta-whole-domain-D-T-infty}-\eqref{d1}.
\end{proof}
 Finally, we
 extend Theorem \ref{thm-generalization1} to the unbounded channel, providing sufficient conditions for the rigidity of traveling waves with arbitrary wave speeds.

\begin{Theorem}\label{beta-plane-unbounded-channel-thm-rigidity}
Let $(u(x-ct,y),v(x-ct,y))\in C^2(D_L^{\infty})$ be a traveling-wave solution
to the $\beta$-plane equation \eqref{Euler equation}-\eqref{boundary condition for euler}. Assume that
$v\in \dot{W}^{2,\infty}(D_L^{\infty})=\{v:\partial_xv,\partial_y v, \Delta v\in L^\infty(D_L^\infty)\}$, $v\partial_yu\to0$ as $y\to \pm\infty$ for $x\in\mathbb{T}_L$,  and one of the
following conditions

{\rm(i)} $(\Delta u)_{\inf}>0$, $\beta\in(0,(\Delta u)_{\inf})$;

{\rm(ii)} $0=\beta\notin \overline{Ran(\Delta u)}$\\
holds. Then $(u(x-ct,y),v(x-ct,y))$ is a shear flow for an arbitrary $c\in\mathbb{R}$.
\end{Theorem}

\begin{Remark}
By \eqref{7211},
$\partial_{y}u\in L^{\infty}(D_L^{\infty})$ is sufficient to ensure $v\partial_yu\to0$   as $y\to \pm\infty$ for $x\in\mathbb{T}_L$.
\end{Remark}
\begin{proof}
Under the assumption (i), we follow the approach used in  Theorem \ref{thm-generalization1} (i). We only discuss the differences.
For
$(x_0,y_0)\in S^0$ a boundary point of $S^-$, we still have the calculation of the boundary term \eqref{boundary-term-beta-plane22}.
As $y\to\pm\infty$, instead of \eqref{boundary-term-beta-plane22} for the boundary points in $\{y=\pm d\}$, we need the following asymptotic behavior
\begin{align}\label{asymptotic-behavior-unbounded-case2}
\lim_{y\to\pm\infty}((u-c)^2F\partial_y F)(x,y)
=\lim_{y\to\pm\infty}v\left(\partial_y v-\partial_y u{\Delta v\over\Delta u-\beta }\right)(x,y)=0
\end{align}
for  $x\in \mathbb{T}_L$, due to \eqref{7211}, $\Delta u-\beta\geq (\Delta u)_{\inf}-\beta>0$ for $(x,y)\in D_L^\infty$, $v\in \dot{W}^{2,\infty}(D_L^\infty)$, and $v\partial_yu\to0$ as $y\to \pm\infty$. A second difference is that in the proof of $v=0$ on $S^+$, we multiply \eqref{v-formulation-beta-plane-2} by $v$ and integrate by parts to get
\begin{align*}
\lim_{d\to\infty}\int_{S^+\cap[-d,d]}\left(-|\nabla v|^2+{\partial_y \gamma +\beta\over u-c} |v|^2\right)dxdy=0.
\end{align*}
Here, we use $\lim\limits_{S^+\ni(x,y)\to(x_0,y_0)} (v\partial v)(x,y)=0$ for $(x_0,y_0)\in S^0$ a boundary point of $S^+$, and $\lim\limits_{y\to\pm\infty} (v\partial_y v)(x,y)=0$ for $x\in\mathbb{T}_L$, due to \eqref{7211} and  $v\in\dot{W}^{2,\infty}(D_L^\infty)$. This, along with $\partial_y\gamma +\beta<0$ and $u>c$ on $S^+$, implies $v=0$ on $S^+$.

Under the assumption (ii), we use the method in Theorem \ref{wave-speed-inside-range} (ii). The main difference is that, instead of \eqref{boundary-term2}, we use the asymptotic behavior
\eqref{asymptotic-behavior-unbounded-case2} as $y\to\pm\infty$ for $x\in\mathbb{T}_L$.
\end{proof}

\section{Applications to  monotone shear flows}\label{monotone-shear-flows}
In this section, based on the theoretical framework developed in Sections 3-4, we present a precise classification of the traveling waves near a monotone shear flow.
 In an  unbounded channel,
 we get a simple description of the propagation speeds of nearby traveling waves in Theorem \ref{Monotone-shear-flows-on-an-unbounded-channel-thm}.
In the bounded channel setting, the situation becomes complicated and we present a comprehensive analysis of nearby traveling waves. First,
we obtain the rigidity of nearby traveling waves with
 wave speeds lying inside the range of the monotone shear flow.
 Under the Rayleigh stability condition, we provide a complete characterization of the
$(\beta,L)$-parameter regimes that result in either the rigidity of traveling waves for arbitrary wave speeds, or the emergence of genuine unidirectional traveling waves that are close to the monotone shear flow.
 In the absence of Rayleigh stability condition, we identify a parameter regime for $(\beta,L)$ that allows for the existence of genuine unidirectional traveling waves in a neighborhood of the monotone shear flow. Outside this parameter regime, the wave speed of any genuine nearby traveling wave must be a generalized inflection value of the zonal velocity,
 see Theorem \ref{rigidity-near-monotone-shear-flow-arbitrary-wave-speed-thm} and Remark \ref{rigidity-near-monotone-shear-flow-arbitrary-wave-speed-thm-rem}. Our results hold for general monotone shear flows and in the special case
 of the Couette flow we improve Theorem 1.1 and Theorem 1.3 in \cite{WZZ}; see Remark \ref{monotone-shear-flow-bounded-channel-thm-rem} and Corollary \ref{8181}.

Under the Rayleigh stability condition, while monotone shear flows are isolated in an unbounded channel, in a bounded channel genuine  traveling waves arise near such shear flows for a large $(\beta, L)$-parameter regime.  Moreover, for $(\beta,L)$ in this regime,
the monotone shear flow is nonlinearly Lyapunov stable yet  asymptotically  unstable.

\subsection{Monotone shear flows in an unbounded channel} We consider
monotone shear flows $(u_0,0)$ satisfying $|u_0'|\geq C_0$ on $\mathbb{R}$ for some $C_0>0$, a typical example being the Couette flow.
A
characteristic of such shear flows is that
the  range of zonal velocity spans the entire real line, which allows us to deal with  arbitrary wave speeds.

\begin{Theorem}\label{Monotone-shear-flows-on-an-unbounded-channel-thm}
Assume that $u_0\in C^2(\mathbb{R})$, and there exist $C_0, C_1>0$ such that  $|u_0'|\geq C_0$ on $\mathbb{R}$
and $\left|{u_0'\over u_0}\right|\leq C_1$ for $|y|$ sufficiently large. Let $\varepsilon_0\in(0,C_0)$. Then
the wave speed of any genuine traveling-wave solution $(u(x-ct,y),v(x-ct,y))\in C^2(D_L^{\infty})$ to the $\beta$-plane equation
\eqref{Euler equation}-\eqref{boundary condition for euler} satisfying $v\in\dot{W}^{1,\infty}(D_L^{\infty})$ and
\begin{align*}
\|u-u_{0}\|_{W^{1,\infty}(D_L^{\infty})}<\varepsilon_0
\end{align*}
must be  a generalized inflection value of $u$.
\end{Theorem}
\begin{Remark}\label{rem-rigidity-unbounded-channel}
We point out an interesting consequence of Theorem \ref{Monotone-shear-flows-on-an-unbounded-channel-thm}:
under the
assumption of Theorem \ref{Monotone-shear-flows-on-an-unbounded-channel-thm}, if $\beta\notin \overline{Ran(u_0'')}$ and $\varepsilon_0$ is taken smaller such that
\begin{align*}
0<\varepsilon_0<\min\left\{C_0,\left| (u_0'')_{\inf}-\beta\right|,\left| (u_0'')_{\sup}-\beta\right|\right\},
\end{align*}
then any traveling-wave solution $(u(x-ct,y),v(x-ct,y))\in C^2(D_L^{\infty})$ to the $\beta$-plane equation
\eqref{Euler equation}-\eqref{boundary condition for euler} satisfying $v\in\dot{W}^{1,\infty}(D_L^{\infty})$ and
\begin{align*}
\|u-u_{0}\|_{W^{2,\infty}(D_L^{\infty})}<\varepsilon_0
\end{align*}
must be  a shear flow for arbitrary $c\in\mathbb{R}$.

In fact, since $\beta\notin \overline{Ran(u_0'')}$, we have
\begin{align*}
|\Delta u-\beta|\geq|u_0''-\beta|-|\Delta u-u_0''|\geq\min\left\{\left| (u_0'')_{\inf}-\beta\right|,\left|\beta- (u_0'')_{\sup}\right|\right\}-\varepsilon_0>0
\end{align*}
for $(x,y)\in D_L^\infty$,
which implies that $\beta\notin Ran(\Delta u)$.
By Theorem \ref{Monotone-shear-flows-on-an-unbounded-channel-thm}, $(u(x-ct,y),v(x-ct,y))$ is a shear flow.

This shows that for a general monotone shear flow $(u_0,0)$ in an unbounded channel, if $\beta\notin \overline{Ran(u_0'')}$, genuine nearby traveling waves are absent.
This is
  a necessary condition for asymptotic stability (i.e. nonlinear inviscid damping) near a shear flow  in a specific space.
  Taking for example the Couette flow,
it is shown in \cite{Bedrossian-Masmoudi2015} that nonlinear inviscid damping near it holds  in a certain Gevrey space for $\beta=0$.
This result was extended to the case $\beta>0$ in \cite{WZZ}. Applying  the above consequence to the Couette flow,  we conclude that
there are no  genuine  traveling waves $W^{2,\infty}(D_L^{\infty})$-close (in velocity) to $(y,0)$ for $\beta>0$. This is consistent with Theorem 1.7 in \cite{WZZ}.
\end{Remark}
\begin{proof}
Since $\varepsilon_0\in(0,C_0)$ and $|u_0'|\geq C_0$ on $\mathbb{R}$, we have
\begin{align}\label{7223}
|\partial_{y}u|\geq|u_0'|-|u_0'-\partial_{y}u|>C_0-\varepsilon_0>0.
\end{align}
Moreover,
\begin{align}\label{7224}
\left|{\partial_y u\over u}\right|\leq\left|{u_0'\over u_0}-{\partial_y u\over u}\right|+\left|{u_0'\over u_0}\right|\leq\left|{ u_0'-\partial_y u\over u}\right|+\left|{u_0'\over u_0}\cdot{ u-u_0\over u}\right|+\left|{u_0'\over u_0}\right|
\leq \varepsilon_0+ C_1\varepsilon_0+C_1
\end{align}
for $x\in\mathbb{T}_L$ and $|y|$ sufficiently large, in which case $|u|>1$.
By \eqref{7211}, \eqref{7223} and \eqref{7224},
all the assumptions in Theorem \ref{beta-plane-unbounded-channel-thm}  hold for $(u,v)$.
Then $c$ is a generalized inflection value of $u$.
\end{proof}

\subsection{Monotone shear flows in a bounded channel}
We first establish a rigidity result for traveling waves near a monotone shear flow $(u_0,0)$ with wave speeds inside $Ran (u_0)$

\begin{Theorem}\label{monotone-shear-flow-bounded-channel-thm}
Let $u_0\in C^2([-d,d])$  and $\delta\in(0,{u_{0,\max}-u_{0,\min}\over2})$.
Assume that $|u_0'|\geq C_0$ on $[-d,d]$ for some $C_0>0$,
$\beta\notin Ran(u_0'')$, and $c\in[u_{0,\min}+\delta,u_{0,\max}-\delta]$.
Let
\begin{align*}
0<\varepsilon_\delta<\min\{\delta,C_0,|(u_0'')_{\min}-\beta|,|(u_0'')_{\max}-\beta|\}.
\end{align*}
Then any traveling-wave solution $(u(x-ct,y),v(x-ct,y))\in C^2(D_L)$ to the $\beta$-plane equation
\eqref{Euler equation}-\eqref{boundary condition for euler} satisfying
\begin{align}\label{7222-beta-plane-bounded}
\|(u,v)-(u_{0},0)\|_{C^{2}(D_L)}<\varepsilon_\delta
\end{align}
must be a shear flow.
\end{Theorem}
\begin{Remark}\label{monotone-shear-flow-bounded-channel-thm-rem}
Theorem \ref{monotone-shear-flow-bounded-channel-thm} extends Theorem 1.1 in \cite{WZZ} from the Couette flow setting (in which case $Ran(u_0'')=\{0\}$) to  general monotone shear flows, while simultaneously improving the regularity assumption from $H^{\geq 5}$ to $C^2$.
\end{Remark}
\begin{proof}
Since $\beta\notin Ran(u_0'')$, by \eqref{7222-beta-plane-bounded} and the choice of $\varepsilon_\delta$
we have $\beta\notin Ran(\Delta u)$.
Moreover,
$
|\partial_{y}u|\geq C_0-\varepsilon_\delta>0$ for $(x,y)\in D_L.
$
Let $(x_1,y_1),(x_2,y_2)\in D_L$ such that
\[u_{0}(x_1,y_1)=u_{0,\min}\quad\text{and}\quad u_{0}(x_2,y_2)=u_{0,\max}.\]
Since $c\in[u_{0,\min}+\delta,u_{0,\max}-\delta]$, by \eqref{7222-beta-plane-bounded} we have
\begin{align*}
c-u_{\min}\geq c-u(x_1,y_1)= c-u_0(x_1,y_1)+u_0(x_1,y_1)-u(x_1,y_1)\geq\delta-\varepsilon_\delta>0,
\end{align*}
and
\begin{align*}
c-u_{\max}\leq c-u(x_2,y_2)= c-u_0(x_2,y_2)+u_0(x_2,y_2)-u(x_2,y_2)\leq-\delta+\varepsilon_\delta<0.
\end{align*}
Thus $c\in(u_{\min},u_{\max})$ and all assumptions in Theorem \ref{wave-speed-inside-range} hold for $(u,v)$.
Therefore $(u(x-ct,y),v(x-ct,y))$ is a shear flow.
\end{proof}

Regarding Theorem \ref{monotone-shear-flow-bounded-channel-thm},
 two natural questions arise: the first is whether there are genuine nearby traveling waves outside the range $[u_{0,\min}+\delta,u_{0,\max}-\delta]$ for $\beta\notin Ran(u_0'')$, the other regards specific features of traveling waves for $\beta\in Ran(u_0'')$.  We will
 provide answers in Theorem \ref{rigidity-near-monotone-shear-flow-arbitrary-wave-speed-thm}. To proceed, we first discuss some relevant spectral aspects.
 The principal eigenvalue of the singular Rayleigh-Kuo  boundary value problem
\begin{align}\label{sturm-Liouville}
-\phi''-{\beta-u_0''\over u_0-u_{0,\min}}\phi=\lambda\phi, \;\;\;\;\phi(\pm d)=0
\end{align}
with $\phi\in H_0^1\cap H^2(-d,d)$, denoted by $\lambda_1(\beta,u_{0,\min})$, is
\begin{align}\label{variation2}
\lambda_1(\beta,u_{0,\min})
=&\inf_{\phi\in H_0^1, \|\phi\|_{L^2}=1}\int_{-d}^{d}\left(|\phi'|^2-{\beta-u_0''\over u_0-u_{0,\min}}|\phi|^2\right)dy.
\end{align}
By the compactness of the embedding $H_0^1(-d,d)\hookrightarrow L^2(-d,d)$ and since $|u_0'|\geq C_0$ on $[-d,d]$ for some $C_0>0$,
$\lambda_1(\beta,u_{0,\min})$ is well-defined and  the
infimum in \eqref{variation2} is attained by a corresponding eigenfunction.
The variational characterization \eqref{variation2} shows that
\begin{align}\label{property-lambda-1}
\lambda_1(\cdot,u_{0,\min})\text{ is decreasing  on }\mathbb{R} \text{ and }\lim_{\beta\to\infty}\lambda_1(\beta,u_{0,\min})=-\infty.
 \end{align}
 For $c<u_{0,\min}$ we denote by $\lambda_1(\beta,c)$ the principal eigenvalue of the regular Rayleigh-Kuo  boundary value problem \eqref{sturm-Liouville} with $u_{0,\min}$ replaced by $c$.
 When we need to specify the dependence of $\lambda_1(\beta,c)$ on the shear-flow profile and on the zonal-band width, we denote it by $\lambda_1(\beta,c;u_0)$ or $\lambda_1(\beta,c;u_0,d)$, for $c\leq u_{0,\min}$.

We first prove the continuity of $\lambda_1(\cdot,u_{0,\min})$.

\begin{Lemma}\label{lambda1continuous}
Assume that $u_0\in C^2([-d,d])$ and $u_0'\neq0$ on $[-d,d]$. Then $\lambda_1(\cdot,u_{0,\min})$ is continuous on $\mathbb{R}$.
\end{Lemma}

\begin{proof}
Without loss of generality, we may assume that $u_{0}'\geq C_0$ on $[-d,d]$ for some $C_0>0$.
Fix $\beta_0\in\mathbb{R}$.
First, we prove that
\begin{align}\label{8131}
\lim_{\beta\to\beta_{0}^{-}}\lambda_1(\beta,u_{0,\min})=\lambda_1(\beta_{0},u_{0,\min}).
 \end{align}
Let $\beta<\beta_0$ and $\phi_{\beta_{0}}\in H_0^1(-d,d)$ with $\|\phi_{\beta_{0}}\|_{L^2(-d,d)}=1$ such that
\begin{align}\label{8132}
\lambda_1(\beta_{0},u_{0,\min})
=&\int_{-d}^{d}\left(|\phi'_{\beta_{0}}|^2-{\beta_{0}-u_0''\over u_0-u_{0,\min}}|\phi_{\beta_{0}}|^2\right)dy.
\end{align}
By \eqref{variation2} we have
\begin{align}\label{8133}
\lambda_1(\beta,u_{0,\min})
\leq&\int_{-d}^{d}\left(|\phi'_{\beta_{0}}|^2-{\beta-u_0''\over u_0-u_{0,\min}}|\phi_{\beta_{0}}|^2\right)dy.
\end{align}
Since $u_{0}'\geq C_0>0$, we get
\begin{align}\label{8134}
u_{0}(y)-u_{0,\min}=u_{0}(y)-u_{0}(-d)\geq C_0(y+d)\geq0
\end{align}
for $y\in[-d,d]$.
By \eqref{property-lambda-1}, \eqref{8132}-\eqref{8134} and the Hardy inequality \eqref{Hardy-inequality}, we obtain
\begin{align*}
0&<\lambda_1(\beta,u_{0,\min})-\lambda_1(\beta_{0},u_{0,\min})
\leq\int_{-d}^{d}{\beta_0-\beta\over u_0-u_{0,\min}}|\phi_{\beta_{0}}|^2dy\\\nonumber
&\leq\int_{-d}^{d}{\beta_0-\beta\over C_0(y+d)}|\phi_{\beta_{0}}(y)|^2dy
\leq{\beta_0-\beta\over C_0}\left(\int_{-d}^{d}\left|{\phi_{\beta_{0}}(y)\over y+d}\right|^{2}dy\right)^{\frac{1}{2}}\|\phi_{\beta_{0}}\|_{L^2(-d,d)}\\\nonumber
&\leq C(\beta_0-\beta)\|\phi'_{\beta_{0}}\|_{L^2(-d,d)}.
\end{align*}
This implies that $\lambda_1(\beta,u_{0,\min})\to\lambda_1(\beta_{0},u_{0,\min})$ as $\beta\to\beta_{0}^{-}$.

Next, we show that
\begin{align}\label{8135}
\lim_{\beta\to\beta_{0}^{+}}\lambda_1(\beta,u_{0,\min})=\lambda_1(\beta_{0},u_{0,\min}).
 \end{align}
Let $\beta>\beta_0$ and  $\phi_{\beta}\in H_0^1(-d,d)$ with $\|\phi_{\beta}\|_{L^2(-d,d)}=1$ such that
\begin{align}\label{8136}
\lambda_1(\beta,u_{0,\min})
=&\int_{-d}^{d}\left(|\phi'_{\beta}|^2-{\beta-u_0''\over u_0-u_{0,\min}}|\phi_{\beta}|^2\right)dy.
\end{align}
Then by \eqref{variation2} we have
\begin{align*}
\lambda_1(\beta_0,u_{0,\min})
\leq&\int_{-d}^{d}\left(|\phi'_{\beta}|^2-{\beta_0-u_0''\over u_0-u_{0,\min}}|\phi_{\beta}|^2\right)dy.
\end{align*}
Thus
\begin{align}\label{8138}
0&>\lambda_1(\beta,u_{0,\min})-\lambda_1(\beta_{0},u_{0,\min})
\geq\int_{-d}^{d}{\beta_0-\beta\over u_0-u_{0,\min}}|\phi_{\beta}|^2dy\\\nonumber
&\geq\int_{-d}^{d}{\beta_0-\beta\over C_0(y+d)}|\phi_{\beta}(y)|^2dy
\geq{\beta_0-\beta\over C_0}\left(\int_{-d}^{d}\left|{\phi_{\beta}(y)\over y+d}\right|^{2}dy\right)^{\frac{1}{2}}\|\phi_{\beta}\|_{L^2(-d,d)}\\\nonumber
&\geq C(\beta_0-\beta)\|\phi'_{\beta}\|_{L^2(-d,d)}.
\end{align}
We claim that $\|\phi'_{\beta}\|_{L^2(-d,d)}, \beta\in[\beta_{0},\beta_0+1]$ is uniformly bounded.
In fact, by \eqref{property-lambda-1} we have
\begin{align}\label{8139}
\lambda_1(\beta_0+1,u_{0,\min})\leq\lambda_1(\beta,u_{0,\min})\leq\lambda_1(\beta_0,u_{0,\min})),\qquad \beta\in[\beta_{0},\beta_0+1].
\end{align}
From \eqref{8136}, \eqref{8139}, \eqref{8134} and the fact that $\|\phi_{\beta}\|_{L^2(-d,d)}=1$ we get
\begin{align*}
\|\phi'_{\beta}\|_{L^2(-d,d)}^2&=\lambda_1(\beta,u_{0,\min})\|\phi_{\beta}\|_{L^2(-d,d)}^{2}
+\int_{-d}^{d}{\beta-u''_{0}\over u_0-u_{0,\min}}|\phi_{\beta}|^2dy\\\nonumber
&\leq C\lambda_1(\beta_0,u_{0,\min})\|\phi_{\beta}\|_{L^2(-d,d)}\|\phi'_{\beta}\|_{L^2(-d,d)}
+C\int_{-d}^{d}{1\over C_0(y+d)}|\phi_{\beta}(y)|^2dy\\\nonumber
&\leq C\lambda_1(\beta_0,u_{0,\min})\|\phi'_{\beta}\|_{L^2(-d,d)}
+C\left(\int_{-d}^{d}\left|{\phi_{\beta}(y)\over y+d}\right|^2dy\right)^{{1\over2}}\\\nonumber
&\leq C\lambda_1(\beta_0,u_{0,\min})\|\phi'_{\beta}\|_{L^2(-d,d)}
+C\|\phi'_{\beta}\|_{L^2(-d,d)}.\nonumber
\end{align*}
This yields
\begin{align}\label{8141}
\|\phi'_{\beta}\|_{L^2(-d,d)}\leq C\lambda_1(\beta_0,u_{0,\min})+C.
\end{align}
It follows from \eqref{8138} and \eqref{8141} that
\begin{align*}
0>\lambda_1(\beta,u_{0,\min})-\lambda_1(\beta_{0},u_{0,\min})&\geq C(\beta_0-\beta)\|\phi'_{\beta}\|_{L^2(-d,d)}\\\nonumber
&\geq C(\beta_0-\beta)(C\lambda_1(\beta_0,u_{0,\min})+C)\,,
\end{align*}
which implies $\lambda_1(\beta,u_{0,\min})\to\lambda_1(\beta_{0},u_{0,\min})$ as $\beta\to\beta_{0}^{+}$.

Combining \eqref{8131} and \eqref{8135} proves that $\lambda_1(\cdot,u_{0,\min})$ is continuous on $\mathbb{R}$.
\end{proof}

We now establish the continuous dependence of the principal eigenvalue $\lambda_1(\beta,u_{0,\min}; u_0)$ on the shear-flow profile in the $C^2$-topology.

\begin{Lemma}\label{lambda1continuous-profile-lem}
Assume that $u_0\in C^2([-d,d])$ and $u_0'\neq0$ on $[-d,d]$. Then
\begin{align}\label{lambda1continuous-profile}
\lambda_1(\beta,\tilde u_{0,\min}; \tilde u_0)\to\lambda_1(\beta,u_{0,\min}; u_0)\qquad \text{as}\quad\tilde{u}_0 \to u_{0}\;\; \text{in}\;\; C^2([-d,d]).
\end{align}
\end{Lemma}
\begin{proof}
Without loss of generality, we may assume that $u_0'>0$ on $[-d,d]$. Then there exist $C_2>C_1>0$ such that
\begin{align*}
C_1(y+d)\leq \tilde u_0(y)-\tilde u_{0,\min}=\tilde u_0(y)-\tilde u_0(-d)\leq C_2 (y+d),\qquad y\in[-d,d]
\end{align*}
for $\tilde u_0$ sufficiently close to $u_0$ in $C^2([-d,d])$,
which implies \begin{align*}
\left|{\beta-\tilde u_0''\over \tilde u_0-\tilde u_{0,\min}}\right|\leq {C\over y+d},\qquad y\in[-d,d].
\end{align*}
Here, the constants $C_1, C_2, C$  are independent of the choice of $\tilde u$.
For any $\phi\in H_0^1(-d,d)$ with $\|\phi\|_{L^2(-d,d)}=1$, we have
\begin{align}\label{lambda1continuous-profile-eigenvalue-bound}
\int_{-d}^{d}\left(|\phi'|^2-{C\over y+d}|\phi|^2\right)dy\leq&\int_{-d}^{d}\left(|\phi'|^2-{\beta-\tilde u_0''\over\tilde u_0-\tilde u_{0,\min}}|\phi|^2\right)dy\\\nonumber
\leq& \int_{-d}^{d}\left(|\phi'|^2+{C\over y+d}|\phi|^2\right)dy
\end{align}
provided that $\tilde u_0$ is sufficiently close to $u_0$ in $C^2([-d,d])$.
Taking the infimum over all $\phi\in H_0^1(-d,d)$ with $\|\phi\|_{L^2(-d,d)}=1$, we obtain
\begin{align*}
\lambda_1(C,-d;y)\leq\lambda_1(\beta,\tilde u_{0,\min}; \tilde u_0)\leq\lambda_1(-C,-d;y)
\end{align*}
uniformly for $\tilde u_0$ sufficiently close to $u_0$ in $C^2([-d,d])$. Let $\tilde \phi_0\in H_0^1(-d,d)$  be an eigenfunction of $\lambda_1(\beta,\tilde u_{0,\min}; \tilde u_0)$ with $\|\tilde \phi_0\|_{L^2(-d,d)}=1$. By the
variational characterization of  $\lambda_1(\beta,\tilde u_{0,\min}; \tilde u_0)$ and by the Poincar\'{e} and Hardy inequalities, we have
\begin{align*}
\|\tilde \phi'_{0}\|_{L^2(-d,d)}^2
&\leq C
\max\{|\lambda_1(C,-d;y)|,|\lambda_1(-C,-d;y)|\}
\|\tilde\phi'_{0}\|_{L^2(-d,d)}
+\int_{-d}^{d}{C\over y+d}|\tilde \phi_{0}|^2dy\\\nonumber
&\leq C(
\max\{|\lambda_1(C,-d;y)|,|\lambda_1(-C,-d;y)|\}+1)
\|\tilde\phi'_{0}\|_{L^2(-d,d)},
\end{align*}
which yields
\begin{align}\label{uni-bound-tilde-phi-prime}
\|\tilde \phi'_{0}\|_{L^2(-d,d)}
&\leq C(
\max\{|\lambda_1(C,-d;y)|,|\lambda_1(-C,-d;y)|\}+1)\triangleq M_0
\end{align}
uniformly for $\tilde u_0$ sufficiently close to $u_0$ in $C^2([-d,d])$. If $ \phi_0\in H_0^1(-d,d)$ is an eigenfunction of $\lambda_1(\beta, u_{0,\min}; u_0)$ with $\| \phi_0\|_{L^2(-d,d)}=1$, then, by \eqref{uni-bound-tilde-phi-prime}, we have
\begin{align*}
0\longleftarrow&-C \|\tilde u_0-u_0\|_{C^2([-d,d])}M_0^2\\
\leq&-C \|\tilde u_0-u_0\|_{C^2([-d,d])}\|\tilde\phi_0'\|_{L^2(-d,d)}^2\\
\leq&\int_{-d}^d{(\tilde u_0''-u_0'')(u_0-u_{0,\min})+(\beta-u_0'')(\tilde u_0-u_0-(\tilde u_{0,\min}-u_{0,\min}))\over (\tilde u_0-\tilde u_{0,\min})( u_0-u_{0,\min})}|\tilde\phi_0|^2 dy\\
=&\int_{-d}^d\left(|\tilde\phi_0'|^2-{\beta-\tilde u_0''\over\tilde u_0-\tilde u_{0,\min}}|\tilde\phi_0|^2\right)dy-\int_{-d}^d\left(|\tilde\phi_0'|^2-{\beta- u_0''\over u_0- u_{0,\min}}|\tilde\phi_0|^2\right)dy\\
\leq&\lambda_1(\beta,\tilde u_{0,\min}; \tilde u_0)-\lambda_1(\beta,u_{0,\min}; u_0)\\
\leq&\int_{-d}^d\left(|\phi_0'|^2-{\beta-\tilde u_0''\over\tilde u_0-\tilde u_{0,\min}}|\phi_0|^2\right)dy-\int_{-d}^d\left(|\phi_0'|^2-{\beta- u_0''\over u_0- u_{0,\min}}|\phi_0|^2\right)dy\\
=&\int_{-d}^d{-(\beta-\tilde u_0'')(u_0-u_{0,\min})+(\beta-u_0'')(\tilde u_0-\tilde u_{0,\min})\over (\tilde u_0-\tilde u_{0,\min})( u_0-u_{0,\min})}|\phi_0|^2 dy\\
=&\int_{-d}^d{(\tilde u_0''-u_0'')(u_0-u_{0,\min})+(\beta-u_0'')(\tilde u_0-u_0-(\tilde u_{0,\min}-u_{0,\min}))\over (\tilde u_0-\tilde u_{0,\min})( u_0-u_{0,\min})}|\phi_0|^2 dy\\
\leq&C \|\tilde u_0-u_0\|_{C^2([-d,d])}\int_{-d}^d{1\over (y+d)^2}|\phi_0|^2 dy\\
\leq&C \|\tilde u_0-u_0\|_{C^2([-d,d])}\|\phi_0'\|_{L^2(-d,d)}^2\longrightarrow0
\end{align*}
since $\tilde u_0\to u_0$ in $C^2([-d,d])$. This proves \eqref{lambda1continuous-profile}.
\end{proof}

It is useful to determine the limits of the principal eigenvalue $\lambda_1(\beta,c)$ as $c\to -\infty$ and as $c$ approaches the left endpoint of the range of $u_0$.

\begin{Lemma}\label{c-tou0min-limit-lem}
Assume that $u_0\in C^2([-d,d])$ and $u_0'\neq0$ on $[-d,d]$. Then
\begin{align}\label{c-tou0min-limit}
\lim_{c\to u_{0,\min}^-}\lambda_1(\beta,c)=\lambda_1(\beta,u_{0,\min})\quad\text{and}\quad \lim_{c\to -\infty}\lambda_1(\beta,c)={\pi^2\over4d^2}>0.
\end{align}
\end{Lemma}
\begin{proof}
Since the second limit in \eqref{c-tou0min-limit} is straightforward, we only provide a proof for the first.
Assume that $u_0'>0$ on $[-d,d]$.
For $c\in(-\infty,u_{0,\min}]$, we  have
$$u_0(y)-c=u_0(y)-u_{0,\min}+u_{0,\min}-c\geq u_0(y)-u_{0,\min}= u_0(y)-u_{0}(-d)\geq C(y+d)\,,\quad y\in[-d,d]\,.$$
It follows that
\begin{align*}
\left|{\beta- u_0''\over  u_0-c}\right|\leq {C\over y+d},\qquad y\in[-d,d],
\end{align*}
where $C>0$ is independent of $c\in(-\infty,u_{0,\min}]$.
By the same argument as in \eqref{lambda1continuous-profile-eigenvalue-bound}, we obtain the uniform eigenvalue-bound
\begin{align*}
\lambda_1(C,-d;y)\leq\lambda_1(\beta,c)\leq\lambda_1(-C,-d;y),\qquad c\in(-\infty,u_{0,\min}].
\end{align*}
This, in turn,  yields a uniform eigenfunction-bound
\begin{align*}
\| \phi'_{c}\|_{L^2(-d,d)}
&\leq M_0,\qquad c\in (-\infty,u_{0,\min}],
\end{align*}
where $\phi_c\in H_0^1(-d,d)$  is an eigenfunction of $\lambda_1(\beta,c)$ normalized by $\| \phi_c\|_{L^2(-d,d)}=1$, and $M_0$ is defined in \eqref{uni-bound-tilde-phi-prime}.
Then
\begin{align*}
0\longleftarrow&C(c-u_{0,\min})M_0^2\\
\leq&C(c-u_{0,\min})\|\phi_{c}'\|_{L^2(-d,d)}^2\\
\leq&C(c-u_{0,\min})\int_{-d}^d{1\over(y+d)^2}|\phi_{c}|^2 dy\\
\leq&\int_{-d}^d\left( |\phi_{c}'|^2-{\beta-u_0''\over u_0-c}|\phi_{c}|^2\right)dy
-\int_{-d}^d\left( |\phi_{c}'|^2-{\beta-u_0''\over u_0-u_{0,\min}}|\phi_{c}|^2\right)dy\\
\leq&\lambda_1(\beta,c)-\lambda_1(\beta,u_{0,\min})\\
\leq&\int_{-d}^d\left( |\phi_{u_{0,\min}}'|^2-{\beta-u_0''\over u_0-c}|\phi_{u_{0,\min}}|^2\right)dy
-\int_{-d}^d\left( |\phi_{u_{0,\min}}'|^2-{\beta-u_0''\over u_0-u_{0,\min}}|\phi_{u_{0,\min}}|^2\right)dy\\
=&(u_{0,\min}-c)\int_{-d}^d{\beta-u_0''\over (u_0-c)(u_0-u_{0,\min})}|\phi_{u_{0,\min}}|^2 dy\\
\leq&C(u_{0,\min}-c)\int_{-d}^d{1\over(y+d)^2}|\phi_{u_{0,\min}}|^2 dy\\
\leq&C(u_{0,\min}-c)\|\phi_{u_{0,\min}}'\|_{L^2(-d,d)}^2\longrightarrow0
\end{align*}
since $c\to u_{0,\min}^-$. This establishes the first limit in \eqref{c-tou0min-limit}.
\end{proof}

Due to \eqref{variation2}, we have $\lambda_1(\beta,u_{0,\min})>0$ for $\beta<(u_0'')_{\min}$.
 This, together with \eqref{property-lambda-1} and Lemma \ref{lambda1continuous}, implies that the equation $\lambda_1(\cdot,u_{0,\min})=0$ has a unique solution in $[(u_0'')_{\min},\infty)$, denoted by $\lambda_1^{-1}(0,u_{0,\min})$.
We  now prove the positiveness of $\lambda_1(0,u_{0,\min})$.
\begin{Lemma}\label{non-negativeness-beta0}
Assume that $u_0\in C^2([-d,d])$ and $u_0'\neq0$ on $[-d,d]$. Then $\lambda_1(0,u_{0,\min})>0$, and consequently, $\lambda_1^{-1}(0,u_{0,\min})>0$.
\end{Lemma}
\begin{proof}
Without loss of generality, we may assume that $u_0'\geq C_0$ on $[-d,d]$ for some $C_0>0$.
Let $\phi_0\in H_0^1(-d,d)$ with $\|\phi_0\|_{L^2(-d,d)}=1$ such that \eqref{8136} holds with $\beta=0$. Then
\begin{align*}
\lambda_1(0,u_{0,\min})=&\int_{-d}^{d}\left(|\phi'_{0}|^2+{u_0''\over u_0-u_{0,\min}}|\phi_{0}|^2\right)dy\\
=&\int_{-d}^{d}|\phi'_{0}|^2dy+{u_0'\phi_0^2\over u_0-u_0(-d)}\bigg|_{y=-d}^d-\int_{-d}^du_0'\left({\phi_0^2\over u_0-u_{0,\min}}\right)'dy.
\end{align*}
Since
\begin{align*}
&\left|{u_0'(y)\phi_0(y)^2\over u_0(y)-u_0(-d)}\right|\leq{|u_0'(y)|\|\phi_0'\|_{L^2(y,-d)}^2(y+d)\over C_0(y+d)}\to 0\quad \text{as}\quad y\to-d, \end{align*}
and
\begin{align*}
&\int_{-d}^d{\phi_0\phi_0'u_0'\over u_0-u_{0,\min}}dy\leq C\|\phi_0'\|_{L^2(-d,d)}^2,\;\;\int_{-d}^d{\phi_0^2u_0'^2\over (u_0-u_{0,\min})^2}dy\leq C\|\phi_0'\|_{L^2(-d,d)}^2,
\end{align*}
we have
\begin{align*}
\lambda_1(0,u_{0,\min})=&\int_{-d}^{d}\left(|\phi'_{0}|^2-{2\phi_0\phi_0'u_0'\over u_0-u_{0,\min}}+{\phi_0^2u_0'^2\over (u_0-u_{0,\min})^2}\right)dy\\
=&\int_{-d}^{d}\left(\phi'_{0}-{\phi_0u_0'\over u_0-u_{0,\min}}\right)^2dy\geq0.
\end{align*}

Suppose that $\lambda_1(0,u_{0,\min})=0$. Then $\phi'_{0}={\phi_0u_0'\over u_0-u_{0,\min}}$ for $y\in(-d,d)$. Thus $|\phi_0|=C(u_0-u_{0,\min})$ on $(-d,d)$ and $0=|\phi_0(d)|=C(u_0(d)-u_0(-d))$. This implies that $C=0$ and $\phi_0\equiv0$ on $(-d,d)$, which contradicts $\|\phi_0\|_{L^2(-d,d)}=1$. Therefore, $\lambda_1(0,u_{0,\min})>0$.
\end{proof}

\begin{Remark}
In Lemma \ref{non-negativeness-beta0}, if we replace the monotone shear flow $u_0$ by a non-monotone shear flow $u_0$ such that $u_0$ attains its minimum only at the endpoints $\pm d$ and $u_0'(\pm d)\neq0$, then $\lambda_1(0,u_{0,\min})\geq0$. Moreover, $\lambda_1(0,u_{0,\min})=0$ can be achieved by considering the shear flow $u_0(y)=\cos\left({\pi\over2} y\right)$ on $[-1,1]$. Indeed,  it can be directly verified that $\cos\left({\pi\over2} y\right)$ is an eigenfunction of the principal eigenvalue $\lambda_1(0,u_{0,\min})=0$ for the eigenvalue problem \eqref{sturm-Liouville}.
\end{Remark}

Now we are ready to completely characterize  traveling waves near a monotone shear flow in a bounded channel.
\begin{Theorem}\label{rigidity-near-monotone-shear-flow-arbitrary-wave-speed-thm}
 Let $u_0\in C^2([-d,d])$ and $u_0'\neq0$ on $[-d,d]$.
\begin{itemize}
\item Let $\beta>0$.\\
 For
$\beta\notin Ran(u_0'')$,
\begin{itemize}
\item[(i)] if $L\in\left(0,{2\pi\over \sqrt{\max\{-\lambda_1(\beta,u_{0,\min}),0\}}}\right)$, then
there exists
$\varepsilon_0>0$
such that any traveling-wave solution $(u(x-ct,y),v(x-ct,y))\in C^2(D_L)$ to the $\beta$-plane equation
\eqref{Euler equation}-\eqref{boundary condition for euler} satisfying
$
\|(u,v)-(u_{0},0)\|_{C^{2}(D_L)}$ $<\varepsilon_0
$
must be a shear flow for arbitrary wave speed $c\in\mathbb{R}$;

\item[(ii)]
if $L\in\left[{2\pi\over \sqrt{\max\{-\lambda_1(\beta,u_{0,\min}),0\}}},\infty\right)$,
 then for any $\varepsilon>0$ there exists a genuine traveling-wave solution $(u_\varepsilon(x-c_\varepsilon t,y),v_\varepsilon(x-c_\varepsilon t,y))\in C^2(D_{L})$ to the $\beta$-plane equation
\eqref{Euler equation}-\eqref{boundary condition for euler} satisfying
$
\|(u_\varepsilon,v_\varepsilon)-(u_{0},0)\|_{C^{2}(D_{L})}<\varepsilon
$ and $c_\varepsilon<(u_\varepsilon)_{\min}.$
\end{itemize}
For $\beta\in Ran(u_0'')$,
\begin{itemize}
\item[(iii)] if $L\in\left(0,{2\pi\over \sqrt{\max\left\{-\inf_{\tilde c\leq u_{0,\min}}\lambda_1(\beta,\tilde c),0\right\}}}\right)$, then
there exists
$\varepsilon_0>0$
such that the wave speed of any genuine  traveling-wave solution $(u(x-ct,y),v(x-ct,y))\in C^2(D_L)$ to the $\beta$-plane equation
\eqref{Euler equation}-\eqref{boundary condition for euler} satisfying
$
\|(u,v)-(u_{0},0)\|_{C^{2}(D_L)}$ $<\varepsilon_0
$
must be a generalized inflection value of $u$;

\item[(iv)]
if $L\in\left[{2\pi\over \sqrt{\max\left\{-\inf_{\tilde c\leq u_{0,\min}}\lambda_1(\beta,\tilde c),0\right\}}},\infty\right)$,
 then the same conclusion in $(\rm{ii})$ holds, with the convention that $[a,b)=\emptyset$ if $a\geq b$.
\end{itemize}

\item Let $\beta=0$.\\
If
$0\notin Ran(u_0'')$, then the same conclusion as in (i) holds.\\
If $0\in Ran(u_0'')$, then the same conclusion as in (iii) holds.
\end{itemize}
\end{Theorem}
\begin{Remark}\label{rigidity-near-monotone-shear-flow-arbitrary-wave-speed-thm-rem}
(i) In Theorem \ref{rigidity-near-monotone-shear-flow-arbitrary-wave-speed-thm} (iii), it is natural to ask whether there  exists a monotone shear flow with genuine nearby  traveling waves having wave speeds which are
generalized inflection values of the zonal velocity.  A concrete example is the monotone shear flow $(U_{b,a},0)$ constructed in (5.6) of \cite{WZZ}, which satisfies $U_{b,a}'>0$ on $[-1,1]$, $U_{b,a}(0)=0$ and $\beta=U_{b,a}''(0)$, with $b>0$ sufficiently small and $a>0$. By Theorem 1.5 in \cite{WZZ},  there  exists a genuine   traveling wave $(u(x,y),v(x,y))$ with zero wave speed (i.e. a steady flow) near the shear flow $(U_{b,a},0)$ for some $b, a>0$ sufficiently small. Since $(u,v)$ is $C^2$ close enough to the monotone shear flow $(U_{b,a},0)$, $c=0\in Ran(u)$ and $c$ is neither a critical value nor a maximum/minimum of $u$. By Theorem \ref{wave-speed-inside-range}, $c=0$ must be a generalized inflection value of $u$.

(ii) Let us describe more precisely what Theorem \ref{rigidity-near-monotone-shear-flow-arbitrary-wave-speed-thm} (i)-(ii) reveals for $\beta\notin Ran(u_0'')$.
Assume that $u_0'\neq0$ on $D_L$.  Let $L$ be the zonal period and $c$ be the wave speed of a traveling wave.  We distinguish the following three cases.\\
{\bf Case 1.} $0<(u_0'')_{\min}$.

In the subcase $\lambda_1((u_0'')_{\max},u_{0,\min})>0$, we have $\lambda_1^{-1}(0,u_{0,\min})>(u_0'')_{\max}>0$, and
\begin{itemize}
\item if $\beta\in[0,(u_0'')_{\min})\cup((u_0'')_{\max},\lambda_1^{-1}(0,u_{0,\min})]$, then $\lambda_1(\beta,u_{0,\min})\geq0$ and
 the rigidity of traveling waves with $c\in\mathbb{R}$ holds for $L>0$,
  \item
  if $\beta\in(\lambda_1^{-1}(0,u_{0,\min}),\infty)$, then $\lambda_1(\beta,u_{0,\min})<0$ and
  \begin{itemize}
\item the rigidity of traveling waves with $c\in\mathbb{R}$  holds for  $L\in\left(0,{2\pi\over \sqrt{-\lambda_1(\beta,u_{0,\min})}}\right)$,
\item
   genuine traveling waves exist for $L\in\left[{2\pi\over \sqrt{-\lambda_1(\beta,u_{0,\min})}},\infty\right)$,
   \end{itemize}
   \end{itemize}
 near  $(u_0,0)$ in $C^2$.

In the other subcase, $\lambda_1((u_0'')_{\max},u_{0,\min})\leq0$, we have $(u_0'')_{\max}\geq\lambda_1^{-1}(0,u_{0,\min})>0$ (due to $\lambda_1(0,u_{0,\min})>\lambda_1((u_0'')_{\min},u_{0,\min})\geq0$), and
\begin{itemize}
\item if $\beta\in[0,(u_0'')_{\min})$, then $\lambda_1(\beta,u_{0,\min})>0$ and
 the rigidity of traveling waves with $c\in\mathbb{R}$ holds for $L>0$,
  \item
  if $\beta\in((u_0'')_{\max},\infty)$, then $\lambda_1(\beta,u_{0,\min})<0$ and
  \begin{itemize}
\item the rigidity of traveling waves with $c\in\mathbb{R}$  holds for  $L\in\left(0,{2\pi\over \sqrt{-\lambda_1(\beta,u_{0,\min})}}\right)$,
\item
   genuine traveling waves exist for $L\in\left[{2\pi\over \sqrt{-\lambda_1(\beta,u_{0,\min})}},\infty\right)$,
   \end{itemize}
   \end{itemize}
 near  $(u_0,0)$ in $C^2$. \\
 {\bf Case 2.} $0\in Ran (u_0'')$.

 In the subcase $\lambda_1((u_0'')_{\max},u_{0,\min})>0$, we have $\lambda_1^{-1}(0,u_{0,\min})>(u_0'')_{\max}\geq0$, and
\begin{itemize}
\item if $\beta\in((u_0'')_{\max},\lambda_1^{-1}(0,u_{0,\min})]$, then $\lambda_1(\beta,u_{0,\min})\geq0$ and
 the rigidity of traveling waves with $c\in\mathbb{R}$ holds for $L>0$,
  \item
  if $\beta\in(\lambda_1^{-1}(0,u_{0,\min}),\infty)$, then $\lambda_1(\beta,u_{0,\min})<0$ and
  \begin{itemize}
\item the rigidity of traveling waves with $c\in\mathbb{R}$  holds for $L\in\left(0,{2\pi\over \sqrt{-\lambda_1(\beta,u_{0,\min})}}\right)$,
\item
   genuine traveling waves exist for  $L\in\left[{2\pi\over \sqrt{-\lambda_1(\beta,u_{0,\min})}},\infty\right)$,
   \end{itemize}
   \end{itemize}
  near  $(u_0,0)$ in $C^2$, see Fig.\,\ref{rigidity-existence-tw} (a).

In the other subcase, $\lambda_1((u_0'')_{\max},u_{0,\min})\leq0$, by Lemma \ref{non-negativeness-beta0} we have $(u_0'')_{\max}\geq\lambda_1^{-1}(0,u_{0,\min})>0$, and
  for every $\beta\in((u_0'')_{\max},\infty)$, $\lambda_1(\beta,u_{0,\min})<0$ and
  \begin{itemize}
\item the rigidity of traveling waves with $c\in\mathbb{R}$  holds for  $L\in\left(0,{2\pi\over \sqrt{-\lambda_1(\beta,u_{0,\min})}}\right)$,
\item
   genuine traveling waves exist for  $L\in\left[{2\pi\over \sqrt{-\lambda_1(\beta,u_{0,\min})}},\infty\right)$,
   \end{itemize}
 near  $(u_0,0)$ in $C^2$,  see Fig.\,\ref{rigidity-existence-tw} (b).

 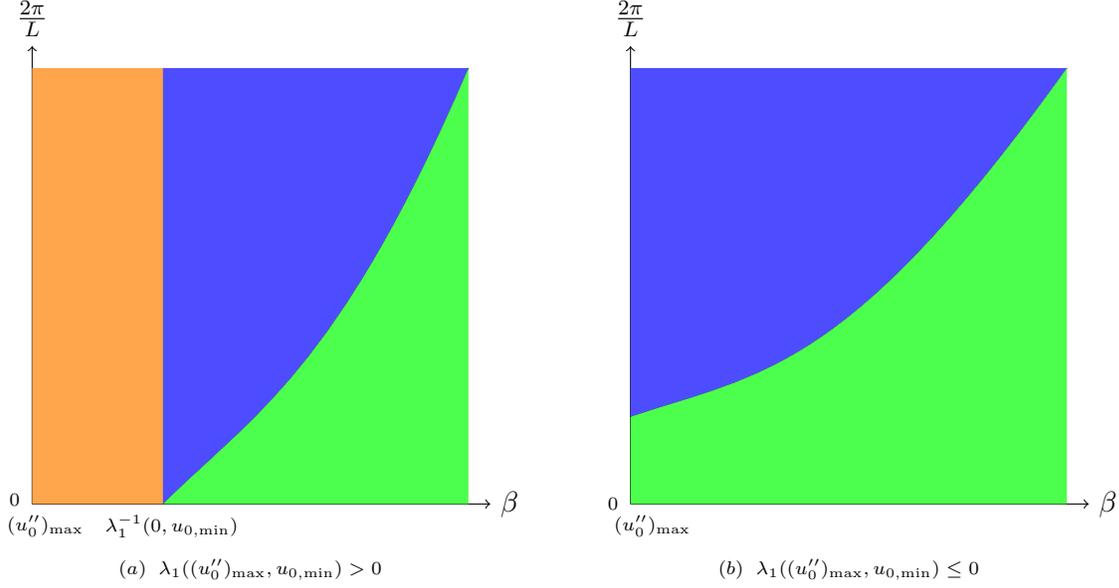
\begin{figure}[h]
\begin{center}
 \begin{tikzpicture}[scale=0.58]
  \draw [->](0, 0)--(10.5, 0)node[right]{$\beta$};
 \draw [->](0, 0)--(0, 10.5)node[above]{${2\pi\over L}$};

       \node (a) at (0.3,-0.5) {\tiny$(u_0'')_{\max}$};
       \node (a) at (-0.4,0.1) {\tiny$0$};
       \node (a) at (3.2,-0.5) {\tiny$\lambda_1^{-1}(0,u_{0,\min})$};
       \draw  (3, 0).. controls (5, 2) and (7, 3)..(10, 10);
       \draw  (3, 0).. controls (3, 5) and (3, 7)..(3, 10);

        \node (a) at (5,-1.5) {\tiny$(a) \;\;\lambda_1((u_0'')_{\max},u_{0,\min})>0$};
        \fill[orange!70] (3, 0).. controls (3, 5) and (3, 7)..(3, 10) -- (0,10) -- (0,0) -- cycle;

\fill[blue!70] (3, 0).. controls (5, 2) and (7, 3)..(10, 10) -- (3,10) -- (3,5) -- cycle;
\fill[green!70] (3, 0).. controls (5, 2) and (7, 3)..(10, 10) -- (10,5) -- (10,0) -- cycle;

 \end{tikzpicture}
 \quad\quad
   \begin{tikzpicture}[scale=0.58]
  \draw [->](0, 0)--(10.5, 0)node[right]{$\beta$};
 \draw [->](0, 0)--(0, 10.5)node[above]{${2\pi\over L}$};

       \node (a) at (0.5,-0.5) {\tiny$(u_0'')_{\max}$};
       \node (a) at (-0.4,0) {\tiny$0$};
       \draw  (0, 2).. controls (3, 3) and (5, 3)..(10, 10);

        \fill[blue!70] (0, 2).. controls (3, 3) and (5, 3)..(10, 10) -- (0,10) -- (0,2) -- cycle;
\fill[green!70] (0, 2).. controls (3, 3) and (5, 3)..(10, 10) -- (10,0) -- (0,0) -- cycle;
 \node (a) at (5,-1.5) {\tiny$(b) \;\;\lambda_1((u_0'')_{\max},u_{0,\min})\leq0$};
 \end{tikzpicture}
\end{center}
	\caption{The illustration for Case 2 in Remark \ref{rigidity-near-monotone-shear-flow-arbitrary-wave-speed-thm-rem} (ii):
 near a monotone shear flow, the rigidity of traveling waves with arbitrary wave speeds holds for $(\beta,{2\pi\over L})$ in the orange and blue regions, while
   genuine traveling waves exist for $(\beta,{2\pi\over L})$ in the green regions.
   The boundary between the blue and green regions is the curve $ {2\pi/L}=\sqrt{-\lambda_1(\beta,u_{0,\min})}$.
}
\label{rigidity-existence-tw}
\end{figure}

\noindent
 {\bf Case 3.} $0>(u_0'')_{\max}$.

 In this case, we have $\lambda_1((u_0'')_{\max},u_{0,\min})>0$, since otherwise,  $0>(u_0'')_{\max}\geq\lambda_1^{-1}(0,u_{0,\min})$, and thus $\lambda_1(0,u_{0,\min})<0$, which contradicts $\lambda_1(0,u_{0,\min})>0$ (see Lemma \ref{non-negativeness-beta0}).
  Since $\lambda_1((u_0'')_{\max},u_{0,\min})>0$, by Lemma \ref{non-negativeness-beta0} we have $\lambda_1^{-1}(0,u_{0,\min})>0>(u_0'')_{\max}$,
 and
\begin{itemize}
\item if $\beta\in[0,\lambda_1^{-1}(0,u_{0,\min})]$, then $\lambda_1(\beta,u_{0,\min})\geq0$ and
 the rigidity of traveling waves with $c\in\mathbb{R}$ holds for $L>0$,
  \item
  if $\beta\in(\lambda_1^{-1}(0,u_{0,\min}),\infty)$, then $\lambda_1(\beta,u_{0,\min})<0$ and
  \begin{itemize}
\item the rigidity of traveling waves with $c\in\mathbb{R}$  holds for  $L\in\left(0,{2\pi\over \sqrt{-\lambda_1(\beta,u_{0,\min})}}\right)$,
\item
   genuine traveling waves exist for  $L\in\left[{2\pi\over \sqrt{-\lambda_1(\beta,u_{0,\min})}},\infty\right)$,
   \end{itemize}
   \end{itemize}
  near  $(u_0,0)$ in $C^2$.
\end{Remark}

\begin{Remark}
Let us compare the rigidity property of traveling waves with arbitrary wave speeds near monotone shear flows in bounded versus unbounded channels.
Assume that the Rayleigh stability condition (i.e. $\beta\notin \overline{Ran(u_0'')}$) holds.
In an unbounded channel, the rigidity  near general monotone shear flows with appropriate asymptotic behavior is proved for arbitrary zonal periods (see Remark \ref{rem-rigidity-unbounded-channel}). In a bounded channel, however, it depends on $\beta$ and on the zonal period $L$, see Theorem \ref{rigidity-near-monotone-shear-flow-arbitrary-wave-speed-thm} and Remark \ref{rigidity-near-monotone-shear-flow-arbitrary-wave-speed-thm-rem}. In particular, we provide critical regions of $(\beta,L)$ such that the rigidity near monotone shear flows does not hold in Theorem \ref{rigidity-near-monotone-shear-flow-arbitrary-wave-speed-thm}, see the green regions of Fig.\,\ref{rigidity-existence-tw} for Case 2 in Remark \ref{rigidity-near-monotone-shear-flow-arbitrary-wave-speed-thm-rem}.
\end{Remark}

\begin{proof} Let $\beta>0$.
We  prove the following equivalent statements for $\beta\notin Ran(u_0'')$.
\\
(1) Assume that $0<(u_0'')_{\min}$. Let $\beta\in(0,(u_0'')_{\min})$ and $L>0$ be arbitrary. Then there exists
$\varepsilon_0>0$
such that any traveling-wave solution $(u(x-ct,y),v(x-ct,y))\in C^2(D_L)$ to the $\beta$-plane equation
\eqref{Euler equation}-\eqref{boundary condition for euler} satisfying
$
\|(u,v)-(u_{0},0)\|_{C^{2}(D_L)}$ $<\varepsilon_0
$
must be a shear flow for an arbitrary wave speed $c\in\mathbb{R}$.
\\
(2) Assume that $\lambda_1((u_0'')_{\max},u_{0,\min})>0$.
 \begin{itemize}
\item[${\rm(2a)}$]
Let  $\beta\in((u_0'')_{\max},\lambda_1^{-1}(0,u_{0,\min})]\cap(0,\infty)$ and $L>0$ be arbitrary.
Then the same conclusion as in (1) holds.
\item[${\rm(2b)}$]
Let  $\beta\in(\lambda_1^{-1}(0,u_{0,\min}),\infty)$. (Note that $0< \lambda_1^{-1}(0,u_{0,\min})$.)
\begin{itemize}
\item[${\rm(2b_1)}$]
If $L\in\left(0,{2\pi\over \sqrt{-\lambda_1(\beta,u_{0,\min})}}\right)$,
 then the same conclusion as in ${\rm (1)}$ holds.
\item[${\rm(2b_2)}$]
If $L\in\left[{2\pi\over \sqrt{-\lambda_1(\beta,u_{0,\min})}},\infty\right)$,
 then for any $\varepsilon>0$ there exists a genuine traveling-wave solution $(u_\varepsilon(x-c_\varepsilon t,y),v_\varepsilon(x-c_\varepsilon t,y))\in C^2(D_{L})$ to the $\beta$-plane equation
\eqref{Euler equation}-\eqref{boundary condition for euler} satisfying
$
\|(u_\varepsilon,v_\varepsilon)-(u_{0},0)\|_{C^{2}(D_{L})}<\varepsilon
$ and $c_\varepsilon<(u_\varepsilon)_{\min}.$
\end{itemize}
\end{itemize}
(3)
Assume that $\lambda_1((u_0'')_{\max},u_{0,\min})\leq0$. Let $\beta\in((u_0'')_{\max},\infty)$. (Note that $0< \lambda_1^{-1}(0,u_{0,\min})\leq (u_0'')_{\max}$.)
\begin{itemize}
\item[${\rm(3a)}$]
If $L\in\left(0,{2\pi\over \sqrt{-\lambda_1(\beta,u_{0,\min})}}\right)$,
 then the same conclusion as in ${\rm (1)}$ holds.
\item[${\rm(3b)}$]
If $L\in\left[{2\pi\over \sqrt{-\lambda_1(\beta,u_{0,\min})}},\infty\right)$,
 then  the same conclusion as in ${\rm (2b_2)}$ holds.
\end{itemize}

We  prove the following equivalent statements for $\beta\in Ran(u_0'')$.\\
(4) Assume that $\inf_{\tilde c\leq u_{0,\min}}\lambda_1(\beta,\tilde c)\geq0$. Let $L>0$ be arbitrary. Then there exists
$\varepsilon_0>0$
such that the wave speed of any genuine  traveling-wave solution $(u(x-ct,y),v(x-ct,y))\in C^2(D_L)$ to the $\beta$-plane equation
\eqref{Euler equation}-\eqref{boundary condition for euler} satisfying
$
\|(u,v)-(u_{0},0)\|_{C^{2}(D_L)}$ $<\varepsilon_0
$
must be a generalized inflection value of $u$.\\
(5) Assume that $\inf_{\tilde c\leq u_{0,\min}}\lambda_1(\beta,\tilde c)<0$.
\begin{itemize}
\item[${\rm(5a)}$]
If $L\in\left(0,{2\pi\over \sqrt{-\inf_{\tilde c\leq u_{0,\min}}\lambda_1(\beta,\tilde c)}}\right)$,
 then the same conclusion as in ${\rm (4)}$ holds.
\item[${\rm(5b)}$]
If $L\in\left[{2\pi\over \sqrt{-\inf_{\tilde c\leq u_{0,\min}}\lambda_1(\beta,\tilde c)}},\infty\right)$,
 then  the same conclusion as in ${\rm (2b_2)}$ holds.
\end{itemize}

First, we prove the statement (1).
Since $\beta\in(0,(u_0'')_{\min})$, we can choose $\varepsilon_0\in(0,1)$ small enough such that $\beta<(\Delta u)_{\min}$ on $D_L$. By Theorem \ref{thm-generalization1}, $(u(x-ct,y),v(x-ct,y))$ is
 a shear flow for all $c\in \mathbb{R}$.

We prove now the statement (2).
By \eqref{property-lambda-1}, we have
\begin{align}\label{principal-eigenvalue-beta+}
&\lambda_1(\beta,u_{0,\min})>0\quad \text{for} \quad\beta\in((u_0'')_{\max},\lambda_1^{-1}(0,u_{0,\min})), \\\nonumber &\lambda_1(\beta,u_{0,\min})<0\quad\text{for} \quad\beta\in(\lambda_1^{-1}(0,u_{0,\min}),\infty).
\end{align}
To prove ${\rm(2a)}$,
without loss of generality, we assume that $u_0'>0$ on $[-d,d]$.
Since $\beta>(u_0'')_{\max}$ and $u_0'>0$ on $[-d,d]$, we can choose $\varepsilon_0\in(0,1)$ small enough such that
$\beta> (\Delta u)_{\max}$ and
$
\partial_{y}u> 0$ on $D_L.
$

Let  $\beta\in((u_0'')_{\max},\lambda_1^{-1}(0,u_{0,\min})]\cap(0,\infty)$ and $L>0$ be arbitrary.
Since $\beta>0$,
by Theorem \ref{beta=0-cor} (i) we know that  $(u(x-ct,y),v(x-ct,y))$ is
 a shear flow for $c\notin[c_\beta^+,u_{\min}]$, where $c_\beta^+$ is defined in \eqref{def-c+} with $d_{\pm}=\pm d$. Now, let $(u(x-ct,y),v(x-ct,y))$ be
 a traveling wave  with $c\in[c_\beta^+,u_{\min}]$. We will show that it has to be a shear flow.
 Since $\partial_y u\geq C_0$ on $D_L$ for some $C_0>0$, we have
 \begin{align}\label{estimate-uxy-c}
 u(x,y)-c\geq u(x,y)-u_{\min}\geq u(x,y)-u(x,-d)\geq C_0(y+d)\end{align}
  for $(x,y)\in D_L$. Moreover,
  $$\|\beta-\Delta u\|_{C^0(D_L)}\leq \|\beta-u_0''\|_{C^0(D_L)}+\|u_0''-\Delta u\|_{C^0(D_L)}\leq C+\varepsilon_0\leq C\,.$$
 By \eqref{un-vn-eq-u-v}, \eqref{estimate-uxy-c},  $v(\cdot,- d)=0$ and the Hardy inequality \eqref{Hardy-inequality},  we have
 \begin{align*}
 &\|\nabla v\|_{L^2(D_L)}^2=\int_{D_L}|\nabla v|^2 dxdy \leq C\int_{D_L}{1\over u-c}v^2dxdy\\\leq& C\int_{D_L}{v^2\over y+d} dxdy
 \leq C\|v\|_{L^2(D_L)}\left(\int_{\mathbb{T}_L}\int_{-d}^d{v^2\over (y+d)^2} dydx\right)^{1\over2}\\
 \leq& C_1\|v\|_{L^2(D_L)}\left(\int_{\mathbb{T}_L}\int_{-d}^d|\partial_y v(x,y)|^2 dydx\right)^{1\over2}
 \leq C_1\|v\|_{L^2(D_L)}\|\nabla v\|_{L^2(D_L)},
 \end{align*}
 which implies
 \begin{align}\label{uniformH1bound}
 \|\nabla v\|_{L^2(D_L)}
 \leq C_1\|v\|_{L^2(D_L)}
 \end{align}
 for some $C_1>0$.
 Since $\beta>(\Delta u)_{\max}$, by \eqref{estimate-uxy-c} we have
 \begin{align}\label{decomposition-I-II}
 0=&\int_{D_L}\left(|\nabla v|^2 -{\beta-\Delta u\over u-c}v^2\right)dxdy\\\nonumber
 \geq&\int_{D_L}\left(|\nabla v|^2 -{\beta-\Delta u\over u-u(x,-d)}v^2\right)dxdy\\\nonumber
 =&\int_{D_L}\left(|\nabla v|^2-{\beta- u_0''\over u_0-u_0(-d)}v^2\right)dxdy\\\nonumber
 &+\int_{D_L}\left( {\beta- u_0''\over u_0-u_0(-d)}v^2 -{\beta-\Delta u\over u-u(x,-d)}v^2\right)dxdy\\\nonumber
 =&I+II.
 \end{align}
 Since $\beta\in((u_0'')_{\max},\lambda_1^{-1}(0,u_{0,\min})]$, by \eqref{principal-eigenvalue-beta+} we have $\lambda_1(\beta,u_{0,\min})\geq0$ and
 \begin{align*}
 \int_{-d}^d\left(|\phi'|^2-{\beta-u_0''\over u_0-u_{0,\min}}|\phi|^2\right)dy\geq \lambda_1(\beta,u_{0,\min})\int_{-d}^d|\phi|^2dy\geq0
 \end{align*}
 for $0\neq \phi\in H_0^1(-d,d)$.
 Then $v(x,y)=\sum_{k\neq0} e^{i{2k\pi\over L}x} v_k(y)$ and
 \begin{align*}
 I=&L\sum_{k\neq0}\int_{-d}^d\left(|v_k'|^2+k^2\left({2\pi\over L}\right)^2|v_k|^2-{\beta-u_0''\over u_0-u_{0,\min}}|v_k|^2\right)dy\\
 \geq&L\sum_{k\neq0}\int_{-d}^d\left(\lambda_1(\beta,u_{0,\min})+k^2\left({2\pi\over L}\right)^2\right)|v_k|^2dy\\
 \geq&\left({2\pi\over L}\right)^2\|v\|_{L^2(D_L)}^2\geq{\left({2\pi/ L}\right)^2\over C_1^2}\|\nabla v\|_{L^2(D_L)}^2,
 \end{align*}
where we have used \eqref{uniformH1bound} and $\int_{\mathbb{T}_L} v dx=0$. On the other hand, by \eqref{estimate-uxy-c}
and
$u_0(y)-u_0(-d)\geq\tilde C_0(y+d)$ for some $\tilde C_0>0$, we have
\begin{align}\label{II-estimate}
-II=&\int_{D_L}\bigg(  {\beta-\Delta u\over u-u(x,-d)}v^2-{\beta- u_0''\over u-u(x,-d)}v^2+{\beta- u_0''\over u-u(x,-d)}v^2\\&
-{\beta- u_0''\over u_0-u_0(-d)}v^2\bigg)dxdy\nonumber\\
=&\int_{D_L}\left(  {u_0''-\Delta u\over u-u(x,-d)}v^2+{(\beta- u_0'')(u_0-u-u_0(-d)+u(x,-d))\over (u-u(x,-d))(u_0-u_0(-d))}v^2\right)dxdy\nonumber\\
\leq&C\|u_0- u\|_{C^2(D_L)}\int_{D_L}{1\over (y+d)}v^2 dxdy+C\|u_0- u\|_{C^0(D_L)}\int_{D_L}{1\over (y+d)^2}v^2 dxdy\nonumber\\
\leq&C\|u_0-u\|_{C^2(D_L)}\|v\|_{L^2(D_L)}\|\nabla v\|_{L^2(D_L)}+C\|u_0- u\|_{C^0(D_L)}\|\nabla v\|_{L^2(D_L)}^2\nonumber\\
\leq&C_2\|u_0-u\|_{C^2(D_L)}\|\nabla v\|_{L^2(D_L)}^2\nonumber
\end{align}
for some $C_2>0$,
where we have used $v(\cdot, -d)=0$ and the Hardy and Poincar\'{e} inequalities in the last two estimates. Combining the estimates for $I$ and $II$, we obtain
\begin{align*}
{\left({2\pi/ L}\right)^2\over C_1^2}\|\nabla v\|_{L^2(D_L)}^2\leq I\leq-II\leq C_2\|u_0-u\|_{C^2(D_L)}\|\nabla v\|_{L^2(D_L)}^2.
\end{align*}
Taking $\varepsilon_0>0$ such that $\varepsilon_0<{\left({2\pi/ L}\right)^2\over C_1^2C_2}$, we get $\nabla v=0$ and $v=0$ on $D_L$,  proving $(\rm{2a})$.

We prove now $(\rm{2b_1})$.
Let $\beta\in(\lambda_1^{-1}(0,u_{0,\min}),\infty)$ and $L\in\left(0,{2\pi\over \sqrt{-\lambda_1(\beta,u_{0,\min})}}\right)$. Then $\beta>\lambda_1^{-1}(0,u_{0,\min})>0$ by Lemma \ref{non-negativeness-beta0}.
By Theorem  \ref{beta=0-cor} (i),  we only need to  show that if $(u(x-ct,y),v(x-ct,y))$ is
 a traveling wave with $c\in[c_\beta^+,u_{\min}]$, then it must be a shear flow.
We still have \eqref{uniformH1bound}, and define $I$ and $II$ by \eqref{decomposition-I-II}.
Since $\beta\in(\lambda_1^{-1}(0,u_{0,\min}),\infty)$, by \eqref{principal-eigenvalue-beta+} and $L\in\left(0,{2\pi\over \sqrt{-\lambda_1(\beta,u_{0,\min})}}\right)$, we have
\begin{align}\label{lambda1T0}
0>\lambda_1(\beta,u_{0,\min})>-\left(2\pi\over L\right)^2
\end{align}
and
\begin{align}\label{lambda1variation}
 \int_{-d}^d\left(|\phi'|^2-{\beta-u_0''\over u_0-u_{0,\min}}|\phi|^2\right)dy\geq \lambda_1(\beta,u_{0,\min})\int_{-d}^d|\phi|^2dy
\end{align}
 for $0\neq \phi\in H_0^1(-d,d)$.
 Then by \eqref{lambda1variation} we obtain
 \begin{align}\label{estimate-I-ib1}
 I \geq&L\sum_{k\neq0}\int_{-d}^d\left(\lambda_1(\beta,u_{0,\min})+\left({2\pi\over L}\right)^2\right)|v_k|^2dy\\\nonumber
 &+L\sum_{k\neq0}\int_{-d}^d\left(-\left({2\pi\over L}\right)^2+k^2\left({2\pi\over L}\right)^2\right)|v_k|^2dy\\\nonumber
 \geq&L\left(\lambda_1(\beta,u_{0,\min})+\left({2\pi\over L}\right)^2\right)\sum_{k\neq0}\int_{-d}^d|v_k|^2dy\\\nonumber
 =&\left(\lambda_1(\beta,u_{0,\min})+\left({2\pi\over L}\right)^2\right)\|v\|_{L^2(D_L)}^2
 \geq{\lambda_1(\beta,u_{0,\min})+\left({2\pi\over L}\right)^2\over C_1^2}\|\nabla v\|_{L^2(D_L)}^2,
 \end{align}
where $v(x,y)=\sum_{k\neq0} e^{i{2k\pi\over L}x} v_k(y)$ and we have used \eqref{uniformH1bound}. Note that
$
{\lambda_1(\beta,u_{0,\min})+\left({2\pi\over L}\right)^2\over C_1^2}>0
$ due to \eqref{lambda1T0}.
The estimates \eqref{II-estimate} for $II$ still hold. By \eqref{estimate-I-ib1} and \eqref{II-estimate} we have
\begin{align*}
{\lambda_1(\beta,u_{0,\min})+\left({2\pi\over L}\right)^2\over C_1^2}\|\nabla v\|_{L^2(D_L)}^2\leq I\leq-II\leq C_2\|u_0-u\|_{C^2(D_L)}\|\nabla v\|_{L^2(D_L)}^2.
\end{align*}
Taking  $\varepsilon_0<{\lambda_1(\beta,u_{0,\min})+\left({2\pi\over L}\right)^2\over C_1^2C_2}$, we conclude that $v=0$ on $D_L$. This proves  $(\rm{2b_1})$.

We now prove $(\rm{2b_2})$.
Let $\beta\in(\lambda_1^{-1}(0,u_{0,\min}),\infty)$.
We first consider the case $L > {2\pi\over \sqrt{-\lambda_1(\beta,u_{0,\min};u_0)}}$, where $\lambda_1(\beta,u_{0,\min};u_0)<0$ by \eqref{principal-eigenvalue-beta+}. Then
$$\lambda_1(\beta, u_{0,\min};u_0)<-\left({2\pi\over L}\right)^2<0\,$$
 This, combined with Lemma \ref{c-tou0min-limit-lem}, implies that
there exists $c_{L}<u_{0,\min}$ such that $\lambda_1(\beta,c_{L};u_0)=-\left({2\pi\over L}\right)^2$. Thus $c_{L}$ is an isolated real eigenvalue of the operator
\begin{align*}
\mathcal{R}_{{2\pi\over L},\beta,u_0}:=-\left(\partial^2_y-\left({2\pi\over L}\right)^2\right)^{-1}\left((u_0''-\beta)-u_0\left(\partial^2_y-\left({2\pi\over L}\right)^2\right)\right).
\end{align*}
Although the spectral condition  $c_{L}\in\sigma_d(\mathcal{R}_{{2\pi\over L},\beta,u_0})\cap\mathbb{R}$ is satisfied, we can not directly apply the bifurcation lemma
(Lemma  2.5  in \cite{LWZZ})  to the shear flow $(u_0,0)$ since it requires the shear flow to possess $H^4$ regularity, whereas we only have $u_0\in C^2.$ Our approach is to perturb $(u_0,0)$ to a nearby shear flow $(\tilde u_0,0)$ with $C^4$ regularity. The only subtlety that remains is then to ensure that the corresponding operator $\mathcal{R}_{{2\pi\over L},\beta,\tilde u_0}$  retains an isolated real eigenvalue. As we will see, this obstacle is readily resolved by Lemmas \ref{lambda1continuous-profile-lem}-\ref{c-tou0min-limit-lem}.
Indeed, for any $\varepsilon>0$, by Lemma \ref{lambda1continuous-profile-lem} and the density of $C^4([-d,d])$ in $C^2([-d,d])$, we can choose $\tilde u_0\in C^4([-d,d])$ such that $\tilde u_0'>0$ on $[-d,d]$,
\begin{align}\label{tilde-u0-u0}
\|\tilde u_0-u_0\|_{C^2([-d,d])}<{\varepsilon\over2}
\end{align}
and
\begin{align*}
|\lambda_1(\beta,\tilde u_{0,\min};\tilde u_0)-\lambda_1(\beta, u_{0,\min}; u_0)|<{-\left(2\pi\over L\right)^2-\lambda_1(\beta,u_{0,\min};u_0)\over 2}.
\end{align*}
Then
\begin{align}\nonumber
-\left(2\pi\over L\right)^2-\lambda_1(\beta,\tilde u_{0,\min};\tilde u_0)
>&-\left(2\pi\over L\right)^2-\lambda_1(\beta,u_{0,\min};u_0)-{-\left(2\pi\over L\right)^2-\lambda_1(\beta,u_{0,\min};u_0)\over 2}\\\label{2pi-L-lambda-tilde}
=&{-\left(2\pi\over L\right)^2-\lambda_1(\beta,u_{0,\min};u_0)\over 2}>0,
\end{align}
which implies
\begin{align}\label{wave-number-mid}\lambda_1(\beta,\tilde u_{0,\min};\tilde u_0)<-\left({2\pi\over L}\right)^2<0.\end{align}
By applying Lemma \ref{c-tou0min-limit-lem} to $\tilde u_0$, we have
\begin{align}\label{c-tou0min-limit-tilde-u0}
\lim_{c\to\tilde u_{0,\min}^-}\lambda_1(\beta,c;\tilde u_0)=\lambda_1(\beta,\tilde u_{0,\min};\tilde u_0)<0\quad\text{and}\quad \lim_{c\to -\infty}\lambda_1(\beta,c;\tilde u_0)={\pi^2\over4d^2}>0.
\end{align}
It follows from \eqref{wave-number-mid}-\eqref{c-tou0min-limit-tilde-u0} that
there exists $\tilde c_L<\tilde u_{0,\min}$ such that $\lambda_1(\beta,\tilde c_L;\tilde u_0)=-\left({2\pi\over L}\right)^2$. Then $\tilde c_L$ is an isolated real eigenvalue of the operator $\mathcal{R}_{{2\pi\over L},\beta,\tilde u_0}$. We now can apply Lemma 2.5 in \cite{LWZZ} to the shear flow $(\tilde u_0,0)$, and obtain the existence of a genuine traveling-wave solution $(u_\varepsilon(x-c_\varepsilon t,y),v_\varepsilon(x-c_\varepsilon t,y))\in C^2(D_{L})$ to the $\beta$-plane equation
\eqref{Euler equation}-\eqref{boundary condition for euler} satisfying $c_\varepsilon<(u_\varepsilon)_{\min}$ and
\begin{align}\label{u-varepsilon-u0}
\|(u_\varepsilon,v_\varepsilon)-(\tilde u_{0},0)\|_{C^{2}(D_{L})}\leq \|(u_\varepsilon,v_\varepsilon)-(\tilde u_{0},0)\|_{H^{3}(D_{L})}<{\varepsilon\over 2} \,.
\end{align}
Combining \eqref{tilde-u0-u0} and \eqref{u-varepsilon-u0}, we conclude that $\|(u_\varepsilon,v_\varepsilon)-(u_{0},0)\|_{C^{2}(D_{L})}<\varepsilon$.

We next consider the case $L={2\pi\over \sqrt{-\lambda_1(\beta,u_{0,\min};u_0)}}$. In this case, the operator
$
\mathcal{R}_{{2\pi\over L},\beta,u_0}$ has a non-isolated eigenvalue $u_{0,\min}$.
The strategy is as follows. We first perturb the shear flow $(u_0,0)$ to a scaled shear flow $(au_0,0)$ with $0<a<1$ so that the corresponding operator $\mathcal{R}_{{2\pi\over L},\beta,a u_0}$ admits an isolated real eigenvalue. Then, as in the previous case, we further perturb $(au_0,0)$ to a shear flow $(\hat u_0,0)$ with $C^4$ regularity and apply the bifurcation lemma to $(\hat u_0,0)$ in order to obtain a nearby genuine traveling-wave solution.
More precisely,
for the shear flow $(au_0,0)$ with $a\in(0,1)$,
 the corresponding
Rayleigh-Kuo  boundary value problem is
\begin{align*}
-\phi''-{\beta-au_0''\over au_0-c}\phi=-\phi''-{{\beta\over a}-u_0''\over u_0-{c\over a}}\phi=\lambda\phi, \;\;\;\;\phi(\pm d)=0
\end{align*}
with $\phi\in H_0^1\cap H^2(-d,d)$, where $c\leq au_{0,\min}$. Then
\begin{align}\label{scaled-eigenvalue}
\lambda_{1}(\beta,c;au_0)=\lambda_{1}\left({\beta\over a},{c\over a};u_0\right)
\end{align}
 for $c\leq a u_{0,\min}$. By
\eqref{property-lambda-1} and ${\beta\over a}>\beta$, we have
 \begin{align}\label{eigenvalue-u0min-beta}
-\left(2\pi\over L\right)^2=\lambda_1(\beta,u_{0,\min};u_0)>\lambda_1\left({\beta\over a},u_{0,\min};u_0\right)=\lambda_1(\beta,au_{0,\min};au_0).\end{align}
By  Lemma \ref{c-tou0min-limit-lem} and \eqref{eigenvalue-u0min-beta}, there exists $c_{L,a}<u_{0,\min}$ such that \begin{align}\label{eigenvalue-u0min-beta-a}
\lambda_1\left({\beta\over a},c_{L,a};u_0\right)=-\left(2\pi\over L\right)^2.\end{align}
Combining \eqref{scaled-eigenvalue} and \eqref{eigenvalue-u0min-beta-a}, we have
\begin{align}\label{2-pi-L-mid-hat-u0}
\lambda_1(\beta,ac_{L,a};au_0)=-\left(2\pi\over L\right)^2.
\end{align}
 Since $ac_{L,a}<au_{0,\min}$, $ac_{L,a}$ is an isolated real eigenvalue of the operator $\mathcal{R}_{{2\pi\over L},\beta,au_0}$. For any $\varepsilon>0$, we take $a\in(0,1)$ close to $1$ such that \begin{align}\label{scaled-varepsilon-3}\|au_0-u_0\|_{C^2([-d,d])}<{\varepsilon\over3}.\end{align}
  By \eqref{2-pi-L-mid-hat-u0}, Lemma \ref{lambda1continuous-profile-lem} and the density of $C^4([-d,d])$ in $C^2([-d,d])$, we can choose $\hat u_0\in C^4([-d,d])$ such that $\hat u_0'>0$ on $[-d,d]$,
\begin{align}\label{hat-scaled-3}
\|\hat u_0-au_0\|_{C^2([-d,d])}<{\varepsilon\over3}
\end{align}
and
\begin{align*}
|\lambda_1(\beta,\hat u_{0,\min};\hat u_0)-\lambda_1(\beta, au_{0,\min}; au_0)|<{-\left(2\pi\over L\right)^2-\lambda_1(\beta,au_{0,\min};au_0)\over 2}.
\end{align*}
  By an argument similar to that in \eqref{2pi-L-lambda-tilde}, we have
  \begin{align}\label{wave-number-mid-hat}
  \lambda_1(\beta,\hat u_{0,\min};\hat u_0)<-\left({2\pi\over L}\right)^2<0.
  \end{align}
  Applying Lemma \ref{c-tou0min-limit-lem} to $\hat u_0$ and using  \eqref{wave-number-mid-hat},  there exists $\hat c_L<\hat u_{0,\min}$ such that $\lambda_1(\beta,\hat c_L;\hat u_0)=-\left({2\pi\over L}\right)^2$. Thus $\hat c_{L}$ is an isolated real eigenvalue of the operator $\mathcal{R}_{{2\pi\over L},\beta,\hat u_0}$.
  Applying Lemma 2.5 in \cite{LWZZ} to the shear flow $(\hat u_0,0)$, we  obtain the existence of a genuine traveling-wave solution $(u_\varepsilon(x-c_\varepsilon t,y),v_\varepsilon(x-c_\varepsilon t,y))\in C^2(D_L)$ such that
  $c_\varepsilon<(u_\varepsilon)_{\min}$ and
   \begin{align}\label{u-varepsilon-hat-u0}
   \|(u_\varepsilon,v_\varepsilon)-(\hat u_0,0)\|_{C^2(D_L)}<{\varepsilon\over3}\,.\end{align}
    Combining \eqref{scaled-varepsilon-3}-\eqref{hat-scaled-3} and \eqref{u-varepsilon-hat-u0}, we conclude that $\|(u_\varepsilon,v_\varepsilon)-( u_0,0)\|_{C^2(D_L)}$ $<\varepsilon$.
    This proves $(\rm{2b_2})$.

Next, we prove the statement (3).
By \eqref{property-lambda-1} we have
\begin{align*}
\lambda_1(\beta,u_{0,\min})<0\quad\text{for} \quad\beta\in((u_0'')_{\max},\infty).
\end{align*}
Let $\beta\in((u_0'')_{\max},\infty)$. If
$L\in\left(0,{2\pi\over \sqrt{-\lambda_1(\beta,u_{0,\min})}}\right)$,
 then $\lambda_1(\beta,c)\geq\lambda_1(\beta,u_{0,\min})>-(2\pi/L)^2$ for all $c\leq u_{0,\min}$. If $L\in\left[{2\pi\over \sqrt{-\lambda_1(\beta,u_{0,\min})}},\infty\right)$, then there exists $c_{L}\leq u_{0,\min}$ such that $\lambda_1(\beta,c_{L})=-(2\pi/L)^2$. Based on the above spectral properties, the proofs of $({\rm 3a})$ and $({\rm 3b})$  are similar to those of $({\rm 2b_1})$ and $({\rm 2b_2})$, respectively.

 Let us now prove $(4)$ and $(5a)$.
 Suppose that $c$ is not a generalized inflection value of $u$. We show that $(u(x-ct,y),v(x-ct,y))$ is a shear flow when $\varepsilon_0>0$ is small enough.
  Note that $c$ is not a critical value of $u$ since $u$ is $C^2$ close to the monotone shear flow $u_0$.
  If  $c=u_{\max}$, by Remark \ref{supplement-to-Theorem-beta=0i} $(u(x-ct,y),v(x-ct,y))$ is a shear flow. If  $c\notin[c_\beta^+,u_{\min}]$, by Theorem \ref{classification-of-wave-speed-for-a-genuinely-travelling-wave-beta-plane} (i) $(u(x-ct,y),v(x-ct,y))$ is a shear flow.
 Now let $(u(x-ct,y),v(x-ct,y))$ be a traveling wave with $c\in[c_\beta^+,u_{\min}]$. We will show that it has to be a shear flow. We still have \eqref{uniformH1bound} for some $C_1>0$.
  We divide the discussions into two cases.

  In the case $c\leq u_{0,\min}$, we have
 \begin{align*}
 0=\int_{D_L}\left(|\nabla v|^2-{\beta-u_0''\over u_0-c}v^2\right)dxdy+\int_{D_L}\left({\beta-u_0''\over u_0-c}v^2-{\beta-\Delta u\over u-c}v^2\right)dxdy=I_0+II_0.
 \end{align*}
 For (4), we have $\lambda_1(\beta,c)\geq\inf_{\tilde c\leq u_{0,\min}}\lambda_1(\beta,\tilde c)\geq0>-\left({2\pi\over L}\right)^2$. For $(5a)$, we have
  $\lambda_1(\beta,c)\geq\inf_{\tilde c\leq u_{0,\min}}\lambda_1(\beta,\tilde c)>-\left({2\pi\over L}\right)^2$.
  Using \eqref{uniformH1bound}, we get
  \begin{align*}
  I_0=&L\sum_{k\neq0}\int_{-d}^d\left(|v_k'|^2+k^2\left({2\pi\over L}\right)^2|v_k|^2-{\beta-u_0''\over u_0-c}|v_k|^2\right)dy\\
  \geq&\left(\lambda_1(\beta,c)+\left({2\pi\over L}\right)^2\right)\|v\|_{L^2(D_L)}^2
  \geq{\inf_{\tilde c\leq u_{0,\min}}\lambda_1(\beta,\tilde c)+\left({2\pi\over L}\right)^2\over C_1^2}\|\nabla v\|_{L^2(D_L)}^2.
  \end{align*}
  Since $c\leq u_{\min}$ and $c\leq u_{0,\min}$,
  similarly to \eqref{II-estimate} we obtain
  \begin{align*}
-II_0=&\int_{D_L}\left(  {u_0''-\Delta u\over u-c}v^2+{(\beta- u_0'')(u_0-u)\over (u-c)(u_0-c)}v^2\right)dxdy
\leq C_2\|u_0-u\|_{C^2(D_L)}\|\nabla v\|_{L^2(D_L)}^2
\end{align*}
for some $C_2>0$. By the estimates for $I_0$ and $II_0$, we have
\begin{align*}
{\inf_{\tilde c\leq u_{\min}}\lambda_1(\beta,\tilde c)+\left({2\pi\over L}\right)^2\over C_1^2}\|\nabla v\|_{L^2(D_L)}^2\leq I_0=-II_0\leq C_2\|u_0-u\|_{C^2(D_L)}\|\nabla v\|_{L^2(D_L)}^2.
\end{align*}
Then $v=0$ if $\varepsilon_0< {\inf_{\tilde c\leq u_{\min}}\lambda_1(\beta,\tilde c)+\left({2\pi\over L}\right)^2\over C_1^2C_2}$.

In the other case, $c>u_{0,\min}$, we have $u_0(-d)=u_{0,\min}<c\leq u_{\min}=u(x_0,-d)$ for some $x_0\in\mathbb{T}_L$. Then
\begin{align}\label{c-u0min}
|c-u_{0,\min}|+|c-u_{\min}|=u_{\min}-u_{0,\min}=u(x_0,-d)-u_0(-d)\leq \|u-u_0\|_{C^0(D_L)}
\end{align}
and
 \begin{align*}
 0=\int_{D_L}\left(|\nabla v|^2-{\beta-u_0''\over u_0-u_{0,\min}}v^2\right)dxdy+\int_{D_L}\left({\beta-u_0''\over u_0-u_{0,\min}}v^2-{\beta-\Delta u\over u-c}v^2\right)dxdy=I_1+II_1.
 \end{align*}
 For both $(4)$ and $(5a)$, we have $\lambda_1(\beta,u_{0,\min})\geq\inf_{\tilde c\leq u_{0,\min}}\lambda_1(\beta,\tilde c)>-\left({2\pi\over L}\right)^2$.
 Then
 \begin{align*}
  I_1=&L\sum_{k\neq0}\int_{-d}^d\left(|v_k'|^2+k^2\left({2\pi\over L}\right)^2|v_k|^2-{\beta-u_0''\over u_0-u_{0,\min}}|v_k|^2\right)dy\\
  \geq&\left(\lambda_1(\beta,u_{0,\min})+\left({2\pi\over L}\right)^2\right)\|v\|_{L^2(D_L)}^2
  \geq{\inf_{\tilde c\leq u_{0,\min}}\lambda_1(\beta,\tilde c)+\left({2\pi\over L}\right)^2\over C_1^2}\|\nabla v\|_{L^2(D_L)}^2.
  \end{align*}
Using \eqref{estimate-uxy-c} and \eqref{c-u0min}, we obtain
\begin{align*}
&-II_1=\int_{D_L}{(\beta-\Delta u)(u_0-u_{0,\min})-(\beta-u_0'')(u-c)\over (u_0-u_{0,\min})(u-c)}v^2dxdy\\
\leq&\int_{D_L}{|((\beta-\Delta u)-(\beta-u_0''))(u_0-u_{0,\min})+(\beta-u_0'')((u_0-u_{0,\min})-(u-c))|\over (u_0-u_{0,\min})(u-c)}v^2dxdy\\
\leq&\int_{D_L}{|u_0''-\Delta u||u_0-u_{0,\min}|+|\beta-u_0''|(|u-u_0|+|c-u_{0,\min}|)\over (u_0-u_{0,\min})(u-c)}v^2dxdy\\
\leq&C\|u-u_0\|_{C^2(D_L)}\int_{D_L}{1\over y+d}v^2dxdy+C\|u-u_0\|_{C^0(D_L)}\int_{D_L}{1\over (y+d)^2}v^2dxdy\\
\leq&C\|u-u_0\|_{C^2(D_L)}\|v\|_{L^2(D_L)}\|\nabla v\|_{L^2(D_L)}+C\|u-u_0\|_{C^0(D_L)}\|\nabla v\|_{L^2(D_L)}^2\\
\leq&C_3\|u-u_0\|_{C^2(D_L)}\|\nabla v\|_{L^2(D_L)}^2
\end{align*}
for some $C_3>0$. By the estimates for $I_1$ and $II_1$, and choosing $\varepsilon_0< {\inf_{\tilde c\leq u_{\min}}\lambda_1(\beta,\tilde c)+\left({2\pi\over L}\right)^2\over C_1^2C_3}$,
we get $v=0$. This proves $(4)$ and $(5a)$.

Finally, we prove $(5b)$. By Lemma \ref{c-tou0min-limit-lem}, there exists $c_0\in(-\infty,u_{0,\min}]$ such that
$$\lambda_1(\beta,c_0;u_0)=\inf_{\tilde c\leq u_{0,\min}}\lambda_1(\beta,\tilde c;u_0)\,.$$ In the case
$L > {2\pi\over \sqrt{-\inf_{\tilde c\leq u_{0,\min}}\lambda_1(\beta,\tilde c;u_0)}}$, we have
$$\inf_{\tilde c\leq u_{0,\min}}\lambda_1(\beta,\tilde c;u_0)< -\left({2\pi\over L}\right)^2<0\,.$$
It then follows again from Lemma \ref{c-tou0min-limit-lem} that there exists $c_L\in(-\infty, c_0)$ such that $\lambda_1(\beta,c_L;u_0)=-\left({2\pi\over L}\right)^2$.
Thus $c_{L}$ is an isolated real eigenvalue of the operator $\mathcal{R}_{{2\pi\over L},\beta, u_0}$.
Since $u_0\in C^2([-d,d])$ lacks the regularity required for the bifurcation, by an argument similar to the first case of $(2b_2)$ we can perturb $(u_0,0)$ to a nearby shear flow $(\tilde u_0,0)$ with $C^4$ regularity, and then construct a genuine traveling-wave solution by bifurcating at $(\tilde u_0,0)$.

In the remaining case $L={2\pi\over \sqrt{-\inf_{\tilde c\leq u_{0,\min}}\lambda_1(\beta,\tilde c;u_0)}}$,
we have $-\left({2\pi\over L}\right)^2=\lambda_1(\beta,c_0;u_0)$. Let $a\in(0,1)$.  Then
$\lambda_1(\beta,c;au_0)=\lambda_1\left({\beta\over a},{c\over a};u_0\right)$ for $c\leq a u_{0,\min}$ and
$$-\left({2\pi\over L}\right)^2=\lambda_1(\beta,c_0;u_0)>\lambda_1\left({\beta\over a},c_0;u_0\right)=\lambda_1(\beta, ac_0;au_0)$$
due to \eqref{property-lambda-1}. This, along with Lemma \ref{c-tou0min-limit-lem}, implies that there exists $c_{L,a}<c_0$ such that
$$\lambda_1(\beta,ac_{L,a};au_0)=\lambda_1\left({\beta\over a},c_{L,a};u_0\right)=-\left({2\pi\over L}\right)^2\,.$$
Since $ac_{L,a}<ac_0\leq au_{0,\min}$, $ac_{L,a}$ is an isolated real eigenvalue of $\mathcal{R}_{{2\pi\over L},\beta, au_0}$. With these preparations, arguing as in the second case of $(2b_2)$, we can first perturb $(u_0,0)$ to the scaled shear flow $(au_0,0)$, then further perturb $(au_0,0)$ to $(\hat u_0,0)$ with $C^4$ regularity, to finally generate a  genuine traveling-wave solution by bifurcating from $(\hat u_0,0)$.

Let $\beta=0$. If $0\notin Ran(u_0'')$, then by taking $\varepsilon_0>0$ small enough, $\Delta u\neq0$ on $D_L$. By Theorem \ref{thm-generalization1} (ii), the same conclusion as in Theorem \ref{rigidity-near-monotone-shear-flow-arbitrary-wave-speed-thm} (i) holds. If $0\in Ran(u_0'')$,
then $0\in Ran(\Delta u)$ is possible and by Theorem \ref{classification-of-wave-speed-for-a-genuinely-travelling-wave-beta-plane}  the same conclusion as in Theorem \ref{rigidity-near-monotone-shear-flow-arbitrary-wave-speed-thm} (iii) holds.
\end{proof}

\begin{Example}
As an illustration of the physical relevance of Theorem \ref{rigidity-near-monotone-shear-flow-arbitrary-wave-speed-thm}, we consider Case 3 of Remark \ref{rigidity-near-monotone-shear-flow-arbitrary-wave-speed-thm-rem}. The data provided in Fig.\,12.5 of \cite{sa} shows the presence of steady waves (with zero speed) in the zonal band between 72$^\circ$N and 78$^\circ$N on Saturn, with an underlying shear flow $u_0(y)$ that increases poleward and has a strictly concave shape. For this zonal band we have $\beta>0>(u_0'')_{\max}$, and Theorem \ref{rigidity-near-monotone-shear-flow-arbitrary-wave-speed-thm} ensures that only  waves with sufficiently large wavelengths can arise as perturbations of this shear flow. The wavelength of the observed waves, which form Saturn's hexagon, is indeed about 14500 km (see \cite{cj2}).
\end{Example}

Applying Theorem \ref{rigidity-near-monotone-shear-flow-arbitrary-wave-speed-thm} to the Couette flow (i.e. $u_0(y)=y$, $y\in[-1,1]$) and using the fact that $\lambda_{1}^{-1}(0,-1)>0$ and $Ran(u_0'')=\{0\}$,
we get the following result.

\begin{Corollary}\label{8181}
Consider the Couette flow $u_0(y)=y$, $y\in[-1,1]$.
\begin{itemize}
\item[(i)] Let $\beta\in(0,\lambda_{1}^{-1}(0,-1)]$ and $L>0$ be arbitrary. Then the rigidity of traveling waves
with $c\in\mathbb{R}$ holds near the Couette flow in $C^2$, as described in Theorem \ref{rigidity-near-monotone-shear-flow-arbitrary-wave-speed-thm} {\rm(i)}.

\item[(ii)] Let $\beta\in(\lambda_{1}^{-1}(0,-1),\infty)$.

\begin{itemize}
\item[{$\rm(ii_{a})$}] If $L\in\left(0,\frac{2\pi}{\sqrt{-\lambda_{1}(\beta,-1)}}\right)$, then the same conclusion
as in Corollary \ref{8181} {\rm(i)} holds.

\item[{$\rm(ii_{b})$}] If $L\in\left[\frac{2\pi}{\sqrt{-\lambda_{1}(\beta,-1)}},\infty\right)$, then
genuine traveling waves exist near the Couette flow in $C^2$, as described in Theorem \ref{rigidity-near-monotone-shear-flow-arbitrary-wave-speed-thm} {$\rm(ii)$}.
\end{itemize}
\end{itemize}
\end{Corollary}

\begin{Remark}
In the  $\beta$-plane setting, compared to Theorem 1.3 in \cite{WZZ}, Corollary \ref{8181} improves the regularity requirement, reducing
it from $H^{\geq5}$ to $C^2$. In the  $f$-plane setting, we mention that the rigidity of traveling waves
with $c\in\mathbb{R}$ holds for $L>0$ near the Couette flow in $H^{>{5\over2}}$ \cite{LZ}.
\end{Remark}

\begin{Remark}\label{numerical-computation-transtional-beta-value}
The cases 1-3 in Remark \ref{rigidity-near-monotone-shear-flow-arbitrary-wave-speed-thm-rem} show that when $\lambda_1((u_0'')_{\max},u_{0,\min})>0$, then the value $\lambda_1^{-1}(0,u_{0,\min})$
serves as the transitional $\beta$-parameter for the onset of genuine traveling waves near the monotone shear flow $(u_0,0)$ in $C^2$.
More precisely, if $\beta\leq\lambda_1^{-1}(0,u_{0,\min})$,
the rigidity of  nearby  traveling waves  holds for all wave speeds  and for all zonal periods,  while if $\beta>\lambda_1^{-1}(0,u_{0,\min})$, there exist genuine nearby traveling waves for some large wavelengths.

We use Matlab to numerically compute the transitional $\beta$-parameter $\lambda_1^{-1}(0,u_{0,\min})$  for monotone shear flows with either linear profiles or concave parabolic profiles.
These computations illustrate how the transitional
$\beta$-value varies as the zonal-band width increases or as the background shear intensifies.
To emphasize the dependence on both the shear profile and the zonal-band width, we rewrite the transitional parameter as $\beta_{\rm{crit}}(u_0, d)=\lambda_1^{-1}(0,u_{0,\min})$.
\bigskip

\noindent
{\bf{Linearly sheared profiles:}}
We first compute the transitional $\beta$-parameter for the Couette flow $u_0(y)=y$ confined to the band $[-1,1]$.  The corresponding singular
Rayleigh-Kuo  boundary value problem \eqref{sturm-Liouville} is
\begin{align}\label{sturm-Liouville-Couette}
-\phi''-{\beta\over y+1}\phi=\lambda\phi, \;\;\;\;\phi(\pm 1)=0.
\end{align}
By numerical computation, we have $\lambda_1(1.8352,-1;y, 1)\approx0$ and thus \begin{align*}
\beta_{\rm{crit}}(y, 1)\approx1.8352.\end{align*}

For a general linearly sheared flow $u_0(y)=ay+b$ with $\pm a>0$ confined to a band $[-d,d]$, the corresponding singular
Rayleigh-Kuo  boundary value problem \eqref{sturm-Liouville} is
\begin{align}\label{sturm-Liouville-linear}
-\phi''-{\beta\over ay\pm ad}\phi=\lambda\phi, \;\;\;\;\phi(\pm d)=0.
\end{align}
By performing the scaling $\phi(y)=\varphi(z)$ with $y=\pm dz$, \eqref{sturm-Liouville-linear} is transformed to
\begin{align}\label{sturm-Liouville-linear-tran}
-\varphi''-{\pm{d\over a}\beta\over z+1}\varphi=d^2\lambda\varphi, \;\;\;\;\varphi(\pm 1)=0.
\end{align}
By comparing \eqref{sturm-Liouville-linear-tran} with \eqref{sturm-Liouville-Couette}, we obtain
\begin{align}\label{crit-beta-linear}
\beta_{\rm{crit}}(ay+b, d)=\pm{a\over d}\beta_{\rm{crit}}(y, 1)\approx\pm{a\over d}\cdot 1.8352\quad {\text for} \quad \pm a>0.
\end{align}
That is, the transitional $\beta$-parameter for the  linearly sheared flow $u_0(y)=ay+b$ $(\pm a>0)$ on $[-d,d]$ is  a scaled version of that for the Couette flow on $[-1,1]$. The relation \eqref{crit-beta-linear} explains the following trends for  linear profiles:
\begin{itemize}
\item[(i)] within a fixed zonal band, the stronger the shear effect is, the larger is the range of $\beta$ for which the rigidity of traveling waves persists  for all wave speeds  and for all zonal periods;
    \item[(ii)] for a fixed linearly sheared profile, the wider the zonal band is, the larger is the admissible range of $\beta$ that allows the existence of genuine traveling waves.
\end{itemize}

Note that for a general linearly sheared flow $u_0(y)=ay+b$ ($a\neq0$) we have $Ran(u_0'')=\{0\}$ and $\lambda_1((u_0'')_{\max},u_{0,\min})=\lambda_1(0,u_{0,\min})>0$ by Lemma \ref{non-negativeness-beta0}. Consequently, the first subcase of Case 2 in Remark \ref{rigidity-near-monotone-shear-flow-arbitrary-wave-speed-thm-rem} always holds.
As $\beta$ exceeds the transitional value, a boundary curve emerges in the $(\beta,L)$-plane, given by
${2\pi\over L}=\sqrt{-\lambda_1(\beta,u_{0,\min})}$, as depicted in Fig.\,\ref{rigidity-existence-tw} (a).
 This curve marks the critical wavelength for any fixed $\beta$ that exceeds the transitional value:
  absence (existence) of genuine traveling waves holds for smaller (larger) wavelength, near $(u_0,0)$ in $C^2$.
 For a specific linearly sheared profile $u_0(y)=2y+3$ confined to the zonal band $[-d,d]$, we numerically compute the critical wavelength  ${2\pi\over L}=\sqrt{-\lambda_1(\beta,u_{0,\min})}, \beta>\beta_{\rm{crit}}(u_0, d)$ and simulate the boundary curves for $d=1, 2, 3$. The resulting computations and observations are summarized in Fig.\,\ref{boundary-curve-linear}.
 \begin{figure}[h]
    \centering
	\includegraphics[scale = 0.27]{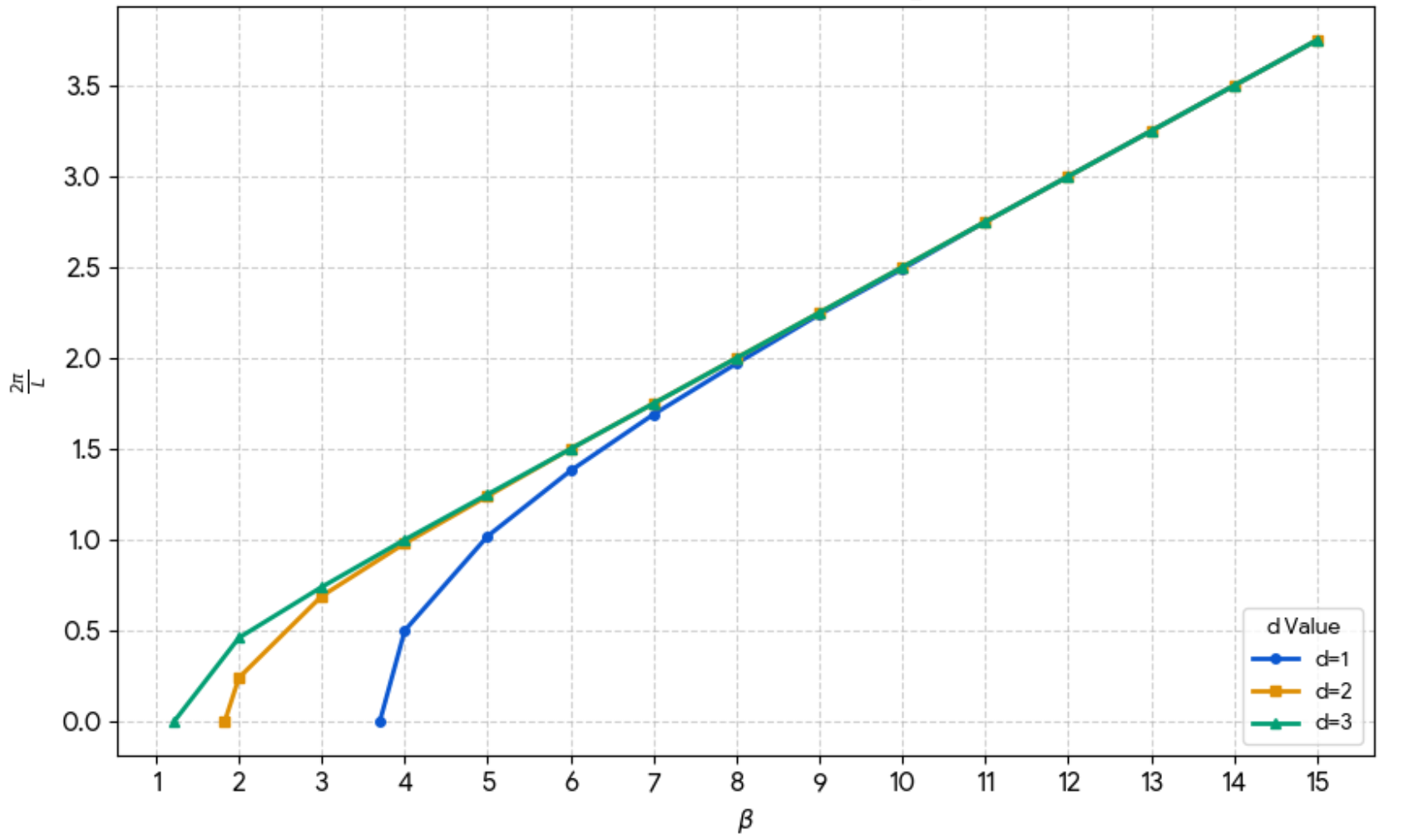}
	\caption{{\footnotesize For values of
$\beta$ slightly above the transitional threshold, the critical wavelength depends sensitively on the zonal-band width, leading to noticeable variation among different widths. As
 $\beta$ increases further, these critical wavelengths become gradually close to one another.}}
	\label{boundary-curve-linear}
\end{figure}
\bigskip

\noindent
{\bf{Concave parabolic profiles:}}
For the concave parabolic profiles $u_0(y)=-y^2+by+e$ confined to the zonal band $[-d,d]$, we impose the condition $d<{b\over 2}$ to ensure that $u_0$ is  monotone  on this band.
The corresponding
Rayleigh-Kuo  boundary value problem \eqref{sturm-Liouville} is
\begin{align}\label{sturm-Liouville-Concave-parabolic}
-\phi''+{\beta+2\over (y+d)(y-d-b)}\phi=\lambda\phi, \;\;\;\;\phi(\pm d)=0.
\end{align}
Note that in a fixed zonal band, increasing
$b$ displaces the parabola horizontally to the right while maintaining its shape, resulting in a stronger shear effect. Numerical computation of the eigenvalues of \eqref{sturm-Liouville-Concave-parabolic}
 yields the transitional $\beta$-values listed in Table  \ref{numerical-computation-concave-parabolic}, which exhibit trends (i)-(ii) comparable to those observed for the linear profiles.

\begin{table}[ht]
  \centering
\begin{tabular}{|c|c|c|c|}
\hline
\diagbox{$b$}{$d$} & $1$  & $2$ & $3 $ \\
\hline
 $7$& $13.2496$ & $6.7922$& $4.6236$\\
\hline
 $8$& $15.0898$ & $7.7176$ &$5.2450$\\
\hline
$9$ &
$16.9289$
 & $8.6416$ &$5.8648$\\
\hline
\end{tabular}
\vspace{0.2cm}
\caption{Numerical computation of  $\beta_{crit}(-y^2+by+e,d)$.}
  \label{numerical-computation-concave-parabolic}
\end{table}
\end{Remark}
\begin{Example}
We use the numerical computation of the transitional $\beta$-value in Remark \ref{numerical-computation-transtional-beta-value} to explain the lack of waves in Jupiter's zonal band between 37$^\circ$S and 39$^\circ$S that was noted in \cite{or}. Indeed, the profile $u_0$ of the zonal flow in this band is linearly sheared, featuring an eastward jet with speed 45 $m/s$ along the northern boundary and a westward jet with speed 5 $m/s$ along the southern boundary (see Fig.\,4 in \cite{Hueso2023} for accurate data). By \eqref{def-f0-beta},
a bandwidth of two degrees of latitude  corresponds to $2d={\pi\over 90}\approx 0.035$.
Using the values listed in Table \ref{v-par-JS} together with the dimensionless velocity expression (2.6) in \cite{csz2024a}, we obtain the boundary values $u_0(-d)=-{1\over30}$ and $u_0(d)={3\over10}$, yielding the linear profile $u_0(y)=ay+b$ with $a={1\over 6d}$ and $b={2\over 15}$, which describes the observed flow pattern.  By   \eqref{def-f0-beta} and \eqref{crit-beta-linear} we have
\begin{align*}
\beta={2\times 1.76\times 10^{-4}\times 69911\times 10^3\over 150} \cos(38^\circ)\approx 129<
\beta_{\rm{crit}}(ay+b, d)
\approx1004.
\end{align*}
\end{Example}
Under the Rayleigh stability condition, $C^2$ shear flows (not necessarily monotone) turn out to be nonlinearly Lyapunov stable (see Theorem \ref{nonlinear-Lyapunov-stable-not-asymptotically-stable-monotone-shear-flow} below). However, due to the existence of  genuine nearby traveling waves, asymptotic stability (i.e. nonlinear inviscid damping) fails near a monotone shear flow $(u_0(y),0)$ for $(\beta,L)\in\bigg(\max\bigg\{(u_0'')_{\max},$ $ \lambda_1^{-1}(0,u_{0,\min})\bigg\},\infty\bigg)\times\left[{2\pi\over \sqrt{-\lambda_1(\beta,u_{0,\min})}},\infty\right)$. This means that in such case,  the long-time dynamics near the monotone shear flow is considerably richer.
The rigidity near the monotone shear flow suggests that asymptotic stability may potentially be proved in certain function spaces for other values of $(\beta,L)$. In the  $f$-plane setting, it is worth noting that nonlinear inviscid damping in certain Gevrey spaces has been proved near a class of monotone shear flows in \cite{Ionescu-Jia,MZ}, where the Rayleigh stability condition does not hold since the shear flow $u_0$ satisfies  $u_0''=0$  near the boundary and thus $0=\beta\in Ran(u_0'')$.

\begin{Theorem}\label{nonlinear-Lyapunov-stable-not-asymptotically-stable-monotone-shear-flow}
Assume that $u_0\in C^2([-d,d])$, $u_0'\neq0$ on $[-d,d]$,
and $(\beta,L)\in(\max\{(u_0'')_{\max}, $ $ \lambda_1^{-1}(0,u_{0,\min})\},\infty)\times\left[{2\pi\over \sqrt{-\lambda_1(\beta,u_{0,\min})}},\infty\right)$.
Then $(u_0, 0)$ is nonlinearly Lyapunov stable in (vorticity) $L^2$, but nonlinear inviscid damping fails near $(u_0, 0)$ in  $C^2$.
\end{Theorem}
\begin{proof}
First, we prove that $(u_0, 0)$ is nonlinearly Lyapunov stable in (vorticity) $L^2$ as long as $\beta\notin Ran(u_0'')$. The proof of  nonlinear Lyapunov stability is valid for general $C^2$ shear flows, not necessarily the monotone ones. We use the standard method introduced by Arnold \cite{Arnold1965,Arnold1969}. Without loss of generality, we assume that $\beta>(u_0'')_{\max}$.
Let $c<u_{0,\min}$. We define $f_0=(\psi_0+cy)\circ(\gamma_0+\beta y)^{-1}$ on $Ran(\gamma_0+\beta y)$, where $\psi_0$ and $\gamma_0$ are the stream function and vorticity of $(u_0,0)$. Then $f_0(\gamma_0+\beta y)=\psi_0+cy$ and $-f_0'(\gamma_0+\beta y)$ has uniform upper and lower positive bounds, that is,
\begin{align*}
0<{u_{0,\min}-c\over \beta-(u_0'')_{\min}}\leq-f_0'(\gamma_0+\beta y)={u_0-c\over \beta-u_0''}\leq{u_{0,\max}-c\over \beta-(u_0'')_{\max}}.
\end{align*}
We now extend $f_0$ to $f\in C^1(\mathbb{R})$ such that $f=f_0$ on $Ran(\gamma_0+\beta y)$ and
\begin{align}\label{f-derivative-bound}
0<{1\over2}{u_{0,\min}-c\over \beta-(u_0'')_{\min}}\leq-f'\leq{2(u_{0,\max}-c)\over \beta-(u_0'')_{\max}}.
\end{align}
We choose $F$ such that $F'=f$. Define the energy-Casimir-momentum functional by
\begin{align*}
H(\gamma)=\int_{D_L}\left(-{1\over2}|\vec{u}|^2-F(\gamma+\beta y)+cy(\gamma+\beta y)\right) dxdy
\end{align*}
along a perturbed solution $\vec{u}=(u,v)$ of the $\beta$-plane equation.
Here, $\gamma=\partial_x v-\partial_y u=\Delta\psi$ and the boundary conditions are $\psi|_{y=\pm d}= B_{\pm}$ (with constants only  depending on time). Without loss of generality, we can take the circulation (which is conserved) of the perturbed solution $\psi$ satisfying
\begin{align}\label{circulation}
\int_{\mathbb{T}_L}\partial_y\psi(x,\pm d)dx=\int_{\mathbb{T}_L}\psi_0'(\pm d)dx=-Lu_0(\pm d)
\end{align}
(i.e. the circulation of the perturbed solution is the same as that of the background shear flow).
Indeed, if the circulation of the perturbed solution is different from that of the shear flow, then we can first modify the shear flow to a nearby shear flow such that the circulation of the modified shear flow is the same as that of the perturbed solution.
Then, by an argument analogous to the one used below for nonlinear stability, we can show that the perturbed solution remains
$L^2$-close (in vorticity) to the modified shear flow for all $t>0$. Since the modified shear flow is $L^2$-close (in vorticity) to the background shear flow, we conclude that the perturbed solution remains
$L^2$-close (in vorticity) to the background shear flow for all times $t>0$.

Note that the functional
$H(\gamma)$ is conserved
along the solution $\vec{u}=(u,v)$.
In fact, by  \eqref{Euler equation} we have
\begin{align*}
{d\over dt}{1\over2}\int_{D_L}|\vec u|^2dxdy=&-\int_{D_L}\vec u\cdot((\vec{u}\cdot\nabla)\vec{u}+\nabla P+\beta yJ\vec
{u})dxdy\\
=&-\int_{D_L}\left(\vec u\cdot\nabla\left({|\vec{u}|^2\over2}\right)+\vec u\cdot\nabla P+\beta y\vec u\cdot J\vec
{u}\right)dxdy\\
=&\int_{D_L}\left((\nabla\cdot\vec u){|\vec{u}|^2\over2}+(\nabla\cdot\vec u) P\right)dxdy=0,
\end{align*}
where we have used $\nabla\cdot\vec u=0$ and $\vec u\cdot J\vec
{u}=0$. Thus, the kinetic energy ${1\over2}\int_{D_L}|\vec u|^2dxdy$ is conserved along the solution $\vec{u}$.

Let $(X(s;x,y), Y(s;x,y))$ be the particle trajectory along the velocity $\vec u$, that is, $(X(s;x,y),$ $ Y(s;x,y))$ solves the equation
\begin{align*}
\begin{cases}
\dot{X}(s)=u(X(s), Y(s)),\\
\dot{Y}(s)=v(X(s), Y(s)),
\end{cases}
\end{align*}
with the initial data $X(0)=x,Y(0)=y$. Note that \eqref{tvore} expresses the conservation of total vorticity along a fluid trajectory and the flow map $(X, Y)$ is area preserving, that is,
\begin{align*}
{d\over dt}(\gamma(t,X(t),Y(t))+\beta Y(t))=0\quad\text{and} \quad {\partial{(X,Y)}\over \partial{(x,y)}}=1.
\end{align*}
Then
\begin{align*}
&{d\over dt}\int_{D_L}F(\gamma+\beta y)dxdy\\
=&\int_{D_L}F'(\gamma(t,X(t),Y(t))+\beta Y(t)){d\over dt}(\gamma(t,X(t),Y(t))+\beta Y(t))dXdY=0\,,
\end{align*}
showing that
the Casimir functional
$\int_{D_L}F(\gamma+\beta y) dxdy$
is conserved along the solution $\vec{u}$.

By \eqref{vore} we have
\begin{align}\label{4231}
&{d\over dt}\int_{D_L}y(\gamma+\beta y)dxdy\\
\nonumber=&\int_{D_L}y\partial_{t}(\gamma+\beta y)dxdy\\
\nonumber=&\int_{D_L}-\left(y(\vec{u}\cdot\nabla)\gamma+y\beta\partial_{x}\psi\right) dxdy\\
\nonumber=&\int_{D_L}-y\nabla\cdot(\vec{u}\gamma) dxdy+\int_{-d}^dy\beta\psi\bigg|_{x=0}^{L}dy\\
\nonumber=&\int_{D_L}-y\partial_{y}(v\gamma)dxdy,
\end{align}
where we have used the fact that both $\gamma$ and $\psi$ are periodic in $x$.
Then \eqref{4231} becomes
\begin{align*}
&{d\over dt}\int_{D_L}y(\gamma+\beta y)dxdy=-\int_{\mathbb{T}_L}y(v\gamma)\bigg|_{y=-d}^{d}dx+\int_{\mathbb{T}_L}\int_{-d}^dv\gamma dydx
\\=&\int_{D_L}\partial_{x}\psi(\partial_{xx}\psi+\partial_{yy}\psi)dydx\\
=&\int_{D_L}\frac{1}{2}\partial_{x}(\partial_{x}\psi)^2dydx
-\int_{D_L}\frac{1}{2}\partial_{x}(\partial_{y}\psi)^2dydx=0,
\end{align*}
where we have used  $v(x,d)=v(x,-d)=0$. Thus, the momentum functional
$\int_{D_L}cy(\gamma+\beta y) dxdy
$
is conserved along the solution $\vec{u}$. This proves that
$H(\gamma)$
is conserved along the solution $\vec{u}$.

 Moreover, $H'(\gamma_0)=\psi_0+cy-f(\gamma_0+\beta y)=0$. By $\psi|_{y=\pm d}= B_{\pm}$ and \eqref{circulation}, we have
 \begin{align*}
 &\int_{D_L}\left(-{1\over2}|\vec{u}|^2+{1\over2}|\vec{u}_0|^2\right)dxdy\\
 =&{1\over2}\int_{D_L}\left(\psi\gamma-\psi_0\gamma_0\right)dxdy-{1\over2}\left(\psi|_{y= d}\int_{\mathbb{T}_L}\psi_y(x,d)
 dx-\psi|_{y=-d}\int_{\mathbb{T}_L}\psi_y(x,-d)dx\right)\\
 &+{1\over2}\left(\psi_0(d)\int_{\mathbb{T}_L}\psi_{0y}(d)dx
 -\psi_0(-d)\int_{\mathbb{T}_L}\psi_{0y}(-d)dx\right)\\
 =&{1\over2}\int_{D_L}\left(\psi\gamma-\psi_0\gamma_0\right)dxdy
 -{1\over2}\int_{D_L}\psi\gamma_0dxdy+\bigg({1\over2}\int_{D_L}\psi\gamma_0dxdy\\
 &-{1\over2}\left(\psi|_{y= d}\int_{\mathbb{T}_L}\psi_{0y}(d)
 dx-\psi|_{y=-d}\int_{\mathbb{T}_L}\psi_{0y}(-d)dx\right)\\
 &+{1\over2}\left(\psi_0(d)\int_{\mathbb{T}_L}\psi_{y}(x,d)dx
 -\psi_0(-d)\int_{\mathbb{T}_L}\psi_{y}(x,-d)dx\right)\bigg)\\
 =&{1\over2}\int_{D_L}\left(\psi\gamma-\psi_0\gamma_0\right)dxdy
 -{1\over2}\int_{D_L}\psi\gamma_0dxdy+{1\over2}\int_{D_L}\psi_0\gamma dxdy\\
 =&\int_{D_L}\left({1\over2}(\psi-\psi_0)(\gamma-\gamma_0)+
 \psi_0(\gamma-\gamma_0)\right)dxdy
 \end{align*}
 and thus
\begin{align*}
&H(\gamma)-H(\gamma_0)\\
=&\int_{D_L}\left(-{1\over2}|\vec{u}|^2+{1\over2}|\vec{u}_0|^2-(F(\gamma+\beta y)-F(\gamma_0+\beta y))+cy(\gamma-\gamma_0)\right)dxdy\\
=&\int_{D_L}\bigg({1\over2}(\psi-\psi_0)(\gamma-\gamma_0)+(\psi_0+cy)(\gamma-\gamma_0)
-(F(\gamma+\beta y)-F(\gamma_0+\beta y))\bigg)dxdy\\
=&\int_{D_L}-{1\over2}|\nabla(\psi-\psi_0)|^2dxdy\\
&+\int_{D_L}
\bigg(-(F(\gamma+\beta y)-F(\gamma_0+\beta y)-F'(\gamma_0+\beta y)(\gamma-\gamma_0))\bigg)dxdy\\
=&I+II.
\end{align*}
By $-F''=-f'$ and \eqref{f-derivative-bound}, we have
\begin{align*}
{1\over4}{u_{0,\min}-c\over \beta-(u_0'')_{\min}}\|\gamma-\gamma_0\|_{L^2(D_L)}^2\leq II\leq {u_{0,\max}-c\over \beta-(u_0'')_{\max}}\|\gamma-\gamma_0\|_{L^2(D_L)}^2.
\end{align*}
By the Poincar\'{e} inequality, we have $\|\nabla(\psi-\psi_0)\|_{L^2(D_L)}^2=-\int_{D_L}(\psi-\psi_0)(\gamma-\gamma_0)dxdy\leq C_0 \|\nabla(\psi-\psi_0)\|_{L^2(D_L)}\|\gamma-\gamma_0\|_{L^2(D_L)}$, and thus $\|\nabla(\psi-\psi_0)\|_{L^2(D_L)}\leq C_0 \|\gamma-\gamma_0\|_{L^2(D_L)}$. Then
\begin{align*}
I\geq -{1\over2} C_0^2\|\gamma-\gamma_0\|_{L^2(D_L)}^2.
\end{align*}
Choose $c<u_{0,\min}$ such that
\begin{align*}
{u_{0,\min}-c\over 2(\beta-(u_0'')_{\min})}>C_0^2.
\end{align*}
The estimates for $I$ and $II$ yield
\begin{align*}
{u_{0,\max}-c\over \beta-(u_0'')_{\max}}\|\gamma(0)-\gamma_0\|_{L^2(D_L)}^2\geq&
H(\gamma(0))-H(\gamma_0)=H(\gamma(t))-H(\gamma_0)\\
\geq&{1\over2}\left({u_{0,\min}-c\over 2(\beta-(u_0'')_{\min})}- C_0^2\right)\|\gamma(t)-\gamma_0\|_{L^2(D_L)}^2.
\end{align*}
This proves that
\begin{align*}
\|\gamma(t)-\gamma_0\|_{L^2(D_L)}^2\leq
{{2(u_{0,\max}-c)/ (\beta-(u_0'')_{\max})}
\over{(u_{0,\min}-c)/ (2\beta-2(u_0'')_{\min})}- C_0^2}\|\gamma(0)-\gamma_0\|_{L^2(D_L)}^2\,,\qquad t >0\,.
\end{align*}

By Theorem \ref{rigidity-near-monotone-shear-flow-arbitrary-wave-speed-thm}, genuine traveling waves exist near the monotone shear flow $(u_0,0)$ in $C^2$, and thus
 nonlinear inviscid damping near $(u_0,0)$ in  $C^2$ fails.
\end{proof}

\section{Applications to the Couette-Poiseuille flow, to the Bickley jet and to the Kolmogorov flow}

In this section, building on the results of Section 3, we investigate the rigidity of traveling waves  near non-monotone shear flows which include the Couette-Poiseuille flow and the Bickley jet in a bounded channel. Moreover, we study
the rigidity of traveling waves with a fixed wave speed near a class of shear flows and describe the inviscid dynamical structures near the Kolmogorov flow in the  $f$-plane setting.

\subsection{Rigidity near the Couette-Poiseuille flow and near the Bickley jet}\label{cp-b}
Based on Theorem  \ref{thm-generalization1}, we present some rigidity results for geophysical  traveling waves with arbitrary speeds near the  Couette-Poiseuille flow and near the Bickley jet. We conclude with a generalization to a  class of shear flows.

\subsubsection{Application to the Couette-Poiseuille flow}
We consider the generalized Couette-Poiseuille flow, given for $\Gamma\in\mathbb{R}$ by
\begin{equation}\label{gcp}
u_{\rm cp}(y)=\Gamma y+(1-\Gamma)y^2, \quad y\in[-1,1]\,.
\end{equation}

\begin{Proposition}\label{Couette-Poiseuille-beta-plane}
 Assume that $\Gamma<1$ and $\beta\in\left[0,2-2\Gamma\right)$.
If $(u(x-ct,y),v(x-ct,y))\in C^2(D_L)$ is a traveling-wave
solution to the $\beta$-plane equation \eqref{Euler equation}-\eqref{boundary condition for euler} with $c\in \mathbb{R}$ and such that
\begin{align}\label{711}
\|\Delta u-u_{\rm cp}''\|_{C^0(D_L)}<2-2\Gamma-\beta,
\end{align}
then $(u(x-ct,y),v(x-ct,y))$ is  a shear flow in $D_L=\mathbb{T}_L\times [-d,d]$ with $d=1$.
\end{Proposition}

\begin{proof}
 It follows from \eqref{711} that
\begin{align*}
|(\Delta u)(x,y)-(2-2\Gamma)|\leq\|\Delta u-u_{\rm cp}''\|_{C^0(D_L)}<2-2\Gamma-\beta,\quad (x,y)\in D_L.
\end{align*}
Then
\begin{align*}
\beta-(2-2\Gamma)<(\Delta u)(x,y)-(2-2\Gamma),\quad(x,y)\in D_L.
\end{align*}
Thus
\begin{align*}
0\leq\beta<(\Delta u)_{\min}.
\end{align*}
By Theorem \ref{thm-generalization1}, $(u(x-ct,y),v(x-ct,y))$ is a shear flow.
\end{proof}
\begin{Remark}
For the Poiseuille flow, corresponding to $\Gamma=0$ in \eqref{gcp}, in the $f$-plane setting,  Proposition \ref{Couette-Poiseuille-beta-plane} complements the results in \cite{Coti Zelati23,dn2024}. Indeed, the rigidity of traveling waves with arbitrary wave speeds  was obtained in \cite{Coti Zelati23} under the high regularity assumption (velocity) $H^{>6}$, and Proposition \ref{Couette-Poiseuille-beta-plane} lowers the regularity  to $C^2$.
In addition, although the rigidity of steady flows near the Poiseuille flow was proved in
 $C^2$ (see \cite{dn2024}), Proposition \ref{Couette-Poiseuille-beta-plane} further extends the conclusion to encompass all traveling wave speeds. See also the recent independent result for $\beta=0$ in Corollary 1.6(i) of \cite{gxx2024}.
\end{Remark}

\subsubsection{Application to the Bickley jet}
The profile of a bounded Bickley jet is
\begin{align*}
u_{\rm b}(y)=-\operatorname{sech}^{2}(y),\quad y\in[-d,d].
\end{align*}
As another application of Theorem  \ref{thm-generalization1}, we investigate the rigidity of traveling waves with arbitrary speeds near the Bickley jet.

\begin{Proposition}\label{bickley-beta-plane}
Assume that $d\in\left(0,\frac{\ln(2+\sqrt{3})}{2}\right)$ and $\beta\in\left[0,2\operatorname{sech}^{4}(d)-4\operatorname{sech}^{2}(d)\operatorname{tanh}^{2}(d)\right)$.
If $(u(x-ct,y),v(x-ct,y))\in C^2(D_L)$ is a traveling-wave
solution to the $\beta$-plane equation \eqref{Euler equation}-\eqref{boundary condition for euler} with $c\in \mathbb{R}$,
and satisfies
\begin{align}\label{771}
\|\Delta u-u_{\rm b}''\|_{C^0(D_L)}<2\operatorname{sech}^{4}(d)-4\operatorname{sech}^{2}(d)\operatorname{tanh}^{2}(d)-\beta,
\end{align}
then $(u(x-ct,y),v(x-ct,y))$ is a shear flow.
\end{Proposition}
\begin{Remark}
 The real zeros of $u_{\rm b}''(y)=2\operatorname{sech}^{4}(y)-4\operatorname{sech}^{2}(y)\operatorname{tanh}^{2}(y)=0$ are $\frac{\ln(2+\sqrt{3})}{2}$ and $\frac{\ln(2-\sqrt{3})}{2}$. Moreover, $u_{\rm b}''>0$ on $\left(\frac{\ln(2-\sqrt{3})}{2},\frac{\ln(2+\sqrt{3})}{2}\right)$.
 If $d\in\left(0,\frac{\ln(2+\sqrt{3})}{2}\right)$, then
 \begin{align}\label{u-b-max}
 2\operatorname{sech}^{4}(d)-4\operatorname{sech}^{2}(d)\operatorname{tanh}^{2}(d)=u_{\rm b}''(\pm d)=(u_{\rm b}'')_{\min}>0.
 \end{align}
 Thus the range of $\beta$ in Proposition $\ref{bickley-beta-plane}$ is non-null.
\end{Remark}
\begin{proof}
By \eqref{771}  we have
\begin{align*}
|(\Delta u)(x,y)-u_{\rm b}''(y)|<2\operatorname{sech}^{4}(d)-4\operatorname{sech}^{2}(d)\operatorname{tanh}^{2}(d)-\beta,\quad (x,y)\in D_L.
\end{align*}
Thus
\begin{align*}
\beta-2\operatorname{sech}^{4}(d)+4\operatorname{sech}^{2}(d)\operatorname{tanh}^{2}(d)<(\Delta u)(x,y)-u_{\rm b}''(y)\leq (\Delta u)(x,y)-(u_{\rm b}'')_{\min}
\end{align*}
for $(x,y)\in D_{L}$. By \eqref{u-b-max} we have
\begin{align*}
0\leq\beta<(\Delta u)_{\min}.
\end{align*}
By Theorem \ref{thm-generalization1}, $(u(x-ct,y),v(x-ct,y))$ is a shear flow.
\end{proof}

\subsubsection{Generalization to rigidity near a class of shear flows} By inspection to the proof of Proposition
\ref{Couette-Poiseuille-beta-plane} and of Proposition \ref{bickley-beta-plane}, we can similarly generalize the rigidity results to a class of shear flows.

\begin{Proposition} \label{beta-plane-general}
Assume that $u_0\in C^2([-d,d])$,  $(u_0'')_{\min}>0$ and $\beta\in\left[0,(u_0'')_{\min}\right)$.
If $(u(x-ct,y),v(x-ct,y))\in C^2(D_L)$ is a traveling-wave
solution to the $\beta$-plane equation \eqref{Euler equation}-\eqref{boundary condition for euler} with $c\in \mathbb{R}$,
and satisfies
\begin{align}\label{861}
\|\Delta u-u_{ 0}''\|_{C^0(D_L)}<(u_0'')_{\min}-\beta,
\end{align}
then $(u(x-ct,y),v(x-ct,y))$ is  a shear flow.
\end{Proposition}

\begin{proof}
It follows from \eqref{861} that
\begin{align*}
|(\Delta u)(x,y)-u_{0}''(y)|
<(u_{0}'')_{\min}-\beta,\quad (x,y)\in D_L.
\end{align*}
Thus
\begin{align*}
\beta-(u_{0}'')_{\min}<(\Delta u)(x,y)-u_{0}''(y)\leq (\Delta u)(x,y)-(u_{0}'')_{\min},\quad(x,y)\in D_{L}.
\end{align*}
Hence $0\leq\beta<(\Delta u)_{\min}$. By Theorem \ref{thm-generalization1}, $(u(x-ct,y),v(x-ct,y))$ is a shear flow.
\end{proof}

\begin{Example}
Saturn's southern polar troposphere features at cloud level a strongly sheared  circumpolar jet flowing eastward in the zonal band between 65$^\circ$S and 70$^\circ$S.
 The  jet profile exhibits a nearly flat slope at 65$^\circ$S,
 with a poleward speed increasing from about $0\; m/s$ at 65$^\circ$S to a maximal speed of $100\; m/s$ at 70$^\circ$S (see Fig.\,12.5 in \cite{sa}). With the scaling parameters provided in Table \ref{v-par-JS}, we obtain the boundary values $u_0(-d)={2\over3}$, $u_0(d)=0$ and $u'(d)=0$. Consequently, we may model this sheared, purely zonal flow by the parabolic profile
\begin{align*}
u_0(y)={1\over 6d^2}y^2-{1\over 3d}y+{1\over6},\quad -d\leq y\leq d,
\end{align*}
in a band of width $2d={\pi\over 36}\approx0.0873$ corresponding to about $5$ degrees of latitude. Using the data in  Table \ref{v-par-JS}, we infer from  \eqref{def-f0-beta} that
\begin{align*}
\beta={2\times 1.62\times 10^{-4}\times 58232\times 10^3\over 150}\cos(68.5^\circ)\approx46<(u_0'')_{\min}={1\over 3d^2}\approx175
\end{align*}
so that Proposition \ref{beta-plane-general} predicts the absence of traveling waves near this zonal flow. This conclusion is validated by the field data (see \cite{sa}).
\end{Example}

\subsection{Application to the $f$-plane setting}

In this subsection, as an application of Theorem \ref{classification-of-wave-speed-for-a-genuinely-travelling-wave-beta-plane},
we study the rigidity of traveling waves in the $f$-plane setting. Moreover,
we prove some rigidity results for traveling waves with non-vanishing speeds near a Kolmogorov flow in the $f$-plane setting. On the other hand, we construct genuine analytic traveling waves with vanishing speed (i.e. steady flow) near the Kolmogorov flow.

\subsubsection{Rigidity of traveling waves in the $f$-plane setting}
The following rigidity result in the $f$-plane setting is a  consequence of Theorem \ref{classification-of-wave-speed-for-a-genuinely-travelling-wave-beta-plane}.

\begin{Corollary}\label{beta=0-cor-c-in-R}
Let $\beta=0$ and $c\in \mathbb{R}$. Assume that $\Delta u\neq0$ whenever $u=c$. If $(u(x-ct,y),v(x-ct,y))\in C^2(D_L)$
is a traveling-wave solution to the $\beta$-plane equation \eqref{Euler equation}-\eqref{boundary condition for euler}, then $(u(x-ct,y),v(x-ct,y))$ is a shear flow.
\end{Corollary}
\begin{Remark}\label{beta=0-cor-c-in-R-rem}
In the $f$-plane setting, by employing a different approach, Corollary \ref{beta=0-cor-c-in-R}  generalizes Corollary 1.6 $({\rm ii})$ in \cite{gxx2024} (for the case of a bounded periodic channel) and Theorem 1.2 in \cite{dn2024} in two aspects: one is
by generalizing steady flows to traveling waves with non-vanishing wave speeds,
and the other aspect is that the traveling waves  are not necessarily near a shear flow. See Proposition \ref{beta-plane-general-beta0} for the rigidity of traveling waves with a fixed wave speed near a shear flow.
\end{Remark}

As an application of Corollary  \ref{beta=0-cor-c-in-R}, we establish the rigidity of traveling waves with a fixed wave speed near a class of shear flows in the $f$-plane setting.

\begin{Proposition} \label{beta-plane-general-beta0}
Let $\beta=0$.
Assume that $u_0\in C^2([-d,d])$ and $c\in\mathbb{R}$ satisfy  $u_0''\neq0$ whenever $u_0-c=0$. Then there exists $\varepsilon_0>0$ such that
if $(u(x-ct,y),v(x-ct,y))\in C^2(D_L)$ is a traveling-wave solution to the $\beta$-plane equation \eqref{Euler equation}-\eqref{boundary condition for euler}
and satisfies
\begin{align*}
\|u-u_{0}\|_{C^2(D_L)}<\varepsilon_0,
\end{align*}
then $(u(x-ct,y),v(x-ct,y))$ has to be a shear flow.
\end{Proposition}
\begin{proof}
If $u=c$ has no zeros, then $(u(x-ct,y),v(x-ct,y))$ is a shear flow by Theorem 1.1 in \cite{Kalisch12}. If $u=c$ has zeros, we choose an arbitrary zero $(x_0,y_0)$.
 We will show that
 \begin{align}\label{laplace ux0y0neq0}
 (\Delta u)(x_0,y_0)\neq0.
 \end{align}
 Since $u_0\in C^2 ([-d,d])$ and $u_0''\neq0$ whenever $u_0-c=0$,
 the set $\{y:u_0-c=0\}$ is finite and is denoted by $\{y_i:1\leq i\leq n\}$.
 Choose $\delta_0>0$ such that $(y_i-\delta_0,y_i+\delta_0)\cap [-d,d]$, $1\leq i\leq n$, are disjoint, and \begin{align}\label{u0second-derivative-positive}
 |u_0''|>{\min_{1\leq i\leq n}|u_0''(y_i)|\over 2}>0\quad\text{on}\quad U_0=\cup_{1\leq i\leq n}((y_i-\delta_0,y_i+\delta_0)\cap[-d,d]).\end{align}
  Moreover,
 $\min_{y\in[-d,d]\setminus U_0}|u_0(y)-c|>0$. Let
 \begin{align*}
 \varepsilon_0=\min\left\{{\min_{1\leq i\leq n}|u_0''(y_i)|\over 2},{\min_{y\in[-d,d]\setminus U_0}|u_0(y)-c|\over 2}\right\}>0.
 \end{align*}
 Since $\|u-u_0\|_{C^2(D_L)}<\varepsilon_0$, we have
 \begin{align*}
 &|u(x,y)-c|\geq |u_0(y)-c|-|u_0(y)-u(x,y)|\\
 >&\min_{y\in[-d,d]\setminus U_0}|u_0(y)-c|-{\min_{y\in[-d,d]\setminus U_0}|u_0(y)-c|\over 2}={\min_{y\in[-d,d]\setminus U_0}|u_0(y)-c|\over 2}>0
 \end{align*}
 for $(x,y)\in\mathbb{T}_L\times([-d,d]\setminus U_0)$. Thus, $u-c\neq0$ on $\mathbb{T}_L\times([-d,d]\setminus U_0)$, which implies $(x_0,y_0)\in \mathbb{T}_L\times U_0$.
 Therefore, by \eqref{u0second-derivative-positive} and $\|u-u_0\|_{C^2(D_L)}<\varepsilon_0$ we have
 \begin{align*}
 |(\Delta u)(x_0,y_0)|\geq |u_0''(y_0)|-|u_0''(y_0)-(\Delta u)(x_0,y_0)|>{\min_{1\leq i\leq n}|u_0''(y_i)|\over 2}-\varepsilon_0\geq0.
 \end{align*}
 This proves \eqref{laplace ux0y0neq0}. Applying Corollary \ref{beta=0-cor-c-in-R}, $(u(x-ct,y),v(x-ct,y))$ is a shear flow.
\end{proof}

The following remark about the $f$-plane setting explains  how Theorem \ref{classification-of-wave-speed-for-a-genuinely-travelling-wave-beta-plane} (ii)
complements the rigidity result established  in \cite{Hamel2017}.
\begin{Remark}
In the $f$-plane setting,
let $(u(x-ct,y),v(x-ct,y))\in C^2(D_L)$
 be a  traveling-wave solution to  \eqref{Euler equation}-\eqref{boundary condition for euler}. Then $(u-c,v)$ is a steady solution of \eqref{Euler equation}-\eqref{boundary condition for euler}.
Theorem \ref{classification-of-wave-speed-for-a-genuinely-travelling-wave-beta-plane} (ii) can be viewed as a rigidity result in the presence of stagnation points of $(u-c,v)$, see \eqref{v-u-c0}. This stands in contrast to the rigidity result by Hamel-Nadirashvili \cite{Hamel2017}, which assumes that
$(u-c,v)$ has no stagnation points throughout the domain. Our result therefore complements Hamel-Nadirashvili's by addressing a regime that was previously excluded.
The combination of  Theorem \ref{classification-of-wave-speed-for-a-genuinely-travelling-wave-beta-plane} (ii) and Hamel-Nadirashvili's result
 can  be stated as
 ``if $c$ is not a generalized inflection value of $u$ or if $(u-c,v)$ has no stagnation points, then $(u(x-ct,y),v(x-ct,y))$ is a shear flow in the $f$-plane setting."
\end{Remark}

\subsubsection{Application to inviscid dynamical structures near the Kolmogorov flow in the $f$-plane setting}
In the finite channel $(x,y)\in D_{L_0}=\mathbb{T}_{L_0}\times [-d,d]$ with $L_0={4\pi\over \sqrt{3}}$ and $d=\pi$, we consider the Kolmogorov flow
\begin{align}\label{Kolmogorov flow}
u_{\rm k}(y)=\sin(y),\quad y\in[-\pi,\pi]\,,
\end{align}
and we study the inviscid dynamical structures near it in the $f$-plane setting.

\begin{Proposition}\label{inviscid-dynamical-structures-near-Kolmogorov-flow}
Let $\beta=0$.

{\rm (i)} For any $\delta>0$, if a traveling-wave solution $(u(x-ct,y),v(x-ct,y))\in C^2(D_{L_0})$ to the $\beta$-plane equation
\eqref{Euler equation}-\eqref{boundary condition for euler}  satisfies that
$c\in(-\infty,-\delta]\cup[\delta,\infty)$ and if
\begin{align}\label{7101}
\|u-u_{\rm k}\|_{C^2(D_{L_0})}<\delta,
\end{align}
then $(u(x-ct,y),v(x-ct,y))$ is a shear flow.

{\rm (ii)} For any $\varepsilon\in\mathbb{R}$, $(u_\varepsilon,v_\varepsilon)$ in \eqref{steady-flow-Kolmogorov} is a non-sheared  steady solution to the $\beta$-plane equation
\eqref{Euler equation}-\eqref{boundary condition for euler}. In particular, for $\varepsilon$  small enough, the non-sheared steady flow $(u_\varepsilon,v_\varepsilon)$ is sufficiently close in
analytic regularity to the Kolmogorov flow \eqref{Kolmogorov flow}.
\end{Proposition}

\begin{proof} (i)
If $c\notin Ran(u)$, then $(u(x-ct,y),v(x-ct,y))$ is a shear flow by Theorem 1.1 in \cite{Kalisch12}.
Thus we only need to consider the case $c\in Ran(u)$.  For any $(x_c,y_c)\in D_{L_0}$ satisfying $u(x_c,y_c)=c$, by \eqref{7101} we get
\begin{align*}
|(\Delta u)(x_c,y_c)|&=|c-(c-u_{\rm k}(y_c))-\left(u_{\rm k}(y_c)+(\Delta u)(x_c,y_c)\right)|\\
&\geq|c|-|c-u_{\rm k}(y_c)|-|u_{\rm k}(y_c)+(\Delta u)(x_c,y_c)|\\
&\geq|c|-|u(x_c,y_c)-u_{\rm k}(y_c)|-|-u_{\rm k}''(y_c)+(\Delta u)(x_c,y_c)|\\
&\geq|c|-\|u-u_{\rm k}\|_{C^2(D_{L_0})}\\
&>\delta-\delta=0,
\end{align*}
where we have used $u_{\rm k}(y_c)=-u_{\rm k}''(y_c)$ and $c\in(-\infty,-\delta]\cup[\delta,\infty)$.
Then $(u(x-ct,y),v(x-ct,y))$ is a shear flow by Corollary \ref{beta=0-cor-c-in-R}.

(ii) The stream function and vorticity of $(u_\varepsilon,v_\varepsilon)$ in \eqref{steady-flow-Kolmogorov} are
\begin{align*}
\psi_\varepsilon(x,y)=\cos(y)+\varepsilon\cos\left({y\over2}\right)\sin\left({\sqrt{3}\over2}x\right)\quad\text{and}\quad
\gamma_\varepsilon(x,y)=-\cos(y)-\varepsilon\cos\left({y\over2}\right)\sin\left({\sqrt{3}\over2}x\right).
\end{align*}
Since $\{\psi_\varepsilon,\gamma_\varepsilon\}=0$ on $D_{L_0}$ and
\begin{align*}
 v_\varepsilon=0\quad \text{on}\quad  y=\pm\pi,
\end{align*}
$(u_\varepsilon,v_\varepsilon)$  is a non-sheared  steady solution to the $\beta$-plane equation
\eqref{Euler equation}-\eqref{boundary condition for euler} with $\beta=0$. Moreover, by its expression in \eqref{steady-flow-Kolmogorov}, $(u_\varepsilon,v_\varepsilon)$ is sufficiently close in
analytic regularity to the Kolmogorov flow \eqref{Kolmogorov flow}
as long as  $\varepsilon$ is taken   small enough.
\end{proof}

\noindent
{\bf Acknowledgement} This research was supported by the Austrian Science Fund (FWF) [grant number Z 387-N] and by the National Natural Science Foundation of China [grant
number 12494544 and grant number 12471229].

\end{CJK*}

\end{document}